\documentclass[10pt, oneside, reqno]{article}                   				

	\newcommand{\TheTitle}{A Smoothed Augmented Lagrangian Method for Mathematical Programs with Vertical Complementarity Constraints}
	\newcommand{\TheFunding}{%
		This work was supported by:
		Fond for Scientific Research Vlaanderen (FWO) projects G033822N, G081222N, G0A0920N;
		KU Leuven internal funding C14/24/103;
		Japan Society for the Promotion of Science KAKENHI grant n. JP24K20737.%
	}
	\newcommand{\TheKeywords}{%
		Mathematical programs with vertical complementarity constraints;
		Constraint qualifications;
		Smoothing method;
		Augmented Lagrangian method;
		Convergence analysis;
		Worst-case complexity%
	}

	\PassOptionsToPackage{nameinlink,capitalize}{cleveref}
	\PassOptionsToPackage{breakspaces}{clevethm}
	\usepackage{graphicx}
	\usepackage{amsmath}
	\usepackage{mathtools}

	\usepackage[%
		arxiv,
		alglinenumber,
		noload={tikz,pgfplots,myNotes},
		eqreset=section,
		thmreset=section,
	]{myPreamble}

	\graphicspath{{./Pics/Tikz/}}

	\AtBeginDocument{%
		\setcounter{tocdepth}{2}
		\setcounter{secnumdepth}{2}
		\hypersetup{bookmarksdepth = 4}
	}
	\Crefname{figure}{Figure}{Figures}
	\numberwithin{algorithm}{section}

	\newcommand{\y}{\bm{y}}
	\newcommand{\z}{\bm{z}}
	\newcommand{\s}{\bm{s}}
	\newcommand{\w}{\bm{w}}

	\newcommand{\be}{\bm{\beta}}

	\RenewDocumentCommand{\set}{s m o}{
		\IfBooleanTF{#1}{
			\left\{#2
				\IfValueT{#3}{
					\,\middle|\,#3
				}
			\right\}
		}{
			\{#2
				\IfValueT{#3}{
					\,|\,#3
				}
			\}
		}
	}

  \newcommandx{\vit}[3][{2=j},{3={}}]{
			#1^{#2\ifthenelse{\equal{#3}{}}{}{,#3}}
			}

  \newcommandx{\sit}[3][{2=j},{3={}}]{
    #1_{#2\ifthenelse{\equal{#3}{}}{}{,#3}}
    }

  \newcommandx{\tb}[1][{1={}}]{
    \mathcal{T}_{\beta}\ifthenelse{\equal{#1}{}}{}{^{#1}}%
  }

	\let\ncone\relax
	\DeclareMathOperator{\ncone}{N}               
	\def\Cncone_#1{\operatorname{N}^{\rm C}_{#1}} 
	\DeclareMathOperator{\Rncone}{\widehat{N}}    
	\DeclareMathOperator{\tcone}{T}               
	\DeclareMathOperator{\spa}{span}
	\DeclareMathOperator{\env}{env}
	\DeclareMathOperator*{\Lim}{Lim}
	\DeclareMathOperator*{\Limsup}{Lim\,sup}
	\DeclareMathOperator*{\Liminf}{Lim\,inf}

	\newcommand{\B}{\varmathbb B}

	\newcommand{\CCPM}{\hyperref[CCP_common]{MPVCC-AMR}}
	\newcommand{\CCPC}{\hyperref[CQ_Clarke]{MPVCC-ACR}}
	\newcommand{\MPVCC}{\hyperref[MPVCC]{MPVCC}}
	\newcommand{\RCPLD}{\hyperref[definition_MPVCC_RCPLD]{MPVCC-RCPLD}}
	\newcommand{\RCRCQ}{\hyperref[definition_MPVCC_RCRCQ]{MPVCC-RCRCQ}}
	\newcommand{\SOSC}{\hyperref[definition_MPVCC_SOSC]{MPVCC-SOSC}}
	\newcommand{\LICQ}{\hyperref[definition_MPVCC_LICQ]{MPVCC-LICQ}}
	\newcommand{\RCPLDi}{\hyperref[RCPLD_i]{MPVCC-RCPLD(i)}}
	\newcommand{\RCPLDii}{\hyperref[RCPLD_ii]{MPVCC-RCPLD(ii)}}
	\newcommand{\RCPLDiii}{\hyperref[RCPLD_iii]{MPVCC-RCPLD(iii)}}
	\newcommand{\SSC}{\hyperref[MPVCC_SSC]{MPVCC-SSC}}
\makeatother

	\title{\TheTitle\thanks{\TheFunding}}

	\author{%
		Jia Wang\thanks{\TheAddressKU.}%
	\and
		Andreas Themelis\thanks{%
			\TheAddressKUJ.
			\newline
			{\it E-mails:}
			\sf
			\{%
				\emailLink[jia.wang@kuleuven.be]{jia.wang},%
				\emailLink[panos.patrinos@kuleuven.be]{panos.patrinos}%
			\}%
			\emailLink[jia.wang@kuleuven.be,panos.patrinos@kuleuven.be]{@kuleuven.be},
			\emailLink{andreas.themelis@ees.kyushu-u.ac.jp}
		}%
	\and
		Panagiotis Patrinos\footnotemark[2]%
	}

	\date{}

\begin{document}

	\maketitle

	\begin{abstract}
		We propose a smoothed augmented Lagrangian method (ALM) for solving mathematical programs with vertical complementarity constraints (VCCs).
The augmented Lagrangian function involving the VCCs is smoothed by the Lasry--Lions double envelope, which is an extension of the Moreau envelope that enjoys enhanced smoothness properties.
This construction yields a smooth unconstrained subproblem within the ALM framework that can be efficiently solved by many fast optimization algorithms.
In the limiting setting, we establish global convergence to Mordukhovich- and Clarke-stationary points.
For the convergence rate, we show that, under appropriate assumptions, the smoothed ALM achieves Q-linear convergence.
We further provide a worst-case iteration complexity analysis for computing an approximate stationary point of the original problem or of an associated feasibility problem.
Preliminary numerical experiments demonstrate the effectiveness of the proposed method.

	\end{abstract}

	\section{Introduction}\label{sec_intro}
		We consider the mathematical program with vertical complementarity constraints (MPVCC) of the form:
\begin{equation*}\label{MPVCC}
	\begin{split}
		\minimize_{x\in\R^n}
		\ \
		& f(x)
		\\
		\stt\ \
		& A(x)\in C
		\\
		& \min \{F_{1}^{j}(x),\ldots,F_{p_{j}}^{j}(x)\} = 0,\ j\in\{1,\ldots,r\},
	\end{split}\tag{MPVCC}
\end{equation*}
where \(f:\R^{n}\to\R\), \(A:\R^n\to\R^m\), and \(F_{i}^{j}:\R^n\to\R\) for \( i\in\{1,\ldots,p_j\}\) with \(p_j\geq2\).
We assume that all the functions involved are continuously differentiable.
The set \(C\subseteq\R^{m}\) is nonempty, closed and convex.
Let \(F^{j}\coloneqq(F_{1}^{j},\ldots,F_{p_j}^{j})\).
The vertical complementarity constraints (VCCs) can be equivalently expressed as \(F^{j}(x)\in D^{j}\), where
\begin{equation}\label{D_set}
D^{j}\coloneqq
\left\{
	z \in \R_+^{p_j}
	\,\middle|\,
	{\textstyle\prod_{i=1}^{p_j} z_i} = 0
\right\}.
\end{equation}
The set \(D^j\) is closed but nonconvex, and exhibits an inherently combinatorial structure.

\subsection{Examples}
	\MPVCC{} is a powerful modeling framework for a variety of challenging problems.
	First, consider the vertical linear complementarity problem (VLCP).
	Let \(x=(x^0,x^1,\ldots,x^\ell)\in\R^{r(\ell+1)}\) with
	\(x^0,x^1,\ldots,x^\ell\in\R^{r}\).
	Given matrices
	\(B_1,\ldots,B_\ell\in\R^{r\times r}\) and vectors
	\(b^1,\ldots,b^\ell\in\R^{r}\),
	the VLCP consists of finding \(x\) satisfying
	\begin{equation}\label{VLCP}
	\min
	\left\{
	x^{0}_{i}, x^{1}_{i},\ldots, x^{\ell}_{i}
	\right\}=0,\
	i\in\{1,\ldots,r\}
	\quad\text{and}\quad
	x^j=B_j x^0+ b^j,\ j\in\{1,\ldots,\ell\}.
	\end{equation}
	The VLCP was introduced in~\cite{cottle1970} and has found applications in areas such as nonlinear networks~\cite{fujisawa1972}, control theory~\cite{sun1989}, and economics~\cite{nagae2008}.
	However, solutions to the VLCP may not be unique.
	Consequently, the problem of finding a solution to~\eqref{VLCP} that also satisfies additional optimization criteria naturally fits within the framework of \MPVCC{}.
	Second, by fixing \(p_j=2\),
	\MPVCC{} reduces to the mathematical program with complementarity constraints (MPCC):
	\begin{equation}\label{MPCC}
	\begin{split}
	\minimize_{x\in\R^n}
	\ \
	& f(x)
	\\
	\stt\ \
	& A(x)\in C
	\\
	& \bm{0}\leq G(x)\perp H(x)\geq \bm{0},
	\end{split}
	\end{equation}
	where \(G,H:\R^n\to\R^r\).
	The notation \(\bm{0}\leq a\perp b\geq \bm{0}\) means that \(a,b\in\R_{+}^{r}\) and \(\langle a,b \rangle =0\),
	which corresponds to the complementarity constraints (CCs).
	In this case, \(F^j\) in \MPVCC{} is decomposed as \(F^j=(G_j,H_j):\R^n\to\R^2\).
	MPCC is an important class of optimization problems that arises in applications of Stackelberg games in economics.
	It also emerges from reformulations of bilevel programming, contact mechanics, and the optimal control of hybrid systems; see~\cite{ferris1997} for details.

	\MPVCC{} can also be used to reformulate a broad class of optimization problems with alternative structural features.
	For instance, we consider the constrained min-max-min problem in~\cite{xiong2010}:
	\begin{equation}\label{min_max_min}
	\begin{split}
	\minimize_{x\in\R^n}
	\ \
	&
	\left\{
	\phi(x)\coloneqq\max_{j\in\{1,\ldots,r\}}\ \ \min_{i\in\{1,\ldots,p_j\}}\ \ \phi_{i}^{j}(x)
	\right\}
	\\
	\stt\ \
	& A(x)\in C
	\\
	&
	\left\{
	g(x)\coloneqq\max_{j'\in\{1,\ldots,r'\}}\ \min_{i'\in\{1,\ldots,p_{j'}'\}} g_{i'}^{j'}(x)
	\right\}
	\leq 0,
	\end{split}
	\end{equation}
	where \(\phi_{i}^{j},g_{i'}^{j'}:\R^{n}\to\R\).
	Problem~\eqref{min_max_min} appears in a range of engineering applications, including the optimal design of electronic circuits under tolerancing and tuning provisions~\cite{bandler1976,polak2003,polak1979}, and the design of heat exchangers and chemical reactors~\cite{halemane1983,ostrovsky1994}; see~\cite{kirjner-neto1998} and references therein for further applications.
	To develop implementable algorithms for~\eqref{min_max_min}, we introduce the auxiliary variables \(t\in\R\),
	\(s\in\R^{r}_+\) and
	\(w\in\R^{r'}_+\).
	Following~\cite{liang2012}, problem~\eqref{min_max_min} admits the reformulation:
	\[
	\begin{aligned}
		\minimize_{x,t,s,w}
		\ \
		& t
		\\
		\stt\ \
		& A(x)\in C,\ s\geq\bm{0},\ w\geq\bm{0}
		\\
		&
		\min_{i\in\{1,\ldots,p_j\}}
		\begin{Bmatrix}
		\phi_{i}^{j}(x)-t+s_j
		\end{Bmatrix}
		=0,\ j\in\{1,\ldots,r\}
		\\
		&
		\min_{i'\in\{1,\ldots,p_{j'}'\}}
		\begin{Bmatrix}
		g_{i'}^{j'}(x)+w_{j'}
		\end{Bmatrix}
		=0,\ j'\in\{1,\ldots,r'\},
	\end{aligned}
	\]
	which can be addressed within the framework of \MPVCC{}.
	Moreover, we consider the mathematical program with set exclusion constraints:
	\begin{equation}\label{MPSEC}
	\begin{split}
	\minimize_{x\in\R^n}
	\ \
	& f(x)
	\\
	\stt\ \
	& A(x)\in C
	\\
	& x\notin O^j,\ j\in\{1,\ldots,r\},
	\end{split}
	\end{equation}
	where
	\(
	O^j
	\coloneqq
	\left\{
		x\in\R^n
		\,\middle|\,
		F^j(x)>\bm{0}
	\right\}
	\) and \(F^j:\R^n\to\R^{p_j}\).
	This type of constraint excludes certain open sets from the feasible region and arises in applications such as vehicle path planning, where set-exclusion constraints represent obstacles in space~\cite{hermans2021}, as well as in the automatic floorplanning of circuit layout design~\cite{moh1996}.
	Note that the constraint \(x\notin O^j\) is equivalent to at least one component of \(F^j(x)\) being negative.
	By introducing slack variables
	\(s^j\in\R^{p_j}\) for \(j\in\{1,\ldots,r\}\),
	problem~\eqref{MPSEC} reads as
	\[
	\begin{aligned}
	\minimize_{x,s^1,\ldots,s^r}
	\ \
	& f(x)
	\\
	\stt\ \
	& A(x)\in C,\ F^j(x)\leq s^j
	\\
	& \min_{i\in\{1,\ldots,p_j\}}
		\begin{Bmatrix}
		s_{i}^{j}
		\end{Bmatrix}
	=0,\ j\in\{1,\ldots,r\},
	\end{aligned}
	\]
	which again yields a reformulation of \MPVCC{}.
	These examples illustrate the strong modeling capability of \MPVCC{} and further motivate the development of efficient algorithms for solving it.

\subsection{Related Work}
	Although \MPVCC{} arises in a wide range of applications, the presence of VCCs makes the problem challenging from both theoretical and numerical perspectives.
	When the VCCs are written as standard nonlinear constraints, e.g.,
	\(F^j(x)\geq\bm{0}\) and \(\prod_{i=1}^{p_j} F^{j}_{i}(x)\leq0\),
	standard constraint qualifications (CQs), such as the linear independence (LICQ) and the Mangasarian--Fromovitz constraint qualification (MFCQ), are violated at every feasible point of a VCC{}.
	Consequently, the classical Karush--Kuhn--Tucker (KKT) optimality conditions are generally inapplicable for characterizing stationarity in \MPVCC{}.
	Moreover, the associated sets of Lagrange multipliers may be unbounded, which can lead to numerical instability and poor convergence behavior in standard primal-dual algorithms.

	Nevertheless, as a special case of \MPVCC{}, significant theoretical and algorithmic progress has been achieved for MPCCs.
	On the theoretical side, various stationarity concepts have been introduced, including strong (S)~\cite{scheel2000}, Mordukhovich (M)~\cite{outrata1999}, Clarke (C)~\cite{scheel2000}, alternative (A)~\cite{flegel2005c}, and weak (W)~\cite{scheel2000} stationarity.
	At the same time, a variety of MPCC-tailored CQs have been proposed,
	such as MPCC-LICQ~\cite[Definition 3(a)]{kanzow2015},
	MPCC-MFCQ~\cite[Definition 3(b)]{kanzow2015},
	MPCC-constant rank CQ (CRCQ)~\cite[Definition 3(c)]{kanzow2015},
	MPCC-constant positive linear dependence CQ (CPLD)~\cite[Definition 3(d)]{kanzow2015},
	MPCC-Abadie CQ~\cite[Definition 3.1]{flegel2005c},
	MPCC-Guignard CQ~\cite[Definition 3.1]{flegel2005}, and so on.
	A comprehensive survey of MPCC-tailored CQs can be found in~\cite{ye2005}, and a detailed study of stationarity concepts and optimality conditions is provided in the PhD thesis~\cite{flegel2005b}.
	On the algorithmic side, a broad range of methods for solving MPCCs has been extensively studied, including regularization methods~\cite{hoheisel2013,scholtes2001a}, penalty methods~\cite{hu2004}, interior-point methods~\cite{leyffer2006,pozharskiy2026}, active-set methods~\cite{izmailov2008}, and augmented Lagrangian methods (ALMs)~\cite{izmailov2012,jia2023}.
	Moreover, standard nonlinear programming (NLP) solvers have also been successfully applied to MPCCs~\cite{fletcher2006,izmailov2012}.

	In contrast, both the theory and algorithms for \MPVCC{} remain relatively underdeveloped.
	In~\cite{liang2012}, stationarity concepts for \MPVCC{} were investigated, and a constrained Levenberg--Marquardt method was proposed for solving the nonlinear equations arising from different stationarity systems.
	In~\cite{xu2023}, \MPVCC{} was addressed via a regularization method, where each subproblem was solved by a standard ALM{}.
	In addition,~\cite{jia2023} proposed an ALM framework that retains the VCCs explicitly in each subproblem and solves the resulting subproblems using projected gradient methods.
	However, most existing methods require the solution of constrained nonlinear subproblems, which reduces their practical attractiveness.
	Furthermore, most existing methods establish convergence only in an asymptotic sense and do not provide worst-case iteration complexity guarantees, despite the importance of such guarantees in many engineering applications.
	Motivated by the above observations, the goal of this paper is to develop an algorithm for \MPVCC{} that is easy to implement in practice
	and supported not only by rigorous asymptotic convergence theory but also by complexity guarantees.

	The iteration complexity of ALMs for nonconvex constrained optimization has recently attracted considerable attention; see, e.g.,~\cite{birgin2020,grapiglia2021,sahin2019}.
	In particular,~\cite{birgin2020} analyzes a safeguarded ALM based on a quadratic penalty of shifted constraints.
	In this paper, we leverage this framework to develop our algorithm for \MPVCC{} and analyze its worst-case iteration complexity.
	However, it is important to note that the quadratic penalty associated with a VCC corresponds to the Moreau envelope of its indicator function, which is nonsmooth.
	To address this difficulty, following the approach in~\cite{wang2025}, we employ the Lasry--Lions double envelope~\cite{attouch1993} to smooth the augmented Lagrangian function.
	This yields a smooth unconstrained subproblem that can be efficiently solved by many mature optimization solvers, such as \texttt{L-BFGS}~\cite{liu1989}.
	The use of the Lasry--Lions double envelope in constrained optimization has only recently been investigated.
	In~\cite{wang2025}, it was used to smooth the subproblems arising in a homotopy method for solving MPCCs.
	In~\cite{simoes2021}, it was applied to nonsmooth nonconvex optimization problems arising in signal processing.

\subsection{Contribution}
	First, we develop a smoothed ALM for \MPVCC{}, leveraging the Lasry--Lions double envelope to construct a smooth approximation of the augmented Lagrangian function associated with VCCs, which are particularly challenging due to their inherent nonsmoothness, nonconvexity, combinatorial structure, and the lack of standard CQs.
	Second, we introduce MPVCC-tailored CQs that are weaker than MPVCC-tailored relaxed constant positive linear dependence (RCPLD) and relaxed constant rank constraint qualification (RCRCQ). Under these conditions, we establish convergence to M- and C-stationary points.
	Third, under appropriate assumptions, we prove that the smoothed ALM{} exhibits a Q-linear convergence rate.
	Fourth, we show that the smoothed ALM achieves a prescribed accuracy tolerance in finitely many iterations for obtaining an approximate stationary point of
	either the original problem or an associated feasibility problem,
	and we derive worst-case iteration complexity bounds.
	Finally, we validate the effectiveness of the proposed method through numerical experiments.
	The results indicate that the smoothed ALM is competitive with the Lasry--Lions double-envelope-based homotopy method, the classical ALM, and the NLP solver \texttt{Algencan} in terms of computational efficiency.
	An open-source package is also provided.\footnote{\url{https://github.com/JiaaaWang/Lldo.jl}}

\subsection{Notation}
	The set of natural numbers is \(\N\coloneqq\set{0,1,\ldots}\).
	The \emph{extended real line} is denoted by \(\overline{\R}\coloneqq\R\cup\set{\infty}\).
	The set of positive and strictly positive reals are \(\R_+\coloneqq[0,\infty)\) and \(\R_{++}\coloneqq(0,\infty)\);
	analogously, \(\R_{-}\) and \(\R_{--}\) denote the sets of negative and strictly negative reals.
	We denote \({[\cdot]}_+\coloneqq\max\set{0,\cdot}\)
	and \({[\cdot]}_{-}\coloneqq\min\set{0,\cdot}\).
	The collection of subsequences of \(\N\) is \(\mathcal{N}_{\infty}^{\sharp}\).
	For a matrix \(A\in\R^{m\times n}\), we denote its null space by
	\(\nullspace A\coloneqq \set{x\in\R^{n}}[Ax=0]\).
	The smallest subspace and convex cone containing the set \(E\) are \(\spa E\) and \(\cone E\), respectively.
	The \emph{convex hull} of a set \(E\) is denoted by \(\conv E\).
	Moreover, the \emph{Euclidean distance} between a point \(x\) and the set \(E\) is \(\dist(x,E)\coloneqq\inf_{x'\in E}\|x'-x\|\)
	and the (possibly set-valued) \emph{Euclidean projection} onto the set \(E\) is \(\Pi_{E}(x)\coloneqq\arg\min_{x'\in E}\|x'-x\|\).
	The \emph{indicator function} of \(E\) is given by \(\delta_{E}(x)=0\) if \(x\in E\), and \(\delta_{E}(x)=+\infty\) otherwise.
	The closed ball with center at \(x\in\R^n\) and radius \(\xi\in\R_{+}\) is \(\B(x,\xi)\coloneqq\set{x'\in\R^n}[\|x'-x\|\leq\xi]\).
	Furthermore, for any \(z\in\R^p\), we define the index set \(\mathcal{I}_{=0}(z)\coloneqq\set{i\in\{1,\ldots,p\}}[z_i=0]\), and analogously the index sets \(\mathcal{I}_{\leq0}(z)\) and \(\mathcal{I}_{>0}(z)\).
	For a smooth mapping \(S:\R^n\to\R^m\), the \emph{Jacobian} of \(S\) at \(x\in\R^n\) is \(JS(x)\in\R^{m\times n}\).
	Given a vector \(a\), we define \(\R a\coloneqq\set{ta}[t\in\R]\) and \(\R_{+}a\coloneqq\set{ta}[t\in\R_{+}]\).
	We denote by \(e_i\) the vector of appropriate dimension whose \(i\)-th component is one and all other components are zero.
	Finally, we denote by \(\bm{0}\) and \(\bm{1}\) the all-zero and all-one vectors, respectively, of appropriate dimensions.

	\section{Preliminaries and Supporting Results}
		\subsection{Fundamentals of Variational Analysis}
	The \emph{epigraph} of \(\varphi:\R^{n}\rightarrow\overline{\R}\) is the set \(\epi \varphi\coloneqq\set{(x,\alpha)\in\R^{n+1}}[\varphi(x)\leq\alpha]\),
	and \(\varphi\) is said to be \emph{lower semicontinuous (lsc)} if \(\epi \varphi\) is closed,
	and \emph{proper} if \(\epi \varphi\) is nonempty and \(\varphi(x)>-\infty\) for all \(x\in\R^n\).
	For a sequence of sets \(\seq{E^\nu}[\nu\in\N]\) with \(E^\nu\subseteq\R^n\), we define the \emph{outer limit} as
	\[
	\Limsup_{\nu\to\infty} E^{\nu}\coloneqq\set{x}
	[\exists K\in \mathcal{N}_{\infty}^{\sharp},\
	x^\nu\to_{K}x:
	x^{\nu}\in E^{\nu}\ \ \forall\nu\in K]
	\]
	and the \emph{inner limit} as
	\[
	\Liminf_{\nu\to\infty} E^{\nu}\coloneqq\set{x}
	[\exists k\in \N,\
	x^\nu\to x:
	x^{\nu}\in E^{\nu}\ \ \forall\nu\geq k].
	\]
	We say that \(\seq{E^\nu}[\nu\in\N]\) \emph{set-converges} if its outer limit and inner limit are equal, i.e.,
	\[
	\Lim_{\nu\to\infty} E^{\nu}\coloneqq\Limsup_{\nu\to\infty} E^{\nu}=\Liminf_{\nu\to\infty} E^{\nu}.
	\]
	We say that a sequence \(\seq{\varphi^{\nu}}[\nu\in\N]\) of functions \(\varphi^{\nu}:\R^{n}\rightarrow\overline{\R}\) \emph{epi-converges} to \(\varphi:\R^{n}\rightarrow\overline{\R}\) if
	\[
	\Lim_{\nu\to\infty} \epi \varphi^{\nu}=\epi \varphi,
	\]
	which is denoted by \(\varphi^{\nu}\overset{\mathrm{e}}{\rightarrow}\varphi\).
	We next review some cone definitions for a nonempty and closed set \(E\subseteq \R^n\).
	The \emph{tangent cone} at a point \(x\in E\) is
	\[
		\tcone_{E}(x)\coloneqq\set{d\in\R^{n}}[\exists E\ni x^{\nu}\to x\ \text{and}\ t_{\nu}\searrow 0\ \text{such that}\ \tfrac{x^{\nu}-x}{t_{\nu}}\to d].
	\]
	The \emph{regular normal cone} at a point \(x\in E\) is
	\[
		\Rncone_{E}(x)\coloneqq\set{
	d\in\R^{n}}[\langle d,d'\rangle \leq0\ \ \forall d'\in \tcone_{E}(x)].
	\]
	The \emph{limiting normal cone} at a point \(x\in E\) is
	\[
		\ncone_{E}(x)\coloneqq\set{d\in\R^{n}}[\exists E \ni x^{\nu}\to x\ \text{and}\ d^{\nu}\in\Rncone_{E}(x^{\nu})\ \text{such that}\ d^{\nu}\to d].
	\]
	Note that the limiting normal cone is \emph{outer semicontinuous}, in the sense that \(\Limsup_{x^\nu\to x}\ncone_E(x^\nu)\subseteq \ncone_E(x)\).
	Consequently, it enjoys better calculus properties than the regular normal cone.
	For a convex set \(C\), we have \(\Rncone_{C}(x)=\ncone_{C}(x)\) for all \(x\in C\).

\subsection{Cones Related to the Vertical Complementarity Set}\label{sec_cones}
	In this subsection, we derive several important cones associated with the vertical complementarity set.
	For notational simplicity, we temporarily suppress the dependence on \(j\) in~\eqref{D_set}.
	That is, we write \(D\coloneqq D^j\) and \(p\coloneqq p_j\) in this subsection.
	We note that \(D\) can be expressed as
	\begin{equation}\label{convex_union}
		D=\bigcup_{i=1}^{p}D_{i}\quad\text{with}\quad
		D_{i}\coloneqq\R_+^{i-1}\times\{0\}\times\R_+^{p-i},
	\end{equation}
	which represents the so-called \emph{disjunctive constraint}; see, for instance,~\cite[Section 5.1]{jia2023} and~\cite{mehlitz2020c}.
	The Euclidean projection onto \(D\) can be characterized as
	\begin{equation}\label{projection_D}
		\Pi_{D}(z)=
		\set{\Pi_{D_i}(z)}[i\in\mathcal{I}(z)]
		\quad\text{with}\quad
		\mathcal{I}(z)\coloneqq\argmin_{i\in\{1,\ldots,p\}} \dist{(z,D_i)}.
	\end{equation}
	Each set \(D_i\) is closed and convex.
	For \(\bar z\in D_i\), its tangent and regular normal cones are respectively given by
	\begin{equation}\label{tangent_cone_D_i}
		\tcone_{D_i}(\bar z)
	=
		\left\{
			d \in \R^{p}
		\,\middle|\,
			\begin{array}{@{} l @{~} l @{}}
				d_i=0
			\\
				d_{i'}\geq0
				&
				\text{if}\ \bar z_{i'}=0\ \text{and}\ i'\neq i
			\\
				d_{i'}\in\R
				&
				\text{if}\ \bar z_{i'}>0\ \text{and}\ i'\neq i
			\end{array}
		\right\}
		\text{ and }
		\Rncone_{D_i}(\bar z)
	=
		\left\{
			d \in \R^{p}
			\,\middle|\,
			\begin{array}{@{} l @{~} l @{}}
				d_i\in\R
			\\
				d_{i'}\leq0
			&
				\text{if}\ \bar z_{i'}=0\ \text{and}\ i'\neq i
			\\
				d_{i'}=0
			&
				\text{if}\ \bar z_{i'}>0\ \text{and}\ i'\neq i
			\end{array}
		\right\}.
	\end{equation}
	Note that, for \(\bar z\in D\), the index set \(\mathcal{I}_{=0}(\bar z)\) coincides with \(\set{i\in\{1,\ldots,p\}}[\bar z\in D_i]\).
	We present the following result characterizing the cones associated with~\eqref{convex_union}.

	\begin{proposition}\label{prop:N_D}%
	Relative to the vertical complementarity set~\eqref{convex_union}, the tangent cone, the regular normal cone, and the limiting normal cone to \(D\) at \(\bar z\in D\) are given by
	\begin{align}
		\tcone_{D}(\bar z)&=
		\left\{
		d \in \R^{p}
		\,\middle|\,
		\min_{i\in\mathcal{I}_{=0}(\bar z)} d_i=0
		\right\}\label{tangent_cone_D}
		\\[5pt]
		\Rncone_{D}(\bar z)&=
		\begin{cases}
				\R e_i
			& \text{ if }
				\mathcal{I}_{=0}(\bar z)=\set{i}
			\\[5pt]
				\cone\set{-e_i}[i\in \mathcal{I}_{=0}(\bar z)]
			& \text{ if }
				|\mathcal{I}_{=0}(\bar z)|>1
			\end{cases}\label{regular_normal_cone}
		\\[5pt]
		\ncone_{D}(\bar z)&=
		\begin{pmatrix}
		\bigcup_{i\in\mathcal{I}_{=0}(\bar z)}\R e_i
		\end{pmatrix}
		\cup
		\cone\set{-e_i}[i\in\mathcal{I}_{=0}(\bar z)].\label{limiting_normal_cone}
	\end{align}
	\end{proposition}
	\begin{proof}
	By~\cite[Lemma 2.2]{mehlitz2020c}, \cref{tangent_cone_D} follows from
	\[
		\tcone_{D}(\bar z)=\bigcup_{i\in\mathcal{I}_{=0}(\bar z)}\tcone_{D_i}(\bar z)
		\overset{\eqref{tangent_cone_D_i}}{=}\bigcup_{i\in\mathcal{I}_{=0}(\bar z)}
		\set{d \in \R^{p} }[d_i=0,\ d_{i'}\geq0,\ i'\in\mathcal{I}_{=0}(\bar z)].
	\]
	\cref{regular_normal_cone} follows by applying~\eqref{tangent_cone_D_i} to
	\(
		\Rncone_{D}(\bar z)
		=
		\bigcap_{i\in\mathcal{I}_{=0}(\bar z)}\Rncone_{D_i}(\bar z)
	\).
	As for~\eqref{limiting_normal_cone}, observe that \(\mathcal{I}_{=0}(z)\subseteq\mathcal{I}_{=0}(\bar z)\) holds for any \(z\) close enough to \(\bar z\);
		more precisely, for those \(\frac{\varepsilon}{2}\)-close to \(\bar z\), where \(\varepsilon=\min\set{\bar z_\ell}[\bar z_\ell>0]\) (with the convention that \(\min\emptyset=\infty\)).
		In particular, any sequence \(z^\nu\) converging to \(\bar z\) from within \(D\) will eventually satisfy \(\mathcal{I}_{=0}(z^\nu)\subseteq\mathcal{I}_{=0}(\bar z)\);
		from \eqref{regular_normal_cone} we thus conclude that
		\begin{equation}\label{eq:N_Dsubseteq}
			\textstyle
			\ncone_D(\bar z)
		\subseteq
			\begin{pmatrix}
				\bigcup_{i\in\mathcal{I}_{=0}(\bar z)}\R e_i
			\end{pmatrix}
			\cup
			\cone\set{-e_i}[i\in\mathcal{I}_{=0}(\bar z)].
		\end{equation}
		Conversely, clearly \(\Rncone_D(\bar z)\subseteq \ncone_D(\bar z)\); hence, it remains to show that \(\R e_i\subseteq \ncone_D(\bar z)\) holds for any \(i\in\mathcal{I}_{=0}(\bar z)\).
		To see this, fix \(i\in\mathcal{I}_{=0}(\bar z)\) and consider \(z^\nu\coloneqq \bar z+\frac{1}{\nu}(\bm{1}-e_i)\) for \(\nu\geq1\).
		Then, \(z^\nu\in D\) and \(\mathcal{I}_{=0}(z^\nu)=\set{i}\), granting via~\eqref{regular_normal_cone} that \(\Rncone_D(z^\nu)=\R e_i\).
		By taking \(\nu\to\infty\), \(z^\nu\to\bar z\) and consequently \(\R e_i\subseteq \ncone_D(\bar z)\) as needed.
	\end{proof}

	Following~\cite[Section 2.1]{wang2025}, to characterize Clarke stationarity in \MPVCC\@, we introduce the \emph{Clarke limiting cone} to \(D\) at \(\bar z\in D\) as
	\begin{align}
		\Cncone_D(\bar z)
		&\coloneqq
		\conv
		\begin{pmatrix}
		\bigcup_{i\in\mathcal{I}_{=0}(\bar z)}\R_{+} e_i
		\end{pmatrix}
		\cup
		\cone\set{-e_i}[i\in\mathcal{I}_{=0}(\bar z)]
		\nonumber\\
		&=
		\cone\set{e_i}[i\in\mathcal{I}_{=0}(\bar z)]
		\cup
		\cone\set{-e_i}[i\in\mathcal{I}_{=0}(\bar z)]
		\nonumber\\
		&=
		\set{d\in\R_{+}^p}[d_\ell=0,\ \ell\notin\mathcal{I}_{=0}(\bar z)]
		\cup
		\set{d\in\R_{-}^p}[d_\ell=0,\ \ell\notin\mathcal{I}_{=0}(\bar z)]
		\nonumber\\
		&=
		\set{d\in\R^p}[d_\ell=0,\ \ell\notin\mathcal{I}_{=0}(\bar z),\ d_{i}d_{i'}\geq0,\ i,i'\in\mathcal{I}_{=0}(\bar z)].
		\label{Clarke_cone}
	\end{align}

	\begin{figure}[h]
		\centering
		\begin{subfigure}{\textwidth}
			\centering
			\includegraphics[width=0.8\textwidth]{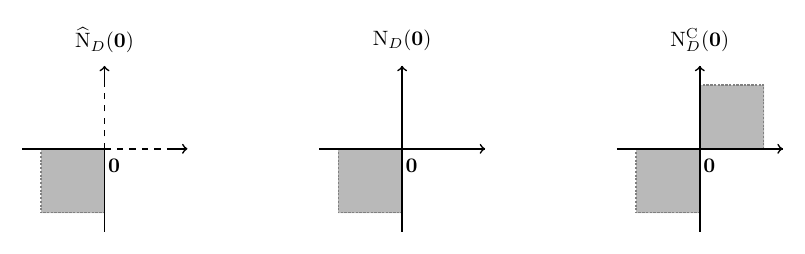}
			\caption{\(\Rncone_{D}\), \(\ncone_{D}\) and \(\Cncone_{D}\) at \(\bm{0}\) when \(p=2\).}
		\end{subfigure}

		\begin{subfigure}{\textwidth}
			\centering
			\includegraphics[width=0.8\textwidth]{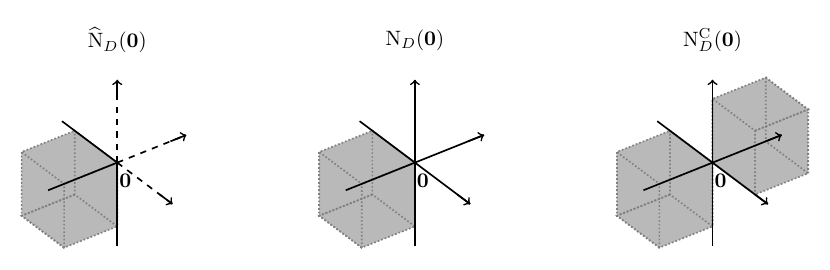}
			\caption{\(\Rncone_{D}\), \(\ncone_{D}\) and \(\Cncone_{D}\) at \(\bm{0}\) when \(p=3\).}
		\end{subfigure}

		\caption{Geometric illustrations of cones related to the vertical complementarity set.}\label{fig_normal_cones}
	\end{figure}

	The cones \(\Rncone_{D}\), \(\ncone_{D}\), and \(\ncone_{D}^{\mathrm{C}}\) are illustrated in \cref{fig_normal_cones}.
	Each of them corresponds to a specific problem-tailored stationarity concept for \MPVCC{}, as discussed in the next subsection.
	Note that \(\Rncone_{D}(\bar z)\subseteq\ncone_{D}(\bar z)\subseteq\ncone_{D}^{\mathrm{C}}(\bar z)\),
	and that \(\ncone_{D}^{\mathrm{C}}(\bar z)\) is closed but, in general, nonconvex.
	It should not be confused with the Clarke normal cone
	\(\conv\ncone_{D}(\bar z)=\set{d\in\R^p}[d_{\ell}=0,\ \ell\notin\mathcal{I}_{=0}(\bar z)]\),
	defined in~\cite[Eq. 6(19)]{rockafellar1998},
	which would lead to the notion of weak stationarity.

\subsection{Stationarity Concepts}
	To simplify the presentation, we define
	\begin{equation}\label{define_constraint}
		\bar p\coloneqq\sum_{j=1}^{r}p_j,
		\quad
		\Phi\coloneqq(A,F^{1},\ldots,F^{r}):\R^{n}\to\R^{m+\bar p}
		\quad\text{and}\quad
		\Gamma\coloneqq C \times D^{1} \times \cdots \times D^{r}\subset\R^{m+\bar p}.
	\end{equation}
	Moreover, we define the vectors
	\begin{equation}\label{define_vector}
		\y\coloneqq(y^0,y^1,\ldots,y^r)\in\R^{m+\bar p}
		\quad\text{and}\quad
		\z\coloneqq(z^0,z^1,\ldots,z^r)\in\Gamma,
	\end{equation}
	where \(y^0,z^0\in\R^m\) and \(y^j,z^j\in\R^{p_j}\).
	By~\cite[Proposition 6.41]{rockafellar1998},
	we have
	\begin{subequations}\label{define_cone}
		\begin{align}
		\Rncone_{\Gamma}(\z)&=\ncone_{C}(z^0)\times\Rncone_{D^1}(z^1)\times\cdots\times\Rncone_{D^r}(z^r)
		\\
		\ncone_{\Gamma}(\z)&=\ncone_{C}(z^0)\times \ncone_{D^1}(z^1)\times\cdots\times \ncone_{D^r}(z^r)
		\\
		\Cncone_{\Gamma}(\z)&\coloneqq\ncone_{C}(z^0)\times \Cncone_{D^1}(z^1)\times\cdots\times \Cncone_{D^r}(z^r).
		\end{align}
	\end{subequations}
	We introduce the following stationarity concepts for \MPVCC{}.

	\begin{definition}[Stationary point]\label{definition_stationary}
	Relative to \MPVCC{} with notation~\eqref{define_constraint}--\eqref{define_cone},
	a feasible point \(\bar x\in\R^{n}\) is said to be \emph{strongly (KKT) stationary (S-stationary)} if there exists \(\bar\y\in\R^{m+\bar p}\) such that
	\begin{equation*}\label{MPVCC_SSC}
	\begin{cases}
	\nabla f(\bar x)
	+
	J\Phi{(\bar x)}^{\top}\bar\y
	=
	\bm{0}
	\\[5pt]
	\bar\y\in \Rncone_{\Gamma}(\Phi(\bar x)).
	\end{cases}\tag{MPVCC-SSC}
	\end{equation*}
	If the above condition holds with \(\Rncone_{\Gamma}(\Phi(\bar x))\) being replaced by \(\ncone_{\Gamma}(\Phi(\bar x))\),
	then we say that \(\bar x\) is \emph{Mordukhovich (KKT) stationary (M-stationary)}.
	Furthermore, if \(\Rncone_{\Gamma}(\Phi(\bar x))\) is instead replaced by \(\Cncone_{\Gamma}(\Phi(\bar x))\),
	then \(\bar x\) is said to be \emph{Clarke (KKT) stationary (C-stationary)}.
	\end{definition}

	In the literature, \(\y\) serves as the multiplier of the \emph{MPVCC-Lagrangian} \(\mathcal{L}:\R^n\times\R^{m+\bar p}\to\R\), given by
	\begin{equation}\label{MPVCC_Lagrangian}
		\mathcal{L}(x,\y)
		\coloneqq
		f(x)
		+
		\langle \y,\Phi(x)\rangle.
	\end{equation}
	Each of these stationarity concepts treats differently the signs of the multipliers associated with the multiactive VCC,
	namely those satisfying \(|\mathcal{I}_{=0}(F^{j}(x))|>1\).
	If \(|\mathcal{I}_{=0}(F^{j}(x))|=1\) for all \(j\in\set{1,\ldots,r}\),
	then the three stationarity concepts coincide. In general, since
	\(\Rncone_{\Gamma}\subseteq\ncone_{\Gamma}\subseteq\ncone_{\Gamma}^{\mathrm{C}}\),
	S-stationarity implies M-stationarity, which in turn implies C-stationarity.
	We now introduce the asymptotic stationarity concepts for \MPVCC{}.

	\begin{definition}[Asymptotically stationary  point]\label{definition_asymptotic_stationary}%
	Relative to \MPVCC{} with notation~\eqref{define_constraint}--\eqref{define_cone},
	a feasible point \(\bar x\in\R^{n}\) is said to be \emph{asymptotically M-stationary (AM-stationary)} if there exist sequences
	\(\seq{x^\nu,\y^\nu,\z^\nu}[\nu\in\N]\) such that
	\[
		\begin{cases}
		\nabla f(x^\nu)
		+J\Phi{(x^\nu)}^\top {\y}^\nu
		\to
		\bm{0}
		\\[5pt]
		(x^\nu,{\z}^\nu)
		\to
		(\bar x,\Phi(\bar x))
		\\[5pt]
		{\z}^\nu\in\Gamma\ \text{and}\ {\y}^\nu\in \ncone_{\Gamma}({\z}^\nu).
		\end{cases}
	\]
	If the above condition holds with \(\ncone_{\Gamma}({\z}^\nu)\) being replaced by \(\Cncone_{\Gamma}({\z}^\nu)\),
	then we say that \(\bar x\) is \emph{asymptotically C-stationary (AC-stationary)}.
	\end{definition}

	AM-stationarity was introduced for general nonsmooth optimization problems in~\cite[Definition 3.1]{mehlitz2020b},
	and the notion of AC-stationarity in the context of MPCCs was introduced in~\cite[Definition 3.3]{andreani2019}.
	Asymptotic stationarity is a sequential optimality condition,
	whereas the stationarity concepts in \cref{definition_stationary} are pointwise optimality conditions.
	As shown in~\cite[Section 3.1.2]{birgin2014} and~\cite[Theorem 3.2]{mehlitz2020b},
	a local minimizer need not be a stationary point,
	but it is an asymptotically stationary point.

\subsection{Constraint Qualifications}
	To guarantee that a local minimizer of \MPVCC{}
	is stationary in the sense of \cref{definition_stationary},
	rather than asymptotically stationary,
	special MPVCC-tailored CQs are required.
	The main CQs employed in this paper are presented below.

	\begin{definition}\label{definition_MPVCC_CCP}
	Relative to \MPVCC{} with notation~\eqref{define_constraint}--\eqref{define_cone},
	let \(\bar x\in\R^{n}\) be a feasible point.
	\begin{enumerate}
		\item\label{CCP_common} \(\bar x\) is said to satisfy the \emph{MPVCC-asymptotic Mordukhovich regularity (MPVCC-AMR)} if the set-valued mapping
		\(
		S(x,\z)\coloneqq
		J\Phi{(x)}^\top \ncone_{\Gamma}(\z)
		\)
		is outer semicontinuous at \((\bar x,\Phi(\bar x))\).
		\hfill\textnormal{(MPVCC-AMR)}
		\item\label{CQ_Clarke} \(\bar x\) is said to satisfy the \emph{MPVCC-asymptotic Clarke regularity (MPVCC-ACR)} if the set-valued mapping
		\(
		S^{\rm C}(x,\z)\coloneqq
		J\Phi{(x)}^\top \Cncone_{\Gamma}(\z)
		\)
		is outer semicontinuous at \((\bar x,\Phi(\bar x))\).
		\hfill\textnormal{(MPVCC-ACR)}
	\end{enumerate}

	\end{definition}

	\CCPM{} and \CCPC{} extend the asymptotic Mordukhovich and Clarke regularity concepts for MPCC from~\cite[Definition 4.1]{andreani2019} to \MPVCC{}.
	The following result establishes conditions under which a feasible point is C- or M-stationary.

	\begin{theorem}\label{theorem_AM_is_M}
	Relative to \MPVCC{} with notation~\eqref{define_constraint}--\eqref{define_cone},
	let \(\bar x\in\R^n\) be a feasible point.
	The following hold:
	\begin{itemize}
		\item If \(\bar x\) is AC-stationary and satisfies \CCPC{},
	then \(\bar x\) is a C-stationary point.
		\item If \(\bar x\) is AM-stationary and satisfies \CCPM{},
	then \(\bar x\) is an M-stationary point.
	\end{itemize}
	\end{theorem}

	\begin{proof}
	Suppose that \(\bar x\) is an AC-stationary point of \MPVCC{}.
	By \cref{definition_asymptotic_stationary}, there exist sequences
	\(\seq{x^\nu,\y^\nu,\z^\nu}[\nu\in\N]\) such that
	\(x^\nu\to\bar x\),
	\(\Gamma\ni\z^\nu\to\Phi(\bar x)\),
	\({\y}^\nu\in \Cncone_{\Gamma}({\z}^\nu)\),
	and
	\(
		r^\nu\coloneqq
		\nabla f(x^\nu)
		+
		J\Phi{(x^\nu)}^\top \y^\nu
		\to
		\bm{0}
	\).
	Then, it is straightforward to see that
	\[
	\begin{aligned}
		-\nabla f(\bar x)
		=
		\lim_{\nu\to\infty}
		-\nabla f(x^\nu)
		&=
		\lim_{\nu\to\infty}
		\begin{pmatrix}
		-r^\nu
		+
		J\Phi{(x^\nu)}^\top \y^\nu
		\end{pmatrix}
		\\
		&
		\in
		\Limsup_{
		(x^\nu,\z^\nu)\to
		(\bar x,\Phi(\bar x))
		}
		J\Phi{(x^\nu)}^\top
		\Cncone_{\Gamma}(\z^\nu)
		\\
		\text{\scriptsize by \rm\CCPC}
		&
		\subseteq
		J\Phi{(\bar x)}^\top
		\Cncone_{\Gamma}(\Phi(\bar x)),
	\end{aligned}
	\]
	which implies that \(\bar x\) is a C-stationary point of \MPVCC{}.
	A similar argument applies to the second item.
	\end{proof}

	In the following, we show that \CCPM{} is weaker than the RCPLD-type CQ,
	which has been shown to be weaker than the problem-tailored versions of
	LICQ,
	MFCQ,
	no nonzero abnormal multiplier constraint qualification (NNAMCQ),
	CPLD and
	RCRCQ{};
	see~\cite{andreani2012,guo2013,guo2013a,jia2023,qi2000,ramos2021,ye2014}.
	To present RCPLD, we restrict our attention to the case in which the set \(C\) encodes the inequality and equality constraints.
	Specifically, let \(m_1,m_2\in\N\) be integers such that \(m=m_1+m_2\),
	and let \(A=(A_1,\ldots,A_m)\) with \(A_i:\R^n\to\R\),
	and \(C=\R_{-}^{m_1}\times{\{0\}}^{m_2}\).
	The constraint \(A(x)\in C\) can be equivalently written as
	\begin{equation}\label{A_eq_ineq}
		A_i(x)\leq0,\ i\in\{1,\ldots,m_1\}
		\quad\text{and}\quad
		A_i(x)=0,\ i\in\{m_1+1,\ldots,m\}.
	\end{equation}
	For any feasible point \(x\in\R^n\) of \MPVCC, let \(\z=\Phi(x)\).
	We introduce the index sets
	\begin{subequations}\label{index_set_uniq_mult}
	\begin{align}
		&\mathcal{A}(x)\coloneqq
		\set{i\in\{1,\ldots,m_1\}}[A_i(x)=0]
		\\
		&
		\mathcal{E}\coloneqq\{m_1+1,\ldots,m\}
		\\
		&\mathcal{I}_{\rm uniq}(\z)\coloneqq\set{(i,j)}[i\in\mathcal{I}_{=0}(z^j)\ \text{such that}\ |\mathcal{I}_{=0}(z^j)|=1]
		\\
		&\mathcal{I}_{\rm mult}(\z)\coloneqq\set{(i,j)}[i\in\mathcal{I}_{=0}(z^j)\ \text{such that}\ |\mathcal{I}_{=0}(z^j)|>1].
	\end{align}
	\end{subequations}
	The subscripts `uniq' and `mult' indicate that \(z^j\) has a unique zero component or multiple zero components, respectively.
	Furthermore, let
	\(\mathcal{I}_{1}\subseteq\{1,\ldots,m\}\) and
	\(\mathcal{I}_{2}\subseteq\set{(i,j)}[i\in\{1,\ldots,p_j\},\ j\in\{1,\ldots,r\}]\), we define
	\begin{equation}\label{define_G}
		\mathcal{G}(x;\mathcal{I}_{1},\mathcal{I}_{2})
		\coloneqq
		\set{\nabla A_\iota(x),\nabla F_{i}^{j}(x)}
		[\iota\in\mathcal{I}_{1},\ (i,j)\in\mathcal{I}_{2}].
	\end{equation}
	Given a vector \(\y\) as defined in~\eqref{define_vector},
	we say that \(\y\) is \emph{supported on} \((\mathcal{I}_{1},\mathcal{I}_{2})\),
	if \(y_{\ell}^{0}=0\) for \(\ell\notin\mathcal{I}_{1}\) and
	\(y_{i}^{j}=0\) for \((i,j)\notin\mathcal{I}_{2}\).
	Clearly, if \(\y\) is supported on \((\mathcal{I}_{1},\mathcal{I}_{2})\),
	then \(J\Phi{(x)}^\top\y\) is a linear combination of the gradients in
	\(\mathcal{G}(x;\mathcal{I}_{1},\mathcal{I}_{2})\),
	with coefficients given by the corresponding components of \(\y\).
	Based on which, MPVCC-RCPLD is introduced as follows, extending MPCC-RCPLD from~\cite[Definition 4.1]{ramos2021} to \MPVCC{}.

	\begin{definition}[MPVCC-RCPLD]\label{definition_MPVCC_RCPLD}
	Relative to \MPVCC{} with notation~\eqref{define_constraint}--\eqref{define_cone} and setting~\eqref{A_eq_ineq},
	a feasible point \(\bar x\in\R^{n}\), with \(\bar \z=\Phi(\bar x)\), is said to have the \emph{MPVCC-relaxed constant positive linear dependence (MPVCC-RCPLD)} property if
	\begin{enumerate}
		\item there exist
		\(\mathcal{E}'\subseteq\mathcal{E}\) and
		\(\mathcal{I}_{\rm uniq}'\subseteq\mathcal{I}_{\rm uniq}(\bar \z)\) such that
		\(\mathcal{G}(\bar x;\mathcal{E}',\mathcal{I}_{\rm uniq}')\) is a basis of
		\(\spa\mathcal{G}(\bar x;\mathcal{E},\mathcal{I}_{\rm uniq}(\bar \z))\),\label{RCPLD_i}
		\item \(\mathcal{G}(x;\mathcal{E},\mathcal{I}_{\rm uniq}(\bar \z))\) has constant rank in a neighborhood of \(\bar x\), and\label{RCPLD_ii}
		\item for each subset \(\mathcal{A}'\subseteq\mathcal{A}(\bar x)\) and \(\mathcal{I}_{\rm mult}'\subseteq\mathcal{I}_{\rm mult}(\bar \z)\),
		if there is a multiplier \(\y\),
		whose components indexed by \(\mathcal{E}'\cup\mathcal{A}'\) and \(\mathcal{I}_{\rm uniq}'\cup\mathcal{I}_{\rm mult}'\) are not all zero,
		satisfying
		\(y^{0}_{\ell}\geq0\) for \(\ell\in\mathcal{A}'\),
		\(y^{j}\in \ncone_{D^j}(\bar{z}^j)\) for all \(j\), and
		\begin{equation}\label{RCPLD_PLD}
		\sum_{\ell\in\mathcal{E}'\cup\mathcal{A}'}
		y^{0}_{\ell}\nabla A_{\ell}(\bar x)
		+
	\sum_{(i,j)\in\mathcal{I}_{\rm uniq}'\cup\mathcal{I}_{\rm mult}'}
		y^{j}_{i}\nabla F^{j}_{i}(\bar x)
		=
		\bm{0},
		\end{equation}
		then the gradients in the set
		\[
		\mathcal{G}
		\left(
			x;
			\mathcal{E}'\cup\mathcal{A}',
			\mathcal{I}_{\rm uniq}'\cup\mathcal{I}_{\rm mult}'
		\right)
		\]
		are linearly dependent in a neighborhood of \(\bar x\).\label{RCPLD_iii}
	\end{enumerate}
	\end{definition}

	We first present the following two results, which will be used to show that \CCPM{} is implied by \RCPLD{}.

	\begin{fact}[{\cite[Lemma 1]{andreani2012}}]\label{lemma_linearly_independent}
	Let \(w=\sum_{i=1}^{s_1+s_2}\alpha_i v_i\) with \(\alpha_i\in\R\), \(v_i\in\R^n\), the set \({\{v_i\}}_{i=1}^{s_1}\) is linearly independent, and \(\alpha_i\neq0\) for \(i\in\{s_1+1,\ldots,s_1+s_2\}\).
	Then, there exist a subset \(\mathcal{S} \subseteq\{s_1+1,\ldots,s_1+s_2\}\) and scalars \(\tilde \alpha_i\) for \(i\in\{1,\ldots,s_1\}\cup\mathcal{S}\) such that
		\(w=\sum_{i\in\{1,\ldots,s_1\}\cup\mathcal{S}}\tilde\alpha_i v_i\),
		\(\alpha_i\tilde\alpha_i>0\) for \(i\in\mathcal{S}\), and
		the set \({\{v_i\}}_{i\in\{1,\ldots,s_1\}\cup\mathcal{S}}\) is linearly independent.
	\end{fact}

	\begin{lemma}\label{lemma_sign_preservation}
	Relative to \MPVCC{} with notation~\eqref{define_constraint}--\eqref{define_cone} and setting~\eqref{A_eq_ineq},
	let \(\bar x\in\R^{n}\) be a feasible point, with \(\bar \z=\Phi(\bar x)\).
	Suppose that there exist
	\(\mathcal{E}'\subseteq\mathcal{E}\) and
	\(\mathcal{I}_{\rm uniq}'\subseteq\mathcal{I}_{\rm uniq}(\bar \z)\)
	such that
	\(\mathcal{G}(x;\mathcal{E}',\mathcal{I}_{\rm uniq}')\) is a basis of
	\(\spa\mathcal{G}(x;\mathcal{E},\mathcal{I}_{\rm uniq}(\bar \z))\)
	for all \(x\) in a neighborhood of \(\bar x\).
	Let
	\(x^\nu\to\bar x\),
	\(\Gamma\ni\z^\nu\to\bar\z\),
	\(\y^\nu\in \Cncone_{\Gamma}(\z^\nu)\), and
	\(w^\nu=J\Phi{(x^\nu)}^\top\y^\nu\).
	Then, there exist
	\(K\in \mathcal{N}_{\infty}^{\sharp}\),
	\(\mathcal{A}'\subseteq\mathcal{A}(\bar x)\),
	\(\mathcal{I}_{\rm mult}'\subseteq\mathcal{I}_{\rm mult}(\bar\z)\),
	and \(\seq{\tilde \y^{\nu}}[\nu\in K]\),
	such that, for all \(\nu\in K\):
	\begin{enumerate}
		\item \label{lemma_sign_preservation_1}
		\(w^\nu=J\Phi{(x^\nu)}^\top\tilde\y^\nu\),
		and \(\tilde\y^\nu\) is supported on
		\((\mathcal{E}'\cup\mathcal{A}',\mathcal{I}_{\rm uniq}'\cup\mathcal{I}_{\rm mult}')\),
		\item \label{lemma_sign_preservation_2}
		the gradients in the set
		\(
		\mathcal{G}
		\left(
			x^{\nu};
			\mathcal{E}'\cup\mathcal{A}',
			\mathcal{I}_{\rm uniq}'\cup\mathcal{I}_{\rm mult}'
		\right)
		\)
		are linearly independent, and
		\item \label{lemma_sign_preservation_3}
		\(\tilde \y^{\nu}\in \Cncone_{\Gamma}(\bar\z)\);
		if, in addition,
		\(\y^\nu\in \ncone_{\Gamma}(\z^\nu)\) for all \(\nu\), then
		\(\tilde \y^{\nu}\in \ncone_{\Gamma}(\bar\z)\).
	\end{enumerate}

	\end{lemma}

	\begin{proof}
	Since \(\z^\nu\to\bar\z\), for \(\nu\) sufficiently large, we have
	\(\vit{z_{\ell}}[0][\nu]<0\) for every
	\(\ell\in\set{1,\ldots,m_{1}}\setminus\mathcal{A}(\bar x)\), and hence \(\vit{y_{\ell}}[0][\nu]=0\);
	moreover, \(\vit{y_{\ell}}[0][\nu]\geq0\) for \(\ell\in\mathcal{A}(\bar x)\),
	and \(\mathcal{I}_{=0}(\vit{z}[j][\nu])\subseteq\mathcal{I}_{=0}(\bar z^{j})\) for \(j\in\set{1,\ldots,r}\).
	By~\eqref{Clarke_cone}, \(\vit{y}[j][\nu]\) is supported on \(\mathcal{I}_{=0}(\vit{z}[j][\nu])\),
	and hence also on \(\mathcal{I}_{=0}(\bar z^{j})\);
	moreover, its nonzero components share a common sign.
	Therefore, \(\y^{\nu}\in \Cncone_{\Gamma}(\bar\z)\),
	Furthermore, if \(\y^\nu\in \ncone_{\Gamma}(\z^\nu)\),
	then by~\eqref{limiting_normal_cone},
	either \(\vit{y}[j][\nu]\) has at most one nonzero component,
	indexed by an element of \(\mathcal{I}_{=0}(\vit{z}[j][\nu])\),
	or all components indexed by \(\mathcal{I}_{=0}(\vit{z}[j][\nu])\) are negative.
	Consequently, by \(\mathcal{I}_{=0}(\vit{z}[j][\nu])\subseteq\mathcal{I}_{=0}(\bar z^{j})\), we obtain \(\y^{\nu}\in \ncone_{\Gamma}(\bar\z)\).
	Hence, for \(\nu\) sufficiently large,
	\[
		w^\nu
		=
		\sum_{\ell\in\mathcal{E}}
		\vit{y_{\ell}}[0][\nu]
		\nabla A_{\ell}(x^\nu)
		+
		\sum_{(i,j)\in\mathcal{I}_{\rm uniq}(\bar\z)}
		\vit{y_{i}}[j][\nu]
		\nabla F^{j}_{i}(x^\nu)
		+
		\sum_{\ell\in\mathcal{A}(\bar x)}
		\vit{y_{\ell}}[0][\nu]
		\nabla A_{\ell}(x^\nu)
		+
		\sum_{(i,j)\in\mathcal{I}_{\rm mult}(\bar\z)}
		\vit{y_{i}}[j][\nu]
		\nabla F^{j}_{i}(x^\nu).
	\]
	By the assumption of \cref{lemma_sign_preservation}, the first two sums can be expressed as a linear combination of the linearly independent vectors in
	\(\mathcal{G}(x^{\nu};\mathcal{E}',\mathcal{I}_{\rm uniq}')\).
	Discarding the terms with zero coefficients in the last two sums and invoking \cref{lemma_linearly_independent},
	we obtain a set \(\mathcal{S}^{\nu}\) of indices in \(\ell\in\mathcal{A}(\bar x)\) and pairs in \(\mathcal{I}_{\rm mult}(\bar\z)\),
	together with a vector \(\tilde \y^{\nu}\) supported on
	\(
		(
		\mathcal{E}'\cup(\mathcal{A}(\bar x)\cap\mathcal{S}^{\nu}),
		\mathcal{I}_{\rm uniq}'\cup(\mathcal{I}_{\rm mult}(\bar\z)\cap\mathcal{S}^{\nu})
		)
	\),
	such that
	\(w^\nu=J\Phi{(x^\nu)}^\top\tilde\y^\nu\).
	The corresponding set of gradients is linearly independent,
	and \(\vit{\tilde{y}_{\ell}}[0][\nu]\vit{y_{\ell}}[0][\nu]>0\),
	\(\vit{\tilde{y}_{i}}[j][\nu]\vit{y_{i}}[j][\nu]>0\) for all indices in \(\mathcal{S}^{\nu}\).
	Since there are finitely many possible sets \(\mathcal{S}^{\nu}\),
	there exists a subsequence \(K\in \mathcal{N}_{\infty}^{\sharp}\) such that \(\mathcal{S}^{\nu}=\mathcal{S}\) for \(\nu\in K\).
	Setting \(\mathcal{A}'\coloneqq\mathcal{A}(\bar x)\cap\mathcal{S}\) and
	\(\mathcal{I}_{\rm mult}'\coloneqq\mathcal{I}_{\rm mult}(\bar\z)\cap\mathcal{S}\) establishes \cref{lemma_sign_preservation_1,lemma_sign_preservation_2}.

	\cref{lemma_sign_preservation_3} follows from the fact that \(\y^{\nu}\in \Cncone_{\Gamma}(\bar\z)\) or \(\y^{\nu}\in \ncone_{\Gamma}(\bar\z)\), together with the sign preservation in \cref{lemma_linearly_independent}.
	\end{proof}

	\begin{theorem}\label{theorem_RCPLD_M}
	Relative to \MPVCC{} with notation~\eqref{define_constraint}--\eqref{define_cone} and setting~\eqref{A_eq_ineq},
	if a feasible point \(\bar x\in\R^{n}\), with \(\bar \z=\Phi(\bar x)\), satisfies \RCPLD{}, then \(\bar x\) also satisfies \CCPM{}.
	\end{theorem}

	\begin{proof}
	Let
	\(
	\bar w
	\in
	\Limsup_{
		(x^\nu,\z^\nu)\to
		(\bar x,\Phi(\bar x))
		}
		J\Phi{(x^\nu)}^\top \ncone_{\Gamma}(\z^\nu)
	\).
	Then, there exist sequences
	\(x^\nu\to\bar x\),
	\(\Gamma\ni\z^\nu\to\bar\z=\Phi(\bar x)\), and
	\(\y^\nu\in \ncone_{\Gamma}(\z^\nu)\) such that
	\(w^\nu=J\Phi{(x^\nu)}^\top\y^\nu\to\bar w\).
	Let \(\mathcal{E}'\) and \(\mathcal{I}_{\rm uniq}'\) be as in \RCPLDi.
	Since \(\mathcal{G}(\bar x;\mathcal{E}',\mathcal{I}_{\rm uniq}')\) is linearly independent,
	it follows from the continuity of the functions involved that \(\mathcal{G}(x;\mathcal{E}',\mathcal{I}_{\rm uniq}')\) is linearly independent for all \(x\) in a neighborhood of \(\bar x\).
	Then, by \RCPLDii, the number of vectors in
	\(\mathcal{G}(x;\mathcal{E}',\mathcal{I}_{\rm uniq}')\)
	equals the rank of
	\(\mathcal{G}(x;\mathcal{E},\mathcal{I}_{\rm uniq}(\bar \z))\),
	and thus the assumption of \cref{lemma_sign_preservation} holds.
	Let
	\(K\),
	\(\mathcal{A}'\),
	\(\mathcal{I}_{\rm mult}'\),
	and \(\seq{\tilde \y^{\nu}}[\nu\in K]\) be as in \cref{lemma_sign_preservation},
	and suppose, for contradiction, that \(\seq{\tilde \y^{\nu}}[\nu\in K]\) is unbounded.
	Dividing both sides of \(w^\nu=J\Phi{(x^\nu)}^\top\tilde\y^\nu\) by
	\(\|\tilde \y^\nu\|\)
	and passing to the limit along the subsequence \(K\), we obtain a unit vector \(\tilde\y\), supported on
	\(
	(
		\mathcal{E}'\cup\mathcal{A}',
		\mathcal{I}_{\rm uniq}'\cup\mathcal{I}_{\rm mult}'
	)
	\), such that
	\(J\Phi{(\bar x)}^\top\tilde\y=\bm{0}\), that is, \eqref{RCPLD_PLD} holds with \(\y=\tilde\y\).
	Since \(\ncone_{\Gamma}(\bar\z)\) is a closed cone, \cref{lemma_sign_preservation_3} implies that \(\tilde \y\in \ncone_{\Gamma}(\bar\z)\).
	By \RCPLDiii, the gradients in the set
	\(
		\mathcal{G}
		\left(
		x;
		\mathcal{E}'\cup\mathcal{A}',
		\mathcal{I}_{\rm uniq}'\cup\mathcal{I}_{\rm mult}'
		\right)
	\)
	are linearly dependent for all \(x\) in a neighborhood of \(\bar x\),
	contradicting \cref{lemma_sign_preservation_2} for \(\nu\) sufficiently large.
	Therefore, \(\seq{\tilde \y^{\nu}}[\nu\in K]\) is bounded,
	and any of its cluster point \(\tilde\y\) satisfies \(\tilde \y\in \ncone_{\Gamma}(\bar\z)\) and \(\bar w=J\Phi{(\bar x)}^\top\tilde\y\),
	which implies \CCPM{}.
	\end{proof}

	By~\cite[Examples 4--6]{andreani2019}, neither \CCPC{} nor \CCPM{} implies the other.
	Nevertheless, we next show that \CCPC{} is weaker than a variant of RCRCQ,
	which is known to be weaker than the problem-tailored versions of
	LICQ, linear CQ and CRCQ{};
	see~\cite{guo2013a,ramos2021}.
	MPVCC-RCRCQ is introduced as follows, extending MPCC-RCRCQ from~\cite[Definition 3.4(ii)]{guo2013a} to \MPVCC{}.

	\begin{definition}[MPVCC-RCRCQ]\label{definition_MPVCC_RCRCQ}
	Relative to \MPVCC{} with notation~\eqref{define_constraint}--\eqref{define_cone} and setting~\eqref{A_eq_ineq},
	a feasible point \(\bar x\in\R^{n}\), with \(\bar \z=\Phi(\bar x)\), is said to have the \emph{MPVCC-relaxed constant rank constraint qualification (MPVCC-RCRCQ)} property if,
	for each subset
	\(\mathcal{A}'\subseteq\mathcal{A}(\bar x)\) and
	\(\mathcal{I}_{\rm mult}'\subseteq\mathcal{I}_{\rm mult}(\bar \z)\),
	the gradients in the set
	\[
		\mathcal{G}
		\left(
		x;
		\mathcal{E}\cup\mathcal{A}',
		\mathcal{I}_{\rm uniq}(\bar \z)\cup\mathcal{I}_{\rm mult}'
		\right)
	\]
	have constant rank in a neighborhood of \(\bar x\).

	\end{definition}

	\begin{lemma}\label{lemma_linearly_dependent}
	Relative to \MPVCC{} with notation~\eqref{define_constraint}--\eqref{define_cone} and setting~\eqref{A_eq_ineq},
	let \(\bar x\in\R^{n}\) be a feasible point, with \(\bar \z=\Phi(\bar x)\).
	Suppose that \(\bar x\) satisfies \RCRCQ{}, and let
	\(\mathcal{E}'\subseteq\mathcal{E}\) and
	\(\mathcal{I}_{\rm uniq}'\subseteq\mathcal{I}_{\rm uniq}(\bar \z)\)
	be such that
	\(\mathcal{G}(\bar x;\mathcal{E}',\mathcal{I}_{\rm uniq}')\) is a basis of
	\(\spa\mathcal{G}(\bar x;\mathcal{E},\mathcal{I}_{\rm uniq}(\bar \z))\).
	Then,
	\(\mathcal{G}(x;\mathcal{E}',\mathcal{I}_{\rm uniq}')\) is a basis of
	\(\spa\mathcal{G}(x;\mathcal{E},\mathcal{I}_{\rm uniq}(\bar \z))\)
	for all \(x\) in a neighborhood of \(\bar x\).
	Moreover, for each
	\(\mathcal{A}'\subseteq\mathcal{A}(\bar x)\) and
	\(\mathcal{I}_{\rm mult}'\subseteq\mathcal{I}_{\rm mult}(\bar\z)\),
	if the gradients in
	\(
		\mathcal{G}
		\left(
		\bar x;
		\mathcal{E}'\cup\mathcal{A}',
		\mathcal{I}_{\rm uniq}'\cup\mathcal{I}_{\rm mult}'
		\right)
	\)
	are linearly dependent, then the gradients in
	\(
		\mathcal{G}
		\left(
		x;
		\mathcal{E}'\cup\mathcal{A}',
		\mathcal{I}_{\rm uniq}'\cup\mathcal{I}_{\rm mult}'
		\right)
	\)
	are also linearly dependent for all \(x\) in a neighborhood of \(\bar x\).
	\end{lemma}

	\begin{proof}
	The first claim follows as in the proof of \cref{theorem_RCPLD_M},
	since \RCRCQ{} with \(\mathcal{A}'=\mathcal{I}_{\rm mult}'=\emptyset\)
	implies that \(\mathcal{G}(x;\mathcal{E},\mathcal{I}_{\rm uniq}(\bar\z))\) has constant rank in a neighborhood of \(\bar x\).
	As for the second claim, let
	\(
		G_{1}(x)\coloneqq
		\mathcal{G}
		\left(
		x;
		\mathcal{E}'\cup\mathcal{A}',
		\mathcal{I}_{\rm uniq}'\cup\mathcal{I}_{\rm mult}'
		\right)
	\) and
	\(
		G_{2}(x)\coloneqq
		\mathcal{G}
		\left(
		x;
		\mathcal{E}\cup\mathcal{A}',
		\mathcal{I}_{\rm uniq}(\bar\z)\cup\mathcal{I}_{\rm mult}'
		\right)
	\).
	By the first claim, \(\spa G_{1}(x)=\spa G_{2}(x)\) for all \(x\) in a neighborhood of \(\bar x\).
	By \RCRCQ{}, the rank of \(G_{2}(x)\) is constant near \(\bar x\).
	Therefore, \(\rank G_{1}(x)=\rank G_{2}(x)=\rank G_{2}(\bar x)=\rank G_{1}(\bar x)<|G_{1}(\bar x)|\)
	for all \(x\) in a neighborhood of \(\bar x\),
	which implies that
	the gradients in
	\(
		\mathcal{G}
		\left(
		x;
		\mathcal{E}'\cup\mathcal{A}',
		\mathcal{I}_{\rm uniq}'\cup\mathcal{I}_{\rm mult}'
		\right)
	\)
	are linearly dependent.
	\end{proof}

	\begin{theorem}
	Relative to \MPVCC{} with notation~\eqref{define_constraint}--\eqref{define_cone} and setting~\eqref{A_eq_ineq},
	if a feasible point \(\bar x\in\R^{n}\), with \(\bar \z=\Phi(\bar x)\), satisfies \RCRCQ{}, then \(\bar x\) also satisfies \CCPC{}.
	\end{theorem}

	\begin{proof}
	Let
	\(
	\bar w
	\in
	\Limsup_{
		(x^\nu,\z^\nu)\to
		(\bar x,\Phi(\bar x))
		}
		J\Phi{(x^\nu)}^\top \Cncone_{\Gamma}(\z^\nu)
	\).
	Then, there exist sequences
	\(x^\nu\to\bar x\),
	\(\Gamma\ni\z^\nu\to\bar\z=\Phi(\bar x)\), and
	\(\y^\nu\in \Cncone_{\Gamma}(\z^\nu)\) such that
	\(w^\nu=J\Phi{(x^\nu)}^\top\y^\nu\to\bar w\).
	Let \(\mathcal{E}'\) and \(\mathcal{I}_{\rm uniq}'\) be as in \cref{lemma_linearly_dependent}.
	Then, the assumption of \cref{lemma_sign_preservation} holds.
	Let
	\(K\),
	\(\mathcal{A}'\),
	\(\mathcal{I}_{\rm mult}'\),
	and \(\seq{\tilde \y^{\nu}}[\nu\in K]\) be as in \cref{lemma_sign_preservation},
	and suppose, for contradiction, that \(\seq{\tilde \y^{\nu}}[\nu\in K]\) is unbounded.
	Dividing both sides of \(w^\nu=J\Phi{(x^\nu)}^\top\tilde\y^\nu\) by
	\(\|\tilde \y^\nu\|\)
	and passing to the limit along the subsequence \(K\),
	we obtain a unit vector \(\tilde\y\), supported on
	\(
	(
		\mathcal{E}'\cup\mathcal{A}',
		\mathcal{I}_{\rm uniq}'\cup\mathcal{I}_{\rm mult}'
	)
	\), such that
	\(J\Phi{(\bar x)}^\top\tilde\y=\bm{0}\).
	By \cref{lemma_linearly_dependent},
	the gradients in
	\(
		\mathcal{G}
		\left(
		x^{\nu};
		\mathcal{E}'\cup\mathcal{A}',
		\mathcal{I}_{\rm uniq}'\cup\mathcal{I}_{\rm mult}'
		\right)
	\)
	are linearly dependent for \(\nu\) sufficiently large, contradicting \cref{lemma_sign_preservation_2}.
	Therefore, \(\seq{\tilde \y^{\nu}}[\nu\in K]\) is bounded.
	Since \(\Cncone_{\Gamma}(\bar\z)\) is a closed cone,
	any cluster point \(\tilde\y\) of \(\seq{\tilde \y^{\nu}}[\nu\in K]\) satisfies \(\tilde \y\in \Cncone_{\Gamma}(\bar\z)\) by \cref{lemma_sign_preservation_3},
	and \(\bar w=J\Phi{(\bar x)}^\top\tilde\y\).
	This establishes \CCPC{}.
	\end{proof}

	\begin{example}
		We use~\cite[Example 4.3]{ramos2021} to illustrate the implications involving \CCPM{} and \CCPC{}.
		Specifically, let \(A_{1}(x)=x_{1}+x_{2}\), \(A_{2}(x)=-x_{1}-x_{2}\), \(A_{3}(x)=x_{1}^{2}+x_{2}^{2}\), \(C=\R_{-}^{3}\), \(F^{1}(x)=(x_{1},x_{2})\) and \(D^{1}=\set{z\in\R_{+}^{2}}[z_{1}z_{2}=0]\).
		Taking \(\mathcal{A}'=\{3\}\) and \(\mathcal{I}_{\rm mult}'=\emptyset\),
		\RCPLD{} is violated at \(\bar x=(0,0)\),
		because \(\nabla A_{3}(x)\) is positive linearly dependent at \(\bar x\)
		but linearly independent in a neighborhood of \(\bar x\).
		In contrast, \CCPM{} holds at \(\bar x\),
		since \(J\Phi{(\bar x)}^\top \ncone_{\Gamma}(\Phi(\bar x))=\R^{2}\).
		Therefore, \CCPM{} does not imply \RCPLD{}.
		Similarly, \RCRCQ{} is violated at \(\bar x\),
		because \(\nabla A_{3}(x)\) has rank zero at \(\bar x\)
		but rank one in a neighborhood of \(\bar x\).
		In contrast, \CCPC{} holds at \(\bar x\),
		since \(J\Phi{(\bar x)}^\top \Cncone_{\Gamma}(\Phi(\bar x))=\R^{2}\).
		Consequently, \CCPC{} does not imply \RCRCQ{}.
	\end{example}

	\begin{example}
		\cref{lemma_linearly_dependent} shows that \RCRCQ{} implies \RCPLD.
		To illustrate that \RCPLD{} is strictly weaker than \RCRCQ{}, let
		\(A_{1}(x)=x_{1}-x_{2}^{2}\), \(A_{2}(x)=x_{1}+x_{2}^{3}\), \(C=\R_{-}^{2}\),
		\(F^{1}(x)=(x_{3},1+x_{3}^{2})\) and \(D^{1}=\set{z\in\R_{+}^{2}}[z_{1}z_{2}=0]\).
		\RCRCQ{} is violated at \(\bar x=(0,0,0)\),
		because the set of gradients \(\{\nabla A_{1}(x),\nabla A_{2}(x),\nabla F_{1}^{1}(x)\}\) has rank two at \(\bar x\)
		but rank three in a neighborhood of \(\bar x\).
		In contrast, \RCPLD{} holds at \(\bar x\),
		since \(\{\nabla F_{1}^{1}(x)\}\) has constant rank in a neighborhood of \(\bar x\),
		and no subset of \(\{\nabla A_{1}(x),\nabla A_{2}(x)\}\) satisfies~\eqref{RCPLD_PLD} together with \(\{\nabla F_{1}^{1}(x)\}\) at \(\bar x\).
	\end{example}

	To summarize, when the set \(C\) encodes the inequality and equality constraints,
	we obtain the following chain of implications:
	\[
			\text{\CCPM}
		\quad
			\raisebox{3pt}{\(
				\overset{
					\textstyle
					\mathrlap{\color{red}\hspace*{4pt}\not{}}
					{\Longrightarrow}
				}{
					\underset{\textstyle{\Longleftarrow}}{}
				}
			\)}
		\quad
			\text{\RCPLD}
		\quad
			\raisebox{3pt}{\(
				\overset{
					\textstyle
					\mathrlap{\color{red}\hspace*{4pt}\not{}}
					{\Longrightarrow}
				}{
					\underset{\textstyle{\Longleftarrow}}{}
				}
			\)}
		\quad
			\text{\RCRCQ}
		\quad
			\raisebox{3pt}{\(
				\overset{
					\textstyle
					\mathrlap{\color{red}\hspace*{6pt}\not{}}
					{\Longleftarrow}
				}{
					\underset{\textstyle{\Longrightarrow}}{}
				}
			\)}
			\quad
			\text{\CCPC}.
	\]
	Identifying sufficient conditions for \CCPM{} and \CCPC{}
	that can be verified directly from the problem data of \MPVCC{}
	remains an interesting topic for future research.

	\section{Lasry--Lions Double Envelope}
		We now introduce the concepts of the Moreau and Lasry--Lions double envelopes.

\begin{definition}[Moreau envelope]\label{definition_Moreau}
Given \(\lambda>0\), the \emph{Moreau envelope} of a function \(\varphi:\R^{n}\to\overline{\R}\) is \(\env_\lambda\varphi:\R^{n}\to\overline{\R}\) given by
\[
	\env_\lambda\varphi(x)
\coloneqq
	\inf_{x'\in\R^{n}}
	\set{
	\varphi(x')
	+\tfrac{1}{2\lambda}\|x'- x\|^2}.
\]
The set of minimizers is denoted by \(\prox_{\lambda\varphi}(x)\).
Furthermore, \(\varphi\) is \emph{prox-bounded} if there exists \(\lambda>0\) such that \(\env_\lambda\varphi(x)>-\infty\) for some \(x\in\R^{n}\).
The supremum of the set of all such \(\lambda\) is the threshold \(\lambda_\varphi\).
\end{definition}

\begin{definition}[Lasry--Lions double envelope]\label{definition_LL}
Given \(\lambda>\mu>0\), the \emph{Lasry--Lions double envelope} of \(\varphi:\R^{n}\to\overline{\R}\) is \(\env_{\lambda,\mu}\varphi:\R^{n}\to\overline{\R}\) given by
\begin{equation}\label{sup_LL_definition}
	\env_{\lambda,\mu}\varphi(x)\coloneqq-\env_{\mu}(-\env_\lambda\varphi)(x)=\sup_{x'\in\R^{n}}\set{\env_\lambda\varphi(x')-\tfrac{1}{2\mu}\|x'- x\|^2}.
\end{equation}
We denote the maximizer of~\eqref{sup_LL_definition} by \(P_{\lambda,\mu}\varphi(x)\coloneqq\prox_{-\mu\env_\lambda\varphi}(x)\).
\end{definition}

Note that when \(\varphi\) is lower bounded, such as in the case of an indicator function, we have \(\lambda_\varphi=\infty\).
We summarize below several important properties of the Lasry--Lions double envelope,
which are useful for the design and theoretical analysis of algorithms;
see~\cite[Section 2]{simoes2021} and~\cite[Proposition 3.3 and Lemma 3.4]{wang2025} for further properties.

\begin{fact}[Some properties of the Lasry--Lions double envelope]\label{proposition_properties_LL}
Let \(\varphi:\R^{n}\to\overline{\R}\) be proper, lsc, and prox-bounded with threshold \(\lambda_\varphi\).
	The following hold:%
	\begin{enumerate}
	\item
	\(\env_{\lambda,\mu}\varphi(x)\) with \(0<\mu<\lambda<\lambda_\varphi\) is continuously differentiable with \(\max\set{\frac{1}{\mu},\frac{1}{\lambda-\mu}}\)-Lipschitz continuous gradient.\label{proposition_Lipschitz}
	\item
			\(\env_\lambda\varphi(x)
	\leq
	\env_{\lambda,\mu}\varphi(x)
	\leq
	\env_{\lambda-\mu}\varphi(x)\leq\varphi(x)\) for all \(0<\mu<\lambda<\lambda_\varphi\).\label{proposition_sandwich}
		\item
			\(\inf\env_{\lambda,\mu}\varphi=\inf\env_{\lambda}\varphi=\inf\varphi\) and \(\arg\min\env_{\lambda,\mu}\varphi=\arg\min\env_{\lambda}\varphi=\arg\min\varphi\) for all \(0<\mu<\lambda<\lambda_\varphi\).\label{proposition_global_optimum}
	\item \(P_{\lambda,\mu}\varphi(x)=x+\mu\nabla\env_{\lambda,\mu}\varphi(x)\) for all \(0<\mu<\lambda<\lambda_\varphi\).\label{proposition_P_grad}
\end{enumerate}
\end{fact}

We are now in a position to derive the Lasry--Lions double envelope of \(\delta_{D^j}\).
As in \cref{sec_cones}, we temporarily suppress the dependence on \(j\) in~\eqref{D_set}, writing \(D\coloneqq D^j\) and \(p\coloneqq p_j\).
In view of~\eqref{convex_union}, the Moreau envelope of \(\delta_{D}\) becomes
\begin{equation}\label{moreau_compl_DO}
	\env_\lambda\delta_{D}(z)=\tfrac{1}{2\lambda}\dist{(z,D)}^2=\tfrac{1}{2\lambda}\min\limits_{i\in\{1,\ldots,p\}}\dist{(z,D_i)}^2.
\end{equation}
Let \( z_{(1)}\leq z_{(2)}\leq\cdots\leq z_{(p)}\) denote the components of \(z\) sorted in nondecreasing order,
and define \(z_{(p+1)}=\infty\).
The explicit expression for \(\env_{\lambda,\mu}\delta_{D}(z)\) is then given by the following result.

\begin{theorem}[Lasry--Lions double envelope of \(\delta_D\)]\label{theorem_LL_env}
Relative to the vertical complementarity set \(D\), let \(z\in\R^p\), \(\lambda>\mu>0\) and \(\beta\coloneqq\nicefrac{\mu}{\lambda}\in(0,1)\).
The Lasry--Lions double envelope of the indicator function \(\delta_D\) is given by
\begin{equation}\label{LL_DO_env}
\env_{\lambda,\mu}\delta_D(z)=\tfrac{1}{\lambda}r_{\beta}(z),
\end{equation}
with
\begin{equation}\label{rb_2_cases}
r_{\beta}(z)=
\begin{cases}
	\tfrac{1}{2(1-\beta)}\sum_{i\in\mathcal{I}_{\leq0}(z)}z_i^2,\ \ \
	&\mbox{if } \mathcal{I}_{\leq0}(z)\neq\emptyset
	\\[5pt]
	\tfrac{1}{2\beta(k_{\beta}(z)-\beta)}
		\begin{pmatrix}
		\sum_{i=1}^{k_{\beta}(z)}z_{(i)}
		\end{pmatrix}^2-\tfrac{1}{2\beta}
		\sum_{i=1}^{k_{\beta}(z)}z_{(i)}^2,\ \ \
		&\mbox{if } \mathcal{I}_{\leq0}(z)=\emptyset,
\end{cases}
\end{equation}
where \(k_{\beta}(z)\in \{1,\ldots,p\}\) is the unique integer satisfying
\begin{equation}\label{M_set}
z_{(k_{\beta}(z))}\leq
\frac{\sum_{i=1}^{k_{\beta}(z)}z_{(i)} }{k_{\beta}(z)-\beta}
<z_{(k_{\beta}(z)+1)}.
\end{equation}
When \(\mathcal{I}_{\leq0}(z)\neq\emptyset\), set \(k_{\beta}(z)=1\).
\end{theorem}

\begin{proof}
By~\eqref{sup_LL_definition}, we have
\begin{subequations}
\allowdisplaybreaks\@
\begin{align}
	\env_{\lambda,\mu}\delta_{D}(z)
	&=
	\sup_{z'\in\R^p}\{\env_\lambda\delta_{D}(z')-\tfrac{1}{2\mu}\|z'-z\|^2 \}
	\nonumber\\
	&=
	\sup_{z'\in\R^p}
	\begin{Bmatrix}
	\tfrac{1}{2\lambda}\min\limits_{i\in\{1,\ldots,p\}}\dist{(z',D_i)}^2-\tfrac{1}{2\mu}\|z'-z\|^2
	\end{Bmatrix}
	\nonumber\\
	&=
	\tfrac{1}{\lambda}
	\sup_{z'\in\R^p}
	\begin{Bmatrix}
	\min\limits_{i\in\{1,\ldots,p\}}\tfrac{1}{2}\dist{(z',D_i)}^2-\tfrac{1}{2\beta}\|z'-z\|^2
	\end{Bmatrix}
	\nonumber\\
	&=
	-\tfrac{1}{\lambda}
	\inf_{z'\in\R^p}
	\begin{Bmatrix}
	\tfrac{1}{2\beta}\|z'-z\|^2 -\min\limits_{i\in\{1,\ldots,p\}}\tfrac{1}{2}\dist{(z',D_i)}^2
	\end{Bmatrix}
	\nonumber\\
	&=
	-\tfrac{1}{\lambda}
	\inf_{z'\in\R^p}
	\begin{Bmatrix}
	\tfrac{1}{2\beta}\|z'-z\|^2 +\max\limits_{i\in\{1,\ldots,p\}}-\tfrac{1}{2}\dist{(z',D_i)}^2
	\end{Bmatrix}
	\nonumber\\
	&=
	-\tfrac{1}{\lambda}
	\inf_{z'\in\R^p}
	\begin{Bmatrix}
	\max\limits_{i\in\{1,\ldots,p\}}
	\begin{pmatrix}
		\tfrac{1}{2\beta}\|z'-z\|^2 -\tfrac{1}{2}\dist{(z',D_i)}^2
	\end{pmatrix}
	\end{Bmatrix}
	\label{p_either_or}
	\\
	&=
	-\tfrac{1}{\lambda}
	\inf_{z'\in\R^p}
	\begin{Bmatrix}
	\max\limits_{y\in\Delta_p}
	\sum_{i=1}^{p}y_i
	\begin{pmatrix}
		\tfrac{1}{2\beta}\|z'-z\|^2 -\tfrac{1}{2}\dist{(z',D_i)}^2
	\end{pmatrix}
	\end{Bmatrix},\label{p_1_simplex}
\end{align}
\end{subequations}
where \(y\coloneqq(y_1,\ldots,y_p)\in\R^p\) and \(\Delta_p\coloneqq\set{y\in{\{0,1\}}^p}[\sum_{i=1}^{p}y_i=1]\subset\R^p\).
In this way, the maximization in~\eqref{p_either_or} is equivalent to a maximization over the discrete set \(\Delta_p\) in~\eqref{p_1_simplex}.
Note that the inner maximization problem of~\eqref{p_1_simplex} corresponds to the \emph{support function} of the nonempty set \(\Delta_p\); see~\cite[Eq. 8(10)]{rockafellar1998}.
Therefore, by~\cite[Example 11.2]{bauschke2011}, it follows that~\eqref{p_1_simplex} is equivalent to
\begin{equation}\label{cl_conv_p_1_simplex}
	\env_{\lambda,\mu}\delta_{D}(z)
	=
	-\tfrac{1}{\lambda}
	\inf_{z'\in\R^p}
	\begin{Bmatrix}
	\max\limits_{y\in\conv\Delta_p}
	\sum_{i=1}^{p}y_i
	\begin{pmatrix}
		\tfrac{1}{2\beta}\|z'-z\|^2 -\tfrac{1}{2}\dist{(z',D_i)}^2
	\end{pmatrix}
	\end{Bmatrix},
\end{equation}
where \(\conv\Delta_p=\set{y\in\R_+^p}[\sum_{i=1}^{p}y_i=1]\subset\R^p\).
Note that the distance functions appearing in~\eqref{cl_conv_p_1_simplex} can be expressed explicitly.
Specifically, relative to~\eqref{convex_union}, we have \(\Pi_{D_i}(z)=(I-e_i e_i^\top){[z]}_+\), or equivalently
\[
	{[\Pi_{D_i}(z)]}_\ell=
	\left\{
	\begin{array}{l}
	{[z_\ell]}_+,\ \ \text{if}\ \ell\neq i
	\\[5pt]
	0,\ \ \ \ \ \ \ \text{if}\ \ell= i.
	\end{array}
	\right.
\]
Therefore, the distance function associated with \(D_i\) can be written as
\begin{align}
	\tfrac{1}{2}\dist{(z',D_i)}^2
	&=
	\tfrac{1}{2}\|z'-\Pi_{D_i}(z')\|^2
	\nonumber\\
	&=
	\tfrac{1}{2}\|z'\|^2-\langle z',\Pi_{D_i}(z')\rangle+\tfrac{1}{2}\|\Pi_{D_i}(z')\|^2
	\nonumber\\
	&=
	\tfrac{1}{2}\|z'\|^2-\sum_{\ell=1}^{p}z'_\ell {[\Pi_{D_i}(z')]}_\ell+\tfrac{1}{2}\sum_{\ell=1}^{p}{[\Pi_{D_i}(z')]}_\ell^2
	\nonumber\\
	&=
	\tfrac{1}{2}\|z'\|^2-\sum_{\ell\neq i}{[z'_\ell]}_+^2+\tfrac{1}{2}\sum_{\ell\neq i}{[z'_\ell]}_+^2
	\nonumber\\
	&=
	\tfrac{1}{2}\|z'\|^2-\tfrac{1}{2}
	\begin{pmatrix}
	\sum_{\ell=1}^{p}{[z'_\ell]}_+^2-{[z'_i]}_+^2
	\end{pmatrix}.\label{dist_D_i}
\end{align}
Substituting~\eqref{dist_D_i} into~\eqref{cl_conv_p_1_simplex} yields
\allowdisplaybreaks\@
\begin{align}
	&\env_{\lambda,\mu}\delta_{D}(z)
	\nonumber\\
	&=
	-\tfrac{1}{\lambda}
	\inf_{z'\in\R^p}
	\begin{Bmatrix}
	\max\limits_{y\in\conv\Delta_p}
	\sum_{i=1}^{p}y_i
	\begin{pmatrix}
		\tfrac{1}{2\beta}\|z'-z\|^2
		-
		\tfrac{1}{2}\|z'\|^2
		+
		\tfrac{1}{2}\sum_{\ell=1}^{p}{[z'_\ell]}_+^2
		-
		\tfrac{1}{2}{[z'_i]}_+^2
	\end{pmatrix}
	\end{Bmatrix}
	\nonumber\\
	&=
	-\tfrac{1}{\lambda}
	\max\limits_{y\in\conv\Delta_p}
	\begin{Bmatrix}
	\inf_{z'\in\R^p}
	\sum_{i=1}^{p}y_i
	\begin{pmatrix}
		\tfrac{1}{2\beta}\|z'-z\|^2
		-
		\tfrac{1}{2}\|z'\|^2
		+
		\tfrac{1}{2}\sum_{\ell=1}^{p}{[z'_\ell]}_+^2
		-
		\tfrac{1}{2}{[z'_i]}_+^2
	\end{pmatrix}
	\end{Bmatrix}
	\nonumber\\
	&=
	-\tfrac{1}{\lambda}
	\max\limits_{y\in\conv\Delta_p}
	\begin{Bmatrix}
	\inf_{z'\in\R^p}
	\begin{pmatrix}
		\tfrac{1}{2\beta}\|z'-z\|^2
		-
		\tfrac{1}{2}\|z'\|^2
		+
		\tfrac{1}{2}\sum_{\ell=1}^{p}{[z'_\ell]}_+^2
		-
		\tfrac{1}{2}\sum_{i=1}^{p}y_i{[z'_i]}_+^2
	\end{pmatrix}
	\end{Bmatrix}
	\nonumber\\
	&=
	-\tfrac{1}{\lambda}
	\max\limits_{y\in\conv\Delta_p}
	\begin{Bmatrix}
	\inf_{z'\in\R^p}
	\begin{pmatrix}
		\tfrac{1}{2\beta}\sum_{i=1}^{p}{(z'_i-z_i)}^2
		-
		\tfrac{1}{2}\sum_{i=1}^{p}{(z'_i)}^2
		+
		\tfrac{1}{2}\sum_{i=1}^{p}{[z'_i]}_+^2
		-
		\tfrac{1}{2}\sum_{i=1}^{p}y_i{[z'_i]}_+^2
	\end{pmatrix}
	\end{Bmatrix}
	\nonumber\\
	&=
	-\tfrac{1}{\lambda}
	\max\limits_{y\in\conv\Delta_p}
	\begin{Bmatrix}
	\sum_{i=1}^{p}
	\inf_{z'_i\in\R}
	\begin{pmatrix}
		\tfrac{1}{2\beta}{(z'_i-z_i)}^2
		-
		\tfrac{1}{2}{(z'_i)}^2
		+
		\tfrac{1}{2}{[z'_i]}_+^2
		-
		\tfrac{1}{2}y_i{[z'_i]}_+^2
	\end{pmatrix}
	\end{Bmatrix},\label{pointwise_min}
\end{align}
in which the second equality follows from the \emph{minimax theorem}, since the inner function is linear in \(y\) and convex in \(z'\), owing to the fact that \(\beta\in(0,1)\).
The third equality follows from the constraint \(\sum_{i=1}^{p}y_i=1\).
Consequently, we obtain \(p\) one-dimensional inner convex minimization problems in~\eqref{pointwise_min}, each of which admits the minimizer, together with the corresponding function value, given by
\[
{z'_i}^\star=
	\begin{cases}
	\tfrac{1}{1-\beta}z_i,\ \ &\mbox{if } z_i\leq0
	\\[5pt]
	\tfrac{1}{1-\beta y_i}z_i,\ \ &\mbox{if } z_i>0,
	\end{cases}
\]
and
\begin{equation}\label{p_1_obj}
	\tfrac{1}{2\beta}{({z'_i}^\star-z_i)}^2
	-
	\tfrac{1}{2}{({z'_i}^\star)}^2
	+
	\tfrac{1}{2}{[{z'_i}^\star]}_+^2
	-
	\tfrac{1}{2}y_i{[{z'_i}^\star]}_+^2
	=
	\begin{cases}
	-\tfrac{1}{2(1-\beta)}z_i^2,\ \ &\mbox{if } z_i\leq0
	\\[5pt]
	-\tfrac{y_i}{2(1-\beta y_i)}z_i^2,\ \ &\mbox{if } z_i>0.
	\end{cases}
\end{equation}
In this way, the Lasry--Lions double envelope in~\eqref{pointwise_min} can be formulated as
\begin{align}
	\env_{\lambda,\mu}\delta_{D}(z)
	&\overset{\eqref{p_1_obj}}{=}
	-\tfrac{1}{\lambda}
	\max\limits_{y\in\conv\Delta_p}
	\begin{Bmatrix}
	\sum_{i\in\mathcal{I}_{\leq0}(z)}
	\begin{pmatrix}
		-\tfrac{1}{2(1-\beta)}z_i^2
	\end{pmatrix}
	+
	\sum_{i\in\mathcal{I}_{>0}(z)}
	\begin{pmatrix}
		-\tfrac{y_i}{2(1-\beta y_i)}z_i^2
	\end{pmatrix}
	\end{Bmatrix}
	\nonumber\\
	&=
	\tfrac{1}{\lambda}
	\min\limits_{y\in\conv\Delta_p}
	\begin{Bmatrix}
	\sum_{i\in\mathcal{I}_{\leq0}(z)}
		\tfrac{1}{2(1-\beta)}z_i^2
	+
	\sum_{i\in\mathcal{I}_{>0}(z)}
		\tfrac{y_i}{2(1-\beta y_i)}z_i^2
	\end{Bmatrix}
	\nonumber\\
	&=
	\tfrac{1}{\lambda}
	\sum_{i\in\mathcal{I}_{\leq0}(z)}
	\tfrac{1}{2(1-\beta)}z_i^2
	+
	\tfrac{1}{\lambda}
	\min\limits_{y\in\conv\Delta_p}
	\begin{Bmatrix}
	\sum_{i\in\mathcal{I}_{>0}(z)}
		\tfrac{y_i}{2(1-\beta y_i)}z_i^2
	\end{Bmatrix}.\label{min_st_clconv}
\end{align}
Recalling that \(\conv\Delta_p=\set{y\in\R_+^p}[\sum_{i=1}^{p}y_i=1]\), the minimization problem in~\eqref{min_st_clconv} is equivalent to
\[
	\minimize_{y\in\R_+^p}
	\
	\sum_{i\in\mathcal{I}_{>0}(z)}
	\tfrac{y_i}{2(1-\beta y_i)}z_i^2
	\quad
	\stt\
	\sum_{i=1}^{p}y_i=1.
\]
Let \(\tau\in\R\) be the Lagrange multiplier associated with the constraint \(\sum_{i=1}^{p}y_i=1\).
Then, the Lagrangian is given by
\begin{align}
	\mathcal{L}(y,\tau)&=\sum_{i\in\mathcal{I}_{>0}(z)}
	\tfrac{y_i}{2(1-\beta y_i)}z_i^2
	+\tau
	\begin{pmatrix}
	1-\sum_{i=1}^{p}y_i
	\end{pmatrix}
	\nonumber\\
	&=
	\tau+
	\sum_{i\in\mathcal{I}_{>0}(z)}
	\begin{pmatrix}
	\tfrac{y_i}{2(1-\beta y_i)}z_i^2-\tau y_i
	\end{pmatrix}
	-
	\tau
	\sum_{i\in\mathcal{I}_{\leq0}(z)}y_i,\label{Lagrangian_negative_positive}
\end{align}
and the dual function is obtained as
\begin{equation}\label{dual_function}
	\psi(\tau)\coloneqq \inf_{y}\mathcal{L}(y,\tau)=
	\tau
	+
	\sum_{i\in\mathcal{I}_{\leq0}(z)}\inf_{y_i\in\R_{+}}\{-\tau y_i\}
	+
	\sum_{i\in\mathcal{I}_{>0}(z)}
	\inf_{y_i\in[0,\nicefrac{1}{\beta})}
	\begin{Bmatrix}
	\tfrac{y_i}{2(1-\beta y_i)}z_i^2-\tau y_i
	\end{Bmatrix},
\end{equation}
where we restrict the domain of \(y_i\) for \(i\in\mathcal{I}_{>0}(z)\) to \([0,\nicefrac{1}{\beta})\),
on which the objective is well defined and convex.
This does not change the primal problem, since
\(y\geq\bm{0}\) and \(\sum_{i=1}^{p}y_{i}=1\) already imply \(y_{i}\leq1<\nicefrac{1}{\beta}\).
To solve the dual problem \(\maximize_{\tau}\psi(\tau)\), we consider the following two cases.
\begin{itemize}
	\item If \(\mathcal{I}_{\leq0}(z)\neq\emptyset\), the first minimization problem in~\eqref{dual_function} implies that \(\tau\leq0\); otherwise \(\psi(\tau)=-\infty\).
	Thus, \(y_i(\tau)=0\) is a minimizer for \(i\in\mathcal{I}_{\leq0}(z)\).
	Considering the second minimization problem in~\eqref{dual_function}, the derivative with respect to \(y_i\) is given by
	\begin{equation}\label{partial_grad}
	\tfrac{\partial}{\partial y_i}
	\begin{pmatrix}
		\tfrac{y_i}{2(1-\beta y_i)}z_i^2-\tau y_i
	\end{pmatrix}
	=
	\tfrac{z_i^2}{2{(1-\beta y_i)}^2}-\tau,
	\end{equation}
	which is strictly positive.
	This implies that \(y_i(\tau)=0\) is a minimizer for \(i\in\mathcal{I}_{>0}(z)\).
	Consequently, \(y(\tau)=0\) is a minimizer in~\eqref{dual_function}, and the dual problem admits \(\max_{\tau\in\R_-}\psi(\tau)=\max_{\tau\in\R_-}\tau=0\).
	It is straightforward to verify that the minimization problem in~\eqref{min_st_clconv} is convex, and that \emph{strong duality} holds.
	The Lasry--Lions double envelope is then given by
	\begin{equation}\label{LL_DO_no_positive}
	\env_{\lambda,\mu}\delta_{D}(z)
	=
	\tfrac{1}{\lambda}
	\sum_{i\in\mathcal{I}_{\leq0}(z)}
	\tfrac{1}{2(1-\beta)}z_i^2,
	\end{equation}
	which corresponds to the first case in~\eqref{rb_2_cases}.
	\item If \(\mathcal{I}_{\leq0}(z)=\emptyset\), the Lagrangian in~\eqref{Lagrangian_negative_positive} and the associated dual function in~\eqref{dual_function} follow
	\begin{equation}\label{Lagrangian_dual}
	\mathcal{L}(y,\tau)
	=
	\tau+
	\sum_{i=1}^{p}
	\begin{pmatrix}
		\tfrac{y_i}{2(1-\beta y_i)}z_i^2-\tau y_i
	\end{pmatrix}
	\quad\text{and}\quad
	\psi(\tau)=
	\tau+
	\sum_{i=1}^{p}
	\inf_{y_i\in[0,\nicefrac{1}{\beta})}
	\begin{Bmatrix}
		\tfrac{y_i}{2(1-\beta y_i)}z_i^2-\tau y_i
	\end{Bmatrix}.
	\end{equation}
	The minimizers of the inner minimization problems in~\eqref{Lagrangian_dual} are given by
	\[
	y_i(\tau)=
	\begin{cases}
		\tfrac{1}{\beta}-\tfrac{z_i}{\beta\sqrt{2\tau}},\ \ &\mbox{if } \tau\geq\tfrac{z_i^2}{2}\\[5pt]
		0,\ \ &\mbox{if } \tau<\tfrac{z_i^2}{2}.
	\end{cases}
	\]
	In this way, we obtain the function \(\psi_i(\tau)\) and its gradient with respect to \(\tau\) as
	\begin{equation}\label{Lagrangian_dual_2}
	\psi_i(\tau)\coloneqq
	\inf_{y_i\in[0,\nicefrac{1}{\beta})}
	\begin{Bmatrix}
		\tfrac{y_i}{2(1-\beta y_i)}z_i^2-\tau y_i
	\end{Bmatrix}
	=
	\begin{cases}
		-\tfrac{{(\sqrt{2\tau}-z_i)}^2}{2\beta},\ \ &\mbox{if } \tau\geq\tfrac{z_i^2}{2}
		\\[5pt]
		0,\ \ &\mbox{if } \tau<\tfrac{z_i^2}{2}
	\end{cases}
	\ \ \text{and}\ \
	\nabla\psi_i(\tau)
	=
	\begin{cases}
		-\tfrac{\sqrt{2\tau}-z_i}{\beta\sqrt{2\tau}},\ \ &\mbox{if } \tau\geq\tfrac{z_i^2}{2}\\[5pt]
		0,\ \ &\mbox{if } \tau<\tfrac{z_i^2}{2}.
	\end{cases}
	\end{equation}
	By~\eqref{Lagrangian_dual_2}, for every \(\tau\geq\frac{z_{(1)}^2}{2}\) there exists a unique integer \(k(\tau)\in\{1,\ldots,p\}\) such that \(\tfrac{z_{(k(\tau))}^2}{2}\leq\tau<\tfrac{z_{(k(\tau)+1)}^2}{2}\) (with \(z_{(p+1)}=\infty\)).
	The uniqueness follows from the nondecreasing ordering of \(z\), since the intervals
	\([ \frac{z_{(i)}^2}{2}, \frac{z_{(i+1)}^2}{2} )\) for \(i\in\{1,\ldots,p\}\) are pairwise disjoint (with some possibly empty when \(z_{(i)}=z_{(i+1)})\).
	To solve the dual problem \(\maximize_{\tau}\psi(\tau)\),
	we identify such unique integer \(k(\tau)\).
	The gradient of \(\psi(\tau)\) in~\eqref{Lagrangian_dual_2} can then be expressed as
	\[
	\nabla\psi(\tau)
	=
	1+\sum_{i=1}^{p}\nabla\psi_i(\tau)
	=
	1+\sum_{i=1}^{k(\tau)}
	\begin{pmatrix}
		-\tfrac{\sqrt{2\tau}-z_{(i)}}{\beta\sqrt{2\tau}}
	\end{pmatrix}.
	\]
	Let \(k_{\beta}(z)\in \{1,\ldots,p\}\) be such that
	\begin{equation}\label{tau}
	z_{(k_{\beta}(z))}\leq
	\frac{\sum_{i=1}^{k_{\beta}(z)}z_{(i)} }{k_{\beta}(z)-\beta}
	<z_{(k_{\beta}(z)+1)}
	\quad\text{and}\quad
	\tau^\star(z) =\tfrac{1}{2}
	{\left(
		\frac{\sum_{i=1}^{k_{\beta}(z)}z_{(i)}}{k_{\beta}(z)-\beta}
	\right)}^2,
	\end{equation}
	we then have \(\nabla\psi(\tau^{\star}(z))=0\).
	The Lasry--Lions double envelope, in the case when \(\mathcal{I}_{\leq0}(z)=\emptyset\), becomes
	\begin{align}
	\env_{\lambda,\mu}\delta_{D}(z)
	&\overset{\eqref{min_st_clconv}}{=}
	\tfrac{1}{\lambda}
	\min\limits_{y\in\conv\Delta_p}
	\begin{Bmatrix}
		\sum_{i=1}^{p}
		\tfrac{y_i}{2(1-\beta y_i)}z_i^2
	\end{Bmatrix}
	\nonumber\\
	&=
	\tfrac{1}{\lambda}\max_{\tau}\psi(\tau)
	\nonumber\\
	\text{\scriptsize by~\eqref{Lagrangian_dual} and~\eqref{Lagrangian_dual_2}}
	&=
	\tfrac{1}{\lambda}
	\begin{pmatrix}
		\tau^\star(z)+\sum_{i=1}^{k_{\beta}(z)}-\tfrac{{(\sqrt{2\tau^\star(z)}-z_{(i)})}^2}{2\beta}
	\end{pmatrix}
	\nonumber\\
	\text{\scriptsize by~\eqref{tau}}
	&=
	\tfrac{1}{\lambda}
	\begin{bmatrix}
		\tfrac{1}{2\beta(k_{\beta}(z)-\beta)}
		\begin{pmatrix}
		\sum_{i=1}^{k_{\beta}(z)}z_{(i)}
		\end{pmatrix}^2
		-
		\tfrac{1}{2\beta}
		\sum_{i=1}^{k_{\beta}(z)}z_{(i)}^2
	\end{bmatrix},\nonumber
	\end{align}
	which corresponds to the second case in~\eqref{rb_2_cases}.
\qedhere
\end{itemize}
\end{proof}

\begin{proposition}[Gradient and proximal operators]\label{proposition_grad_prox}
Relative to the vertical complementarity set \(D\), let \(z\in\R^p\), \(\lambda>\mu>0\) and \(\beta=\nicefrac{\mu}{\lambda}\in(0,1)\).
The gradient of the Lasry--Lions double envelope and the proximal operator \(P_{\lambda,\mu}\delta_D\) in~\eqref{sup_LL_definition} are given by
\begin{equation}\label{LL_DO_grad}
	\nabla\env_{\lambda,\mu}\delta_D(z)=\tfrac{1}{\lambda}R_{\beta}(z)
	\quad\text{and}\quad
	P_{\lambda,\mu}\delta_D(z)=P_{\beta}(z),
\end{equation}
respectively, with
\begin{equation}\label{Rb_Pb_I}
	R_{\beta}(z)=
	\frac{1}{1-\beta}{[z]}_{-}
	\quad\text{and}\quad
	P_{\beta}(z)=
	z+
\frac{\beta}{1-\beta}{[z]}_{-},
\end{equation}
when \(\mathcal{I}_{\leq0}(z)\neq\emptyset\).
On the other hand, if \(\mathcal{I}_{\leq0}(z)=\emptyset\), we have
\begin{equation}\label{Pb_Mb}
	R_{\beta}(z)=
	\frac{1}{\beta}
	{\left[
	\frac{\sum_{i=1}^{k_{\beta}(z)}z_{(i)}}{k_{\beta}(z)-\beta}
		\bm{1}
		-
		z
	\right]}_{+}
	\quad\text{and}\quad
	P_{\beta}(z)=
	z+
	{\left[
	\frac{\sum_{i=1}^{k_{\beta}(z)}z_{(i)}}{k_{\beta}(z)-\beta}
		\bm{1}
		-
		z
	\right]}_{+},
\end{equation}
where \({[\cdot]}_{+}\) and \({[\cdot]}_{-}\) are applied componentwise.
\end{proposition}

\begin{proof}
The result follows directly from \cref{proposition_P_grad} and \cref{theorem_LL_env}.
\end{proof}

\begin{proposition}[Some properties of \(R_\beta\)]\label{proposition_R}
Relative to the operator \(R_\beta\) in~\eqref{LL_DO_grad}, let \(z\in\R^p\), \(\lambda>\mu>0\) and \(\beta=\nicefrac{\mu}{\lambda}\in(0,1)\). The following hold:
\begin{enumerate}
	\item \(R_{\beta}\) is globally Lipschitz continuous with modulus \(L_{\beta}=\max\{\tfrac{1}{\beta},\tfrac{1}{1-\beta}\}\).\label{proposition_Rb_smooth}
	\item\label{proposition_connect_Rb_Pi_2} Let \(z_{(1)}\leq0\).
	Then, we have
	\begin{equation}\label{connect_Rb_Pi_2}
	R_{\beta}(z)
	=
	\frac{1}{1-\beta}
	\left(
		z-\Pi_D(z)
	\right).
	\end{equation}
	\item\label{proposition_connect_Rb_Pi} Let \(z_{(1)}>0\) and define
	\begin{equation}\label{Ib}
	\mathcal{I}_{\beta}(z)\coloneqq\set{i\in\{1,\ldots,p\}}[z_i\leq z_{(k_{\beta}(z))}].
	\end{equation}
	Then, we have
	\begin{equation}\label{connect_Rb_Pi}
	R_{\beta}(z)
	=
	Q_{\beta}(z)
	\left(
		z-\Pi_D^{\rm C,\beta}(z)
	\right),
	\end{equation}
	where the matrix \(Q_{\beta}(z)\in\R^{p\times p}\) and the operator \(\Pi_D^{\rm C,\beta}(z)\) are given by
	\begin{equation}\label{Qb}
	Q_{\beta}(z)=
	\frac{1}{\beta(k_{\beta}(z)-\beta)}
	\left(
		\sum_{i\in\mathcal{I}_{\beta}(z)}\hspace*{-0.5em}e_i
	\right)
	{\left(
		\sum_{i\in\mathcal{I}_{\beta}(z)}\hspace*{-0.5em}e_i
	\right)}^\top
	-
	\frac{1}{\beta}
	\sum_{i\in\mathcal{I}_{\beta}(z)}\hspace*{-0.5em}e_{i}e_i^\top
	\ \ \text{and}\ \
	\Pi_D^{\rm C,\beta}(z)=
	\left(
		I-\sum_{i\in\mathcal{I}_{\beta}(z)}\hspace*{-0.5em}e_{i}e_i^\top
	\right)z.
	\end{equation}
	When \(|\mathcal{I}_{\beta}(z)|=1\), \(R_{\beta}\) in~\eqref{connect_Rb_Pi} reduces to~\eqref{connect_Rb_Pi_2}.
	\item\label{proposition_cone_Rb} \(R_{\beta}\) admits inclusions involving cones:
	\begin{equation}
		R_{\beta}(z)\in
		\begin{cases}
		\Cncone_D(\Pi_D^{\rm C,\beta}(z)),\ \ \ &\mbox{if } k_{\beta}(z)>1
		\\[5pt]
		\Rncone_D(\Pi_D(z)),\ \ \ &\mbox{otherwise}.
		\end{cases}
	\end{equation}
	\item\label{proposition_Rb_homogeneity} \(R_{\beta}\) is positively homogeneous: \(R_{\beta}(az)=aR_{\beta}(z)\) for any \(a>0\).
\end{enumerate}
\end{proposition}

\begin{proof}
\cref{proposition_Rb_smooth} follows from \cref{proposition_Lipschitz}, and
\cref{proposition_connect_Rb_Pi_2} follows from~\eqref{Rb_Pb_I}.
To verify \cref{proposition_connect_Rb_Pi}, let \(\Omega_{\beta}(z)= \sum_{i\in\mathcal{I}_{\beta}(z)}e_{i}e_i^\top\) and \(\omega_{\beta}(z)=\sum_{i\in\mathcal{I}_{\beta}(z)}e_{i}\).
Then, \cref{Qb} can be written as
\[
	Q_{\beta}(z)=
	\frac{1}{\beta(k_{\beta}(z)-\beta)}
	\omega_{\beta}(z)
	{\omega_{\beta}(z)}^\top
	-
	\frac{1}{\beta}
	\Omega_{\beta}(z)
	\quad\text{and}\quad
	\Pi_D^{\rm C,\beta}(z)=
	\left(
	I-\Omega_{\beta}(z)
	\right)z,
\]
and \(R_{\beta}(z)\) becomes
\[
\begin{aligned}
	R_{\beta}(z)
	&\overset{\eqref{Pb_Mb}}{=}
	\frac{1}{\beta}
	\left(
	\frac{\langle\omega_{\beta}(z), z\rangle }{k_{\beta}(z)-\beta}
		\omega_{\beta}(z)
		-
		\Omega_{\beta}(z)z
	\right)
	\\
	&=
	\frac{1}{\beta}
	\left(
	\frac{1}{k_{\beta}(z)-\beta}
		\omega_{\beta}(z)\omega_{\beta}{(z)}^\top z
		-
		\Omega_{\beta}(z)z
	\right).
\end{aligned}
\]
Note that \(\Omega_{\beta}(z)\) has eigenvalues 0 and 1, with \(\omega_{\beta}(z)\) being an eigenvector corresponding to the eigenvalue 1,
and thus, \(\omega_{\beta}{(z)}^\top\Omega_{\beta}(z)=\omega_{\beta}{(z)}^\top\).
Furthermore, \(\Omega_{\beta}(z)\) is idempotent, i.e., \(\Omega_{\beta}{(z)}^2=\Omega_{\beta}(z)\).
Then, we have
\[
\begin{aligned}
	Q_{\beta}(z)
	\left(
	z-\Pi_D^{\rm C,\beta}(z)
	\right)
	&=
	Q_{\beta}(z)\Omega_{\beta}(z)z
	\\
	&=
	\left(
	\frac{1}{\beta(k_{\beta}(z)-\beta)}
	\omega_{\beta}(z)
	{\omega_{\beta}(z)}^\top\Omega_{\beta}(z)
	-
	\frac{1}{\beta}
	\Omega_{\beta}{(z)}^2
	\right)z
	\\
	&=
	\left(
	\frac{1}{\beta(k_{\beta}(z)-\beta)}
	\omega_{\beta}(z)
	{\omega_{\beta}(z)}^\top
	-
	\frac{1}{\beta}
	\Omega_{\beta}(z)
	\right)z
	=R_{\beta}(z).
\end{aligned}
\]
When \(|\mathcal{I}_{\beta}(z)|=1\), we have \(k_{\beta}(z)=1\),
and the unique element of \(\mathcal{I}_{\beta}(z)\) is \(\arg\min_{i}z_{i}\).
Therefore, the equivalence between~\eqref{connect_Rb_Pi_2} and~\eqref{connect_Rb_Pi} follows immediately.

We next establish \cref{proposition_cone_Rb}.
Suppose that \(k_{\beta}(z)>1\).
According to~\eqref{M_set} and~\eqref{Pb_Mb}, we have \({[R_{\beta}(z)]}_{i}\geq0\) for \(i\in\mathcal{I}_{\beta}(z)\),
and \({[R_{\beta}(z)]}_{i}=0\) for \(i\notin\mathcal{I}_{\beta}(z)\).
Moreover, from \(\Pi_D^{\rm C,\beta}(z)\) in~\eqref{Qb},
it follows that \({[\Pi_D^{\rm C,\beta}(z)]}_{i}=0\) if \(i\in\mathcal{I}_{\beta}(z)\).
Therefore, by~\eqref{Clarke_cone}, we obtain \(R_{\beta}(z)\in \Cncone_D(\Pi_D^{\rm C,\beta}(z))\).
If instead \(k_{\beta}(z)=1\), then by \cref{proposition_connect_Rb_Pi_2,proposition_connect_Rb_Pi},
\(R_{\beta}(z)\) is a proximal normal to \(D\) at \(\Pi_D(z)\),
and therefore \(R_{\beta}(z)\in \Rncone_D(\Pi_D(z))\) by~\cite[Example 6.16]{rockafellar1998}.

For \cref{proposition_Rb_homogeneity},
let \(a>0\).
Since \(D\) is a cone, the change of variables \(z'=a\tilde z\) yields
\[
	\dist(az,D)=\inf_{z'\in D}\|z'-az\|
	=\inf_{\tilde z\in D}\|a\tilde z-az\|
	=a\dist(z,D),
\]
and consequently,
\(\dist{(az,D)}^2=a^2\dist{(z,D)}^2\).
Based on this, we obtain
\[
\begin{aligned}
	\env_{\lambda,\mu}\delta_{D}(az)
	&=
	\sup_{z'\in\R^p}\{\tfrac{1}{2\lambda}\dist{(z',D)}^2-\tfrac{1}{2\mu}\|z'-az\|^2 \}
	\\
	&=
	\sup_{\tilde z\in\R^p}\{\tfrac{1}{2\lambda}\dist{(a\tilde z,D)}^2-\tfrac{1}{2\mu}\|a\tilde z-az\|^2 \}
	\\
	&=
	a^2\sup_{\tilde z\in\R^p}\{\tfrac{1}{2\lambda}\dist{(\tilde z,D)}^2-\tfrac{1}{2\mu}\|\tilde z-z\|^2 \}
	=a^2\env_{\lambda,\mu}\delta_{D}(z).
\end{aligned}
\]
Taking the gradient of the above equality with respect to \(z\) yields the positive homogeneity of \(R_{\beta}\).
\end{proof}

By~\eqref{projection_D}, the projection mapping \(\Pi_D\) is set-valued.
For convenience, when \(k_{\beta}(z)=1\), we use \(\Pi_D(z)\) to denote the selection \(\Pi_{D_i}(z)\), where \(i=\min\arg\min_{\ell} {[z]}_{\ell}\).

\subsection{Geometric Illustrations}

In this part, we illustrate the vertical complementarity set \(D\),
its projection,
and the Moreau and Lasry--Lions double envelopes.
For \(p=2\), the complementarity set is given by
\[
	D=D_1 \cup D_2
	\quad \text{with}\quad
	D_1=\{0\}\times\R_+\quad \text{and}\quad
	D_2=\R_+\times\{0\}.
\]
The projection onto \(D\) coincides with the projection onto a convex set, except in the case \(z_1=z_2>0\), where it becomes multi-valued, as illustrated in the left subfigure of \cref{fig_CC_proj_Moreau_LL}.
This multi-valuedness renders the Moreau envelope nonsmooth,
which is given by
\(
	\tfrac{1}{2\lambda}\min\{\dist{(z,D_1)}^2,\dist{(z,D_2)}^2\}.
\)
For \(\lambda=1\), the Moreau envelope is shown in the middle subfigure of \cref{fig_CC_proj_Moreau_LL}, where it is nonsmooth along the ray \(z_1=z_2>0\).
In contrast, for \(\lambda=1\) and \(\mu=0.5\), the Lasry--Lions double envelope is smooth everywhere,
as shown in the right subfigure of \cref{fig_CC_proj_Moreau_LL}.

\begin{figure}[H]
	\centering
	\begin{subfigure}{0.3\textwidth}
		\centering
		\includegraphics[width=0.8\linewidth]{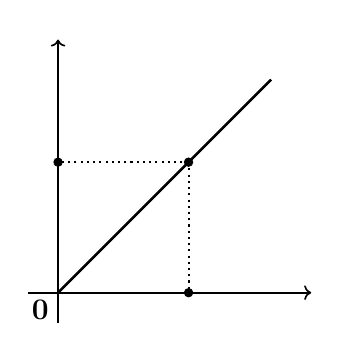}
	\end{subfigure}
	\hfill
	\begin{subfigure}{0.3\textwidth}
		\centering
		\includegraphics[width=0.9\linewidth]{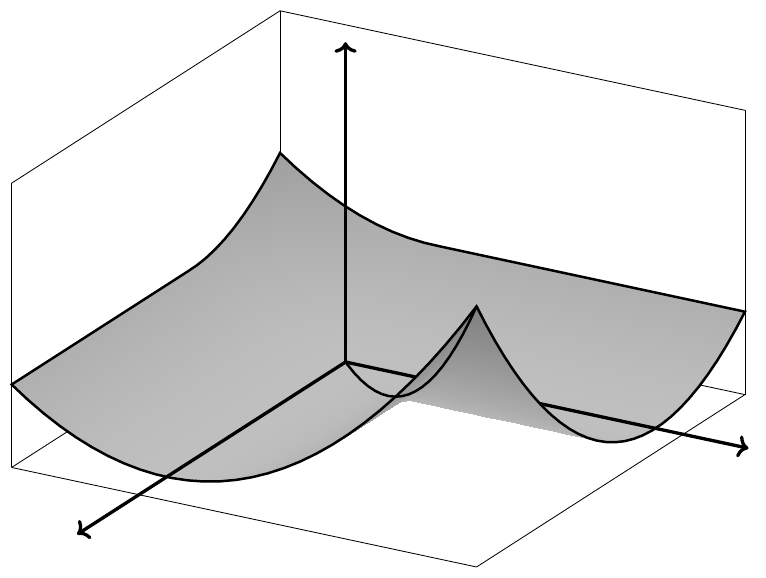}
	\end{subfigure}
	\hfill
	\begin{subfigure}{0.3\textwidth}
		\centering
		\includegraphics[width=0.9\linewidth]{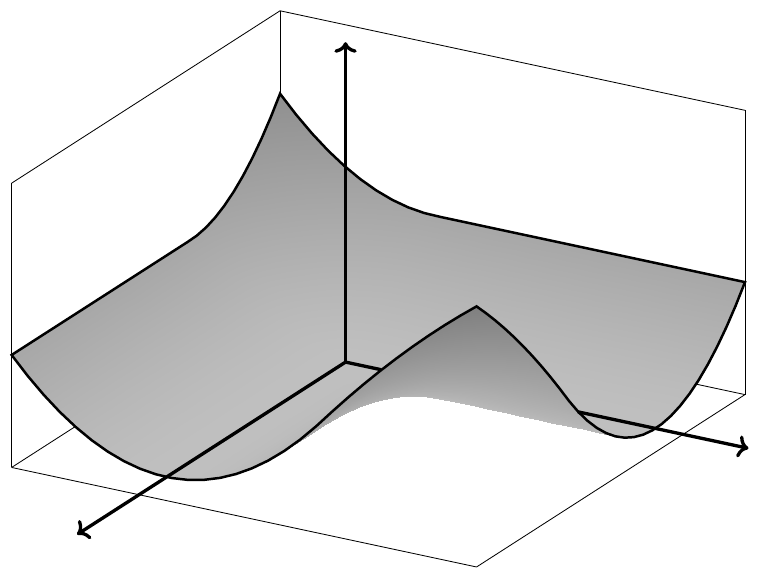}
	\end{subfigure}
	\caption{Geometric illustrations of the projection onto \(D\), the Moreau and Lasry--Lions double envelopes.}\label{fig_CC_proj_Moreau_LL}
\end{figure}

For \(p=3\), the vertical complementarity set is given by
\[
	D=D_1 \cup D_2\cup D_3
	\quad \text{with}\quad
	D_1=\{0\}\times\R_+^2,
	\quad
	D_2=\R_+\times\{0\}\times\R_+
	\quad \text{and}\quad
	D_3=\R_+^2\times\{0\}.
\]
The projection onto \(D\) is two-valued at points \(z\) satisfying \(z_1>z_2=z_3>0\), \(z_2>z_1=z_3>0\), or \(z_3>z_1=z_2>0\), corresponding to the three light gray triangular planes in the left subfigure of \cref{fig_VCC_proj_LL}.
Moreover, \(\Pi_D(z)\) is three-valued at points satisfying \(z_1=z_2=z_3>0\).
In all other cases, the projection is single-valued, as in the convex setting.
Recall that the multi-valuedness of the projection onto \(D\) leads to the nonsmoothness of the Moreau envelope,
and it can be smoothed by the Lasry--Lions double envelope.
For \(\lambda=1\) and \(\mu=0.5\), the smoothing regions induced by the Lasry--Lions double envelope are shown in the right subfigure of \cref{fig_VCC_proj_LL}.
The three light gray polyhedral regions correspond to points \(z\) satisfying \(k_{\beta}(z)=2\).
These regions are interpreted as smoothings of the three light gray triangular planes in the left subfigure.
Moreover, the dark gray polyhedral region corresponds to points \(z\) satisfying \(k_{\beta}(z)=3\), and smooths the ray \(z_1=z_2=z_3>0\).

\begin{figure}[H]
	\centering
	\begin{subfigure}{0.3\textwidth}
		\centering
		\includegraphics[width=0.9\linewidth]{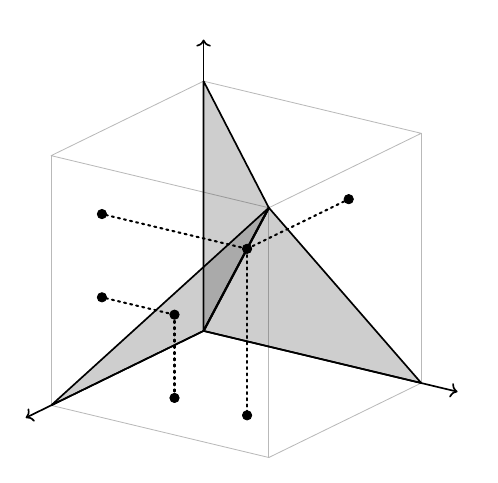}
	\end{subfigure}
	\hspace{0.1\textwidth}
	\begin{subfigure}{0.3\textwidth}
		\centering
		\includegraphics[width=0.9\linewidth]{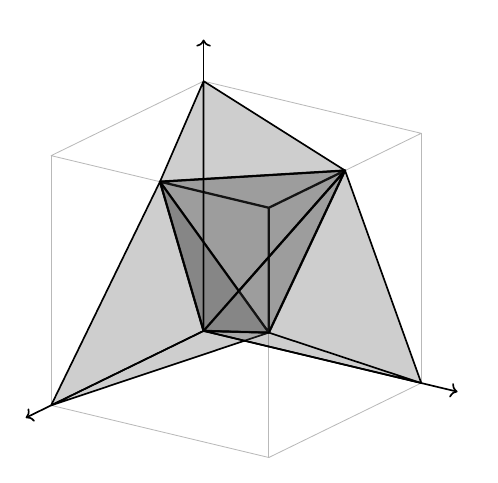}
	\end{subfigure}
	\caption{Geometric illustrations of the projection onto \(D\) and the smoothing regions.}\label{fig_VCC_proj_LL}
\end{figure}

The ratio \(\beta=\nicefrac{\mu}{\lambda}\) controls the size of the smoothing regions.
As shown in the left subfigure of \cref{fig_beta}, when \(\beta\) approaches one, the smoothing region expands and eventually covers the entire positive orthant.
In contrast, as \(\beta\) approaches zero, the smoothing region shrinks and eventually reduces to the ray \(z_1=z_2>0\);
that is, the Lasry--Lions double envelope approaches the Moreau envelope (see \cref{proposition_sandwich}).
A similar behavior is observed in the case \(p=3\), as shown in the middle and right subfigures of \cref{fig_beta}.
When \(\beta\) approaches one, they expand to cover the entire region \(\R_{+}^3\);
when \(\beta\) approaches zero, the smoothing regions shrink to the three light gray triangular planes.

\begin{figure}[h]
	\centering
	\begin{subfigure}{0.3\textwidth}
		\centering
		\includegraphics[width=\linewidth]{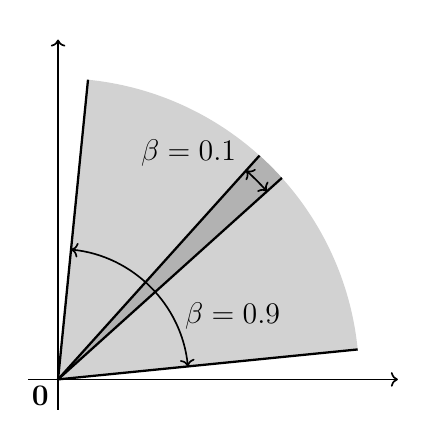}
	\end{subfigure}
	\hfill
	\begin{subfigure}{0.3\textwidth}
		\centering
		\includegraphics[width=\linewidth]{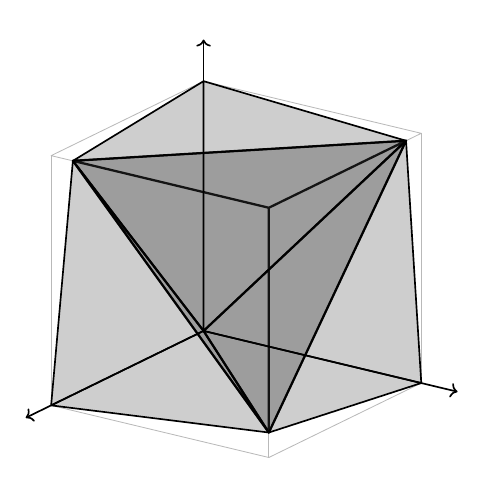}
	\end{subfigure}
	\hfill
	\begin{subfigure}{0.3\textwidth}
		\centering
		\includegraphics[width=\linewidth]{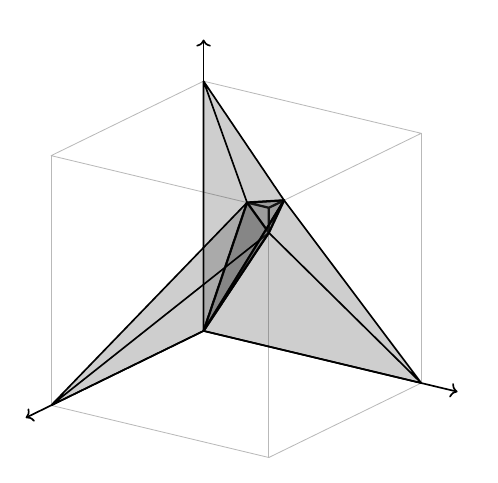}
	\end{subfigure}
	\caption{Geometric illustrations showing how the parameter \(\beta\) affects the smoothing regions.}\label{fig_beta}
\end{figure}

	\section{A Smoothed Augmented Lagrangian Method}
			In this section, we present a smoothed ALM for solving \MPVCC{},
	in which the nonsmooth subproblem in the standard ALM framework is replaced by a smooth surrogate constructed via the Lasry--Lions double envelope.
	Let
	\(\s=(s^0,s^1,\ldots,s^r)\in\R^{m+\bar p}\),
	where \(s^0\in\R^m\) and \(s^j\in\R^{p_j}\).
	Define
	\(\delta_{\Gamma}(\s)\coloneqq\delta_{C}(s^0)+\sum_{j=1}^{r}\delta_{D^j}(s^j)\).
	This yields the following reformulation of \MPVCC{}:
	\begin{equation}\label{MPVCC_reformulated}
		\begin{split}
		\minimize_{x,\s}
		\ \
		& f(x)+\delta_{\Gamma}(\s)
		\\
		\stt\ \
		& \Phi(x)=\s.
		\end{split}
	\end{equation}
	We define the \emph{Lagrangian}
	\(\mathcal{L}(x,\s,\y):\R^n\times\R^{m+\bar p}\times\R^{m+\bar p}\rightarrow\overline\R\) by
	\[
		\mathcal{L}(x,\s,\y)\coloneqq
		f(x)+
		\delta_{\Gamma}(\s)
		+
		\langle \y,\Phi(x)-\s\rangle.
	\]
	Let \(\lambda\in\R_{++}\), the \emph{augmented Lagrangian} (AL)
	\(\mathcal{L}_{\lambda}:\R^n\times\R^{m+\bar p}\times\R^{m+\bar p}\rightarrow\overline\R\) is defined as
	\begin{align}
		&
		\mathcal{L}_{\lambda}(x,\s,\y)\coloneqq
		\mathcal{L}(x,\s,\y)
		+\tfrac{1}{2\lambda}\|\Phi(x)-\s\|^2
		\nonumber\\
		&
		=f(x)
		+\delta_C(s^0)+\tfrac{1}{2\lambda}\|A(x)-s^0+\lambda y^0\|^2-\tfrac{\lambda}{2}\|y^0\|^2
		+
		\sum_{j=1}^{r}
		\begin{pmatrix}
		\delta_{D^j}(s^j)+\tfrac{1}{2\lambda}\|F^j(x)-s^j+\lambda y^j\|^2-\tfrac{\lambda}{2}\|y^j\|^2
		\end{pmatrix}.\label{AL_function_with_indicator}
	\end{align}
	The slack variable \(\s\) can be eliminated by minimizing \(\mathcal{L}_{\lambda}(x,\s,\y)\) with respect to \(\s\).
	This yields the \emph{implicit} AL function \(\mathcal{L}_{\lambda}:\R^n\times\R^{m+\bar p}\rightarrow\overline\R\), defined by
	\allowdisplaybreaks\@
	\begin{align}
	\mathcal{L}_{\lambda}(x,\y)
	&\coloneqq
	\inf_{\s\in\R^{m+\bar p}}\mathcal{L}_{\lambda}(x,\s,\y)
	\nonumber\\
	&=f(x)
	+\env_{\lambda}\delta_C(A(x)+\lambda y^0)-\tfrac{\lambda}{2}\|y^0\|^2
	+\sum_{j=1}^{r}
	\left(
	\env_{\lambda}\delta_{D^j}(F^j(x)+\lambda y^j)-\tfrac{\lambda}{2}\|y^j\|^2
	\right).\label{implicit_AL}
	\end{align}
	Let \(\be\coloneqq(\sit{\beta}[1],\ldots,\sit{\beta}[r])\in\R^r\),
	where \(\sit{\beta}[j]\in(0,1)\).
	The \emph{smoothed} implicit AL function is defined as
	\begin{equation}\label{smoothened_implicit_AL}
		\mathcal{L}_{\lambda,\be}(x,\y)
		\coloneqq
		f(x)
		+
		\env_{\lambda}\delta_C(A(x)+\lambda y^0)-\tfrac{\lambda}{2}\|y^0\|^2
		+
		\sum_{j=1}^{r}
		\left(
		\tfrac{1}{\lambda}
		r_{\sit{\beta}}(F^j(x)+\lambda y^j)
		-
		\tfrac{\lambda}{2}\|y^j\|^2
		\right).
	\end{equation}

	\begin{algorithm}[p]
	\caption{Smoothed ALM for \MPVCC{}.}
	\begin{algorithmic}[1]\label{algorithm_ALM}
	\itemsep=3pt%
	\renewcommand{\algorithmicindent}{0.3cm}%

	\item[\algfont{Data:}]

	Algorithm parameters \(\rho\), \(\theta\),
	\(\kappa\) and \(\beta_{\min}\) all belong to \((0,1)\).
	Initial penalty and smoothing parameters \(\lambda_0\in\R_{++}\) and \(\be_0\in{[\beta_{\min},1)}^r\).
	Initial point \(x^{-1}\in\R^{n}\) and multiplier \(\y^{-1}\in\R^{m+\bar p}\).
	Nonempty, convex, and compact set \(Y\subset\R^{m+\bar p}\).

	\For{\(\nu=0,1,2,\ldots\)}

	\If{\(\nu>0\) and a termination criterion (see \cref{sec_implementation_salm}) is satisfied}\label{alg_ALM_stop}
	\State\@ return \(x^{\nu-1}\)
	\EndIf\@

	\State\@\label{alg_ALM_safeguarded} Compute the safeguarded multiplier:
	\begin{equation}
	\hat \y^\nu=\Pi_Y(\y^{\nu-1})
	\end{equation}

	\State\@\label{alg_ALM_stationary} Choose \(\sit{\epsilon}[\rm{d}][\nu]\geq0\) and compute an \(\sit{\epsilon}[\rm{d}][\nu]\)-stationary point \(x^\nu\) to the problem:
	\begin{equation}
	\minimize_{x\in\R^n}\ \mathcal{L}_{\lambda_\nu,\be_\nu}(x,\hat \y^\nu),
	\end{equation}
	i.e.,
	\(
	\|
		\nabla_x
		\mathcal{L}_{\lambda_\nu,\be_\nu}(x^\nu,\hat\y^{\nu})
	\|
	\leq
	\sit{\epsilon}[\rm{d}][\nu]
	\),
	using \(x^{\nu-1}\) as the starting point,
	where \(\mathcal{L}_{\lambda_\nu,\be_\nu}\) is defined in~\eqref{smoothened_implicit_AL}

	\State\@ Compute the feasible approximation \(\z^\nu\):
	\begin{align}
	&
	\vit{z}[0][\nu]\leftarrow
	\Pi_C
	\left(
		A(x^\nu)+\lambda_\nu \vit{\hat y}[0][\nu]
	\right)\label{update_v_C}
	\\[5pt]
	&
	\vit{z}[j][\nu]\leftarrow
	\begin{cases}
		\Pi_{D^j}^{\mathrm{C},\sit{\beta}[j][\nu]}
		\left(
		F^j(x^\nu)+\lambda_\nu \vit{\hat y}[j][\nu]
		\right),
		&\quad\mbox{if }
		k_{\sit{\beta}[j][\nu]}
		\left(
		F^j(x^\nu)+\lambda_\nu \vit{\hat y}[j][\nu]
		\right)>1
		\\[5pt]
		\Pi_{D^j}
		\left(
		F^j(x^\nu)+\lambda_\nu \vit{\hat y}[j][\nu]
		\right),
		&\quad\mbox{otherwise,}
	\end{cases}
	\label{update_v_j}
	\end{align}
	for \(j\in\set{1,\ldots,r}\),
	where \(k_{\sit{\beta}[j][\nu]}\) and
	\(\Pi_{D^j}^{\mathrm{C},\sit{\beta}[j][\nu]}\) are defined in~\eqref{M_set} and~\eqref{Qb}, respectively

	\State\@ Compute the infeasibility measure:
	\begin{equation}\label{update_residual}
	V_\nu\leftarrow
	\begin{Vmatrix}
		\Phi(x^\nu)-\z^\nu
	\end{Vmatrix},
	\end{equation}
	where \(\Phi\) is defined in~\eqref{define_constraint}

	\State\@ Update the multiplier:
	\begin{align}
	\y^\nu\leftarrow \hat \y^\nu+\frac{1}{\lambda_\nu}(\Phi(x^\nu)-\z^\nu)
	\label{multiplier_update}
	\end{align}

	\State\@ Update the penalty parameter:
	\begin{align}
	\lambda_{\nu+1}
	\leftarrow
	\begin{cases}
		\rho\lambda_\nu,
		&\quad\mbox{if }
		\nu>0\ \ \text{and}\ \ V_\nu> \theta V_{\nu-1}
		\\[5pt]
		\lambda_\nu,
		&\quad\mbox{otherwise}
	\end{cases}
	\label{alg_ALM_update_lambda}
	\end{align}

	\State\@ Update the smoothing parameter:
	\begin{align}
	\sit{\beta}[j][\nu+1]
	\leftarrow
	\begin{cases}
		\max\{\kappa \sit{\beta}[j][\nu],\beta_{\min}\},
		&\quad\mbox{if }
		k_{\sit{\beta}[j][\nu]}
		\left(
		F^j(x^\nu)+\lambda_\nu\vit{\hat y}[j][\nu]
		\right)>1
		\\[5pt]
		\sit{\beta}[j][\nu],
		&\quad\mbox{otherwise,}
	\end{cases}
	\label{alg_ALM_reduce_beta}
	\end{align}
	for \(j\in\set{1,\ldots,r}\)

	\EndFor\@
	\end{algorithmic}
	\end{algorithm}

	By \cref{proposition_sandwich}, function~\eqref{smoothened_implicit_AL} provides an upper approximation of~\eqref{implicit_AL};
	consequently, employing these subproblems within the ALM framework generally yields iterates that differ from those of the (nonsmooth) standard ALM.
	Nevertheless, by~\cite[Example 1.46 and Proposition 7.4(c)]{rockafellar1998}, we obtain both pointwise and epi-convergence:
	\(
		\mathcal{L}_{\lambda,\be}(x,\y)
		\searrow
		\mathcal{L}_{\lambda}(x,\y)
	\)
	and
	\(
		\mathcal{L}_{\lambda,\be}(\cdot,\y)
		\overset{\mathrm{e}}{\rightarrow}
		\mathcal{L}_{\lambda}(\cdot,\y)
	\)
	as \(\be\to\bm{0}\).
	These results indicate that the approximation accuracy of~\eqref{implicit_AL} by~\eqref{smoothened_implicit_AL} can be improved by decreasing \(\be\).
	We also emphasize that, in our notation, \(\lambda\) is the reciprocal of the penalty parameter:
	a smaller \(\lambda\) corresponds to a larger penalty.

	The proposed method is summarized in \cref{algorithm_ALM}.
	Throughout this section, we assume that \cref{algorithm_ALM} does not stop after a finite number of iterations.
	The termination criterion in \cref{alg_ALM_stop} is discussed in \cref{sec_implementation_salm}.
	\cref{alg_ALM_safeguarded} corresponds to the safeguarded step
	obtained by projecting the multiplier \(\y^{\nu-1}\) onto the nonempty, convex, and compact set \(Y\).
	This safeguarding strategy provides improved global convergence properties;
	see~\cite{birgin2014}.
	In our implementation, the set \(Y\) is chosen as the box
	\({[y_{\min},y_{\max}]}^{m+\bar p}\).
	\cref{alg_ALM_stationary} constitutes the main computational burden of the algorithm,
	since an unconstrained nonlinear optimization problem must be solved approximately at each outer iteration.
	Thanks to the Lasry--Lions double envelope, this subproblem is smooth, thereby enabling the use of many off-the-shelf smooth optimization solvers.
	However, as shown in \cref{proposition_cone_Rb},
	the price paid for this additional smoothness is the possible convergence to a C- but not M-stationary point.
	This behavior can be tracked through the integer
	\(k_{\sit{\beta}[j][\nu]}(F^j(x^\nu)+\lambda_\nu\vit{\hat y}[j][\nu])\).
	Whenever this quantity exceeds one, for instance when
	\(F^j(x^\nu)+\lambda_\nu \vit{\hat y}[j][\nu]\)
	lies in the gray polyhedral regions shown in the right subfigure of \cref{fig_VCC_proj_LL},
	the feasible approximation in~\eqref{update_v_j} employs the operator
	\(\Pi_{D^j}^{\mathrm{C},\sit{\beta}[j][\nu]}\).
	We define the index set
	\begin{equation}\label{Tb_set}
		\tb[\nu]
		\coloneqq
		\set{j\in\{1,\ldots,r\}}
		[k_{\sit{\beta}[j][\nu]}(F^j(x^\nu)+\lambda_\nu\vit{\hat y}[j][\nu])>1].
	\end{equation}
	The set \(\tb[\nu]\) is used in \cref{theorem_AM_AC,corollary_stationarity} to determine whether a cluster point is M- or C-stationary.
	Moreover, if \(j\in\tb[\nu]\), the smoothing parameter \(\sit{\beta}[j][\nu]\) is reduced by~\eqref{alg_ALM_reduce_beta}.
	The purpose of this reduction is to shrink the smoothing regions,
	as illustrated in \cref{fig_beta},
	thereby increasing the likelihood that the sequence
	\( \seq{F^j(x^\nu)+\lambda_\nu\vit{\hat y}[j][\nu]}[\nu\in\N]\)
	eventually leaves these regions.
	To avoid potential numerical issues, each smoothing parameter is bounded below by \(\beta_{\min}\in(0,1)\).
	Hence, we have \(\be_{\nu}=\bar\be\) for \(\nu\) sufficiently large, where
	\begin{equation}\label{define_limit_beta}
		\bar\be\coloneqq(\bar\beta_{1},\ldots,\bar\beta_{r})
		\quad\text{with}\quad
		\bar\beta_j\coloneqq
		\lim_{\nu\to\infty}
		\sit{\beta}[j][\nu]\in[\beta_{\min},\sit{\beta}[j][0]]
		\quad\text{for}\quad
		j\in\set{1,\ldots,r}.
	\end{equation}

	\begin{figure}[b]\centering
	\includegraphics[width=\textwidth]{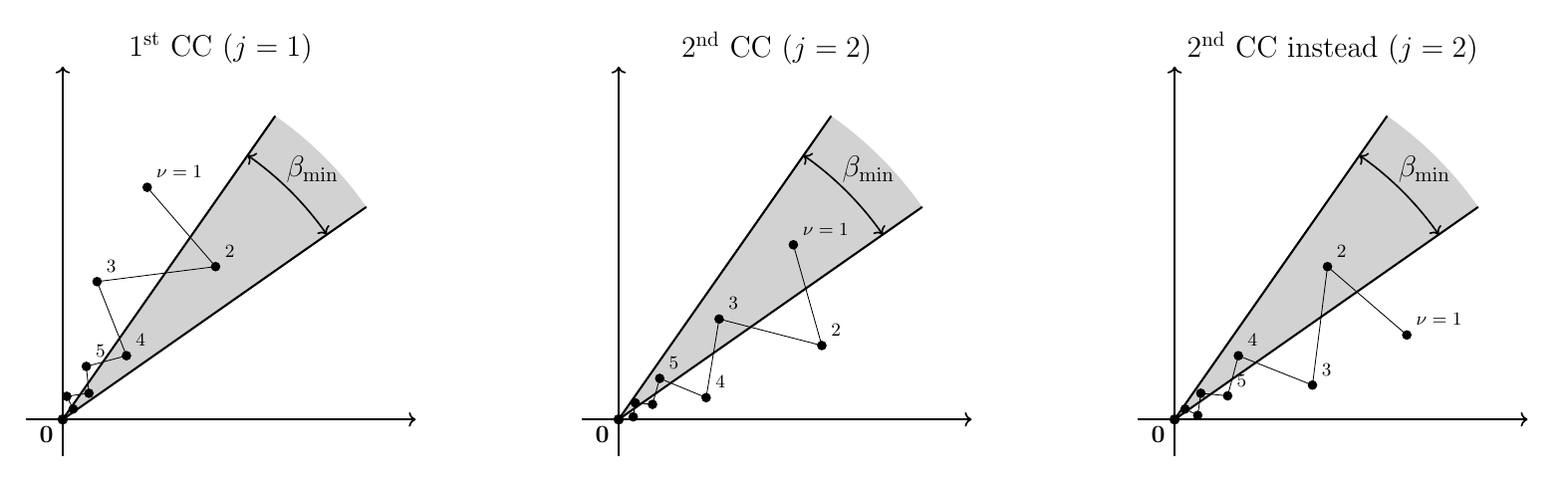}
	\caption{Illustrations of the index set \(\tb[\nu]\).}\label{fig_Tb}
	\end{figure}

	\begin{example}
	We use two CCs to illustrate the index set \(\tb[\nu]\).
	Suppose that the sequences
	\( \seq{F^j(x^\nu)+\lambda_\nu\vit{\hat y}[j][\nu]}[\nu\in\N]\) for \(j=1,2\)
	converge to \(\bm{0}\),
	as shown in the left and middle subfigures of \cref{fig_Tb}.
	Since both sequences admit subsequences satisfying
	\(k_{\sit{\beta}[j][\nu]}(F^j(x^\nu)+\lambda_\nu\vit{\hat y}[j][\nu])>1\),
	the corresponding smoothing parameters \(\sit{\beta}[j][\nu]\) for \(j=1,2\) are reduced to \(\beta_{\min}\) according to~\eqref{alg_ALM_reduce_beta}.
	By~\eqref{Tb_set}, for odd \(\nu\) we have \(\tb[\nu]=\{2\}\),
	whereas for even \(\nu\) we have \(\tb[\nu]=\{1\}\).
	Consequently,  \(\tb[\nu]\neq\emptyset\) at every iteration and, by \cref{theorem_AM_AC}, the cluster point \(\bar x\) is AC-stationary.
	On the other hand, if the CC shown in the right subfigure of \cref{fig_Tb} is used as the second CC, together with the first CC, then
	\(\tb[\nu]=\emptyset\) for odd \(\nu\),
	whereas \(\tb[\nu]=\{1,2\}\) for even \(\nu\).
	Consequently, by \cref{theorem_AM_AC}, the cluster point \(\bar x\) is AM-stationary,
	since there exists a subsequence
	\(K\in \mathcal{N}_{\infty}^{\sharp}\) such that \(x^\nu\to_K \bar x\) and \(\tb[\nu]=\emptyset\) for infinitely many \(\nu\in K\).
	In the present example, this corresponds to the subsequence of odd iterates shown in \cref{fig_Tb}.
	\end{example}

\subsection{Global Convergence Analysis}
	As with all penalty-type methods for nonconvex optimization problems, ALM has the drawback that iterates may have cluster points that are not feasible for the original problem.
	The following result provides conditions under which feasibility of cluster points is guaranteed.

	\begin{proposition}[Feasibility]\label{proposition_feasibility}
	Relative to \MPVCC{} with notation~\eqref{define_constraint}--\eqref{define_cone},
	consider a sequence \(\seq{x^\nu}[\nu\in\N]\) generated by \cref{algorithm_ALM}.
	Then, each cluster point \(\bar x\) of \(\seq{x^\nu}[\nu\in\N]\) is feasible if one of the following conditions holds:
	\begin{itemize}
		\item \(\seq{\lambda_\nu}[\nu\in\N]\) is bounded away from zero, or
		\item there exists some \(B\in\R\) such that \(\mathcal{L}_{\lambda_\nu,\be_\nu}(x^\nu,\hat \y^\nu)\leq B\) holds for all \(\nu\in\N\).
	\end{itemize}
	\end{proposition}

	\begin{proof}
	Let \(K\in \mathcal{N}_{\infty}^{\sharp}\) be a subsequence such that \(x^\nu\to_K \bar x\).
	We consider the two cases as follows.
	\begin{itemize}
		\item
		If \(\seq{\lambda_\nu}[\nu\in\N]\) is bounded away from zero,
		then \(\lambda_\nu\) can be reduced only finitely many times,
		so that \(V_{\nu}\leq\theta V_{\nu-1}\) holds for \(\nu\) sufficiently large by the update rule~\eqref{alg_ALM_update_lambda};
		hence, \(V_\nu\to0\) as \(\nu\to\infty\).
		Note that we have the upper bound
		\begin{equation}\label{dist_upper_bound}
		\dist(\Phi(x^\nu),\Gamma)\leq \|\Phi(x^\nu)-\z^\nu\|.
		\end{equation}
		Therefore, by \(x^\nu\to_K \bar x\) and the continuity of \(\Phi\), we conclude \(\dist(\Phi(\bar x),\Gamma)=0\).
	\item Considering the case \(\lambda_\nu\searrow 0\), the upper bound assumption yields
	\[
	\begin{aligned}
		B\geq\mathcal{L}_{\lambda_\nu,\be_\nu}(x^\nu,\hat\y^\nu)=
		f(x^\nu)
		&+
		\env_{\lambda_\nu}\delta_C(A(x^\nu)+\lambda_\nu \vit{\hat y}[0][\nu])
		-\tfrac{\lambda_\nu}{2}\|\vit{\hat y}[0][\nu]\|^2
		\\
		&+
		\sum_{j=1}^{r}
		\begin{pmatrix}
		\tfrac{1}{\lambda_\nu}r_{\sit{\beta}[j][\nu]}(F^j(x^\nu)+\lambda_\nu \vit{\hat y}[j][\nu])
		-
		\tfrac{\lambda_\nu}{2}\|\vit{\hat y}[j][\nu]\|^2
		\end{pmatrix}.
		\end{aligned}
	\]
	In light of \cref{proposition_sandwich}, we have
	\[
	\begin{aligned}
		\tfrac{1}{2\lambda_\nu}\dist{(\Phi(x^\nu)+\lambda_\nu \hat \y^\nu,\Gamma)}^2
		&\leq
		\tfrac{1}{2\lambda_\nu}\dist{
		(
			A(x^\nu)+\lambda_\nu \vit{\hat y}[0][\nu], C
		)
		}^2
		+
		\sum_{j=1}^{r}
		\tfrac{1}{\lambda_\nu}r_{\sit{\beta}[j][\nu]}
		(F^j(x^\nu)+\lambda_\nu \vit{\hat y}[j][\nu])
		\\
		&\leq
		B-f(x^\nu)+
		\sum_{j=0}^{r}\tfrac{\lambda_\nu}{2}\|\vit{\hat y}[j][\nu]\|^2.
		\end{aligned}
	\]
	Multiplying both sides by \(2\lambda_\nu\), we have
	\[
		\dist{(\Phi(x^\nu)+\lambda_\nu \hat \y^\nu,\Gamma)}^2
		\leq
		2\lambda_\nu
		\begin{pmatrix}
		B-f(x^\nu)+
		\sum_{j=0}^{r}\tfrac{\lambda_\nu}{2}\|\vit{\hat y}[j][\nu]\|^2
		\end{pmatrix}
		\rightarrow_K0.
	\]
	Therefore, by \(x^\nu\to_K \bar x\), the continuity of \(\Phi\), the boundedness of \(\seq{\hat\y^{\nu}}[\nu\in\N]\), and \(\lambda_\nu\searrow 0\),
	we have \(\dist(\Phi(\bar x),\Gamma)=0\).
	\qedhere
	\end{itemize}
	\end{proof}

	The second condition in \cref{proposition_feasibility} holds, for example, when a feasible starting point \(x^{0}\) of \MPVCC{} is available
	and the inner solver is warm-started so that \(\mathcal{L}_{\lambda_\nu,\be_\nu}(x^\nu,\hat \y^\nu)\leq \mathcal{L}_{\lambda_\nu,\be_\nu}(x^0,\hat \y^\nu)\).
	Indeed, by \cref{proposition_sandwich},
	\[
		\tfrac{1}{\lambda_\nu}r_{\sit{\beta}[j][\nu]}(F^j(x^0)+\lambda_\nu \vit{\hat y}[j][\nu])
		\leq
		\env_{(1-\sit{\beta}[j][\nu])\lambda_\nu}\delta_{D^j}(F^{j}(x^0)+\lambda_\nu \vit{\hat y}[j][\nu])
		\leq
		\tfrac{\lambda_\nu}{2(1-\sit{\beta}[j][\nu])}
		\|\vit{\hat y}[j][\nu]\|^2,
	\]
	and an analogous bound holds for the set \(C\).
	Therefore, \(\mathcal{L}_{\lambda_\nu,\be_\nu}(x^0,\hat \y^\nu)\) is bounded above,
	since \(Y\) is bounded, \(\lambda_{\nu}\leq\lambda_{0}\), and \(\sit{\beta}[j][\nu]\leq\max_{j}\sit{\beta}[j][0]<1\).

	Let us consider the following \emph{feasibility problem}:
	\begin{equation}\label{feasibility_problem}
	\minimize_{x\in\R^n}
	\ \
	h_{\be}(x)\coloneqq
	\tfrac{1}{2}\dist{(A(x),C)}^2
	+
	\sum_{j=1}^{r}
	r_{\beta_j}(F^j(x)).
	\end{equation}
	Although parameter-dependent, problem~\eqref{feasibility_problem} serves as an infeasibility measure,
	since any feasible point is a global minimizer;
	otherwise the objective value is strictly positive due to~\eqref{moreau_compl_DO} and \cref{proposition_sandwich}.
	The next proposition implies that when a cluster point \(\bar x\) is not feasible,
	it nevertheless satisfies \(\nabla h_{\bar\be}(\bar x)=\bm{0}\).

	\begin{proposition}[Infeasibility]\label{proposition_infeasibility}
	Suppose that the sequence \(\seq{\sit{\epsilon}[\rm{d}][\nu]}[\nu\in\N]\) in \cref{algorithm_ALM} is chosen bounded, and \(\bar x\) is a cluster point of the sequence \(\seq{x^\nu}[\nu\in\N]\) generated by \cref{algorithm_ALM}.
	Then, \(\bar x\) is a stationary point of the feasibility problem~\eqref{feasibility_problem} with the smoothing parameters \(\bar\be\) given by~\eqref{define_limit_beta}.
	\end{proposition}

	\begin{proof}
	In light of \cref{proposition_feasibility}, if \(\seq{\lambda_\nu}[\nu\in\N]\) is bounded away from zero,
	then each cluster point of \(\seq{x^\nu}[\nu\in\N]\) is a global minimizer of~\eqref{feasibility_problem}.
	Therefore, it is sufficient to consider the case \(\lambda_\nu\searrow 0\).
	In view of \cref{alg_ALM_stationary} of \cref{algorithm_ALM} and~\cite[Corollary 12.30]{bauschke2011}, there exists a vector \(\eta^\nu\in\R^n\) such that \(\|\eta^\nu\|\leq\sit{\epsilon}[\rm{d}][\nu]\) and
	\begin{align}
		\eta^\nu
		&=
		\nabla f(x^\nu)
		+
		\nabla
		\left[
		\env_{\lambda_\nu}\delta_C(A(x^\nu)+\lambda_\nu \vit{\hat y}[0][\nu])
		\right]
		+
		\sum_{j=1}^{r}
		\nabla
		\left[
		\tfrac{1}{\lambda_\nu}r_{\sit{\beta}[j][\nu]}(F^j(x^\nu)+\lambda_\nu \vit{\hat y}[j][\nu])
		\right]
		\nonumber\\
		&
		=\nabla f(x^\nu)
		+
		\tfrac{1}{\lambda_\nu}
		JA{(x^\nu)}^\top
		\left[
		A(x^\nu)+\lambda_\nu \vit{\hat y}[0][\nu]
		-
		\Pi_C(A(x^\nu)+\lambda_\nu \vit{\hat y}[0][\nu])
		\right]
		\nonumber
		\\
		&\quad+
		\tfrac{1}{\lambda_\nu}
		\sum_{j=1}^{r}
		JF^j{(x^\nu)}^\top
		R_{\sit{\beta}[j][\nu]}(F^j(x^\nu)+\lambda_\nu \vit{\hat y}[j][\nu] ).
		\label{eta_1}
	\end{align}
	Let \(K\in \mathcal{N}_{\infty}^{\sharp}\) be such that \(x^\nu\to_K \bar x\).
	Then, the sequences \(\seq{\eta^\nu,\nabla f(x^\nu),\hat \y^\nu}[K]\) are bounded.
	By multiplying both sides of~\eqref{eta_1} by \(\lambda_\nu\),
	letting \(\nu\to\infty\),
	and using the continuity of \(\Pi_C\),
	the eventual constancy of \(\be_{\nu}=\bar\be\) for \(\nu\) sufficiently large,
	and the Lipschitz continuity of \(R_{\bar\beta_j}\) from \cref{proposition_Lipschitz},
	we obtain
	\[
		\bm{0}
		= JA{(\bar x)}^\top
		\left[
			A(\bar x)-\Pi_C(A(\bar x))
		\right]
		+
		\sum_{j=1}^{r}
		JF^j{(\bar x)}^\top
		R_{\bar\beta_j}(F^j(\bar x)).
	\]
	Consequently, it follows that \(\bar x\) satisfies \(\nabla h_{\bar\be}(\bar x)=\bm{0}\).
	\end{proof}

	The asymptotic stationarity of cluster points is established in the next theorem.
	To this end, for each \(\nu\in\N\) and \(j\in\set{1,\ldots,r}\),
	we define the \emph{scaled multiplier} associated with \(\y^{\nu}\) by
	\begin{equation}\label{scaled_y}
		\y_{\beta}^{\nu}
		\coloneqq
		(
		\vit{y}[0][\nu],
		\vit{y_{\beta}}[1][\nu],
		\ldots,
		\vit{y_{\beta}}[r][\nu]
		)
		\in\R^{m+\bar p}
		\quad \text{with}\quad
		\vit{y_{\beta}}[j][\nu]
		\coloneqq
		\begin{cases}
		Q_{\sit{\beta}[j][\nu]}(\vit{u}[j][\nu]) \vit{y}[j][\nu],\ \ \ &\mbox{if } j\in\tb[\nu]
		\\[5pt]
		\tfrac{1}{1-\sit{\beta}[j][\nu]} \vit{y}[j][\nu],\ \ \ &\mbox{otherwise,}
		\end{cases}
	\end{equation}
	where
	\begin{equation}\label{define_u}
		\vit{u}[j][\nu]
		\coloneqq
		F^j(x^\nu)+\lambda_\nu \vit{\hat y}[j][\nu],
	\end{equation}
	the matrix \(Q_{\beta}\) is defined in~\eqref{Qb}, and
	\(\y^{\nu}\) is the multiplier computed in~\eqref{multiplier_update}.

	\begin{theorem}[Asymptotic stationarity]\label{theorem_AM_AC}
	Relative to \MPVCC{} with notation~\eqref{define_constraint}--\eqref{define_cone},
	suppose that the sequence \(\seq{\sit{\epsilon}[\rm{d}][\nu]}[\nu\in\N]\) in \cref{algorithm_ALM} is chosen such that \(\sit{\epsilon}[\rm{d}][\nu]\searrow 0\),
	and that \(\bar x\) is a feasible cluster point of the sequence \(\seq{x^\nu}[\nu\in\N]\) generated by \cref{algorithm_ALM}.
	Then, \(\bar x\) is an AC-stationary point of \MPVCC{}.
	Let \(\tb[\nu]\) be defined in~\eqref{Tb_set}.
	If, in addition, there exists a subsequence
	\(K\in \mathcal{N}_{\infty}^{\sharp}\) such that \(x^\nu\to_K \bar x\) and such that \(\tb[\nu]=\emptyset\) for infinitely many \(\nu\in K\),
	then \(\bar x\) is an AM-stationary point.
	\end{theorem}

	\begin{proof}
	Let \(K\in \mathcal{N}_{\infty}^{\sharp}\) be such that \(x^\nu\to_K \bar x\).
	Following~\eqref{eta_1}, for \(\nu\) sufficiently large, we have
	\begin{subequations}
	\allowdisplaybreaks\@
		\begin{align}
		\eta^\nu
		&=
		\nabla f(x^\nu)
		+
		JA{(x^\nu)}^\top
		\vit{y}[0][\nu]
		+
		\tfrac{1}{\lambda_\nu}
		\sum_{j=1}^{r}
		JF^j{(x^\nu)}^\top
		R_{\sit{\beta}[j][\nu]}(F^j(x^\nu)+\lambda_\nu \vit{\hat y}[j][\nu] )
		\label{connect_L_sta_C}\\
		&
		=
		\nabla f(x^\nu)+
		JA{(x^\nu)}^\top
		\vit{y}[0][\nu]
		\nonumber\\
		&\quad+
		\tfrac{1}{\lambda_\nu}
		\sum_{j\in\tb[\nu]}
		JF^j{(x^\nu)}^\top
		Q_{\sit{\beta}[j][\nu]}(F^j(x^\nu)+\lambda_\nu \vit{\hat y}[j][\nu] )
		\begin{bmatrix}
			F^j(x^\nu)+\lambda_\nu \vit{\hat y}[j][\nu]
			-
			\vit{z}[j][\nu]
		\end{bmatrix}
		\nonumber\\
		&\quad+
		\tfrac{1}{\lambda_\nu}
		\sum_{j\in\{1,\ldots,r\}\setminus\tb[\nu]}
		JF^j{(x^\nu)}^\top
		\tfrac{1}{1-\sit{\beta}[j][\nu]}
		\begin{bmatrix}
			F^j(x^\nu)+\lambda_\nu \vit{\hat y}[j][\nu]
			-
			\vit{z}[j][\nu]
		\end{bmatrix}
		\nonumber\\
		\text{\scriptsize by~\eqref{multiplier_update},~\eqref{define_u}}
		&=
		\nabla f(x^\nu)
		+
		JA{(x^\nu)}^\top
		\vit{y}[0][\nu]
		+
		\tfrac{1}{\lambda_\nu}
		\sum_{j\in\tb[\nu]}
		JF^j{(x^\nu)}^\top
		Q_{\sit{\beta}[j][\nu]}(\vit{u}[j][\nu])
		\begin{bmatrix}
			\lambda_\nu \vit{y}[j][\nu]
		\end{bmatrix}
		\nonumber\\
		&\quad+
		\tfrac{1}{\lambda_\nu}
		\sum_{j\in\{1,\ldots,r\}\setminus\tb[\nu]}
		JF^j{(x^\nu)}^\top
		\tfrac{1}{1-\sit{\beta}[j][\nu]}
		\begin{bmatrix}
			\lambda_\nu \vit{y}[j][\nu]
		\end{bmatrix}
		\nonumber\\
		&=
		\nabla f(x^\nu)
		+
		JA{(x^\nu)}^\top
		\vit{y}[0][\nu]
		+
		\sum_{j\in\tb[\nu]}
		JF^j{(x^\nu)}^\top
		Q_{\sit{\beta}[j][\nu]}(\vit{u}[j][\nu])
		\vit{y}[j][\nu]
		+
		\sum_{j\in\{1,\ldots,r\}\setminus\tb[\nu]}
		\hspace*{-1em}
		JF^j{(x^\nu)}^\top
		\tfrac{1}{1-\sit{\beta}[j][\nu]}
		\vit{y}[j][\nu]
		\nonumber\\
		\text{\scriptsize by~\eqref{scaled_y}}
		&=
		\nabla f(x^\nu)
		+
		J\Phi{(x^\nu)}^\top
		\y_{\beta}^\nu,\label{connect_L_sta}
		\end{align}
		\end{subequations}
		where the second equality follows from \cref{proposition_connect_Rb_Pi,proposition_connect_Rb_Pi_2},
		and the fourth equality follows from \cref{proposition_Rb_homogeneity}.
		By \cref{alg_ALM_stationary} in \cref{algorithm_ALM},
		\(
		\|\nabla f(x^\nu)
		+
		J\Phi{(x^\nu)}^\top
		\y_{\beta}^\nu\|\leq\sit{\epsilon}[\rm{d}][\nu]\searrow0
		\).
		Moreover, in view of \cref{proposition_cone_Rb}, the scaled multiplier defined in~\eqref{scaled_y} satisfies
		\begin{equation}\label{multiplier_cone}
		\vit{y}[0][\nu]
		\in
		\ncone_C
		\left(
			\vit{z}[0][\nu]
		\right)
		\quad\text{and}\quad
		\vit{y_{\beta}}[j][\nu]
		\in
		\begin{cases}
			\Cncone_{D^j}
			\left(
			\vit{z}[j][\nu]
			\right),\ \ \ &\mbox{if } j\in\tb[\nu]
			\\[5pt]
			\Rncone_{D^j}
			\left(
			\vit{z}[j][\nu]
			\right),\ \ \ &\mbox{otherwise.}
		\end{cases}
		\end{equation}
		It remains to show \(\z^\nu\to_K \Phi(\bar x)\).
		To this end, we distinguish two cases, depending on whether the sequence
		\(\seq{\lambda_\nu}[\nu\in\N]\) vanishes or not.

		In the first case, let \(\lambda_\nu\searrow0\).
		Then, \(A(x^\nu)+\lambda_\nu \vit{\hat y}[0][\nu]\to_{K}A(\bar x)\in C\),
		and by~\eqref{define_u},
		\(\vit{u}[j][\nu]\to_{K}F^{j}(\bar x)\in D^{j}\),
		since \(\hat \y^\nu\in Y\) is bounded and \(\bar x\) is feasible.
		By the continuity of \(\Pi_{C}\), it follows that \(\vit{z}[0][\nu]\to_{K}\Pi_{C}(A(\bar x))=A(\bar x)\).
		For \(j\in\tb[\nu]\), let \(u=\vit{u}[j][\nu]\), \(\beta=\sit{\beta}[j][\nu]\), and \(k=k_{\beta}(u)\).
		By~\eqref{M_set},
		\[
		(k-\beta)u_{(k)}
		\leq
		\sum_{i=1}^{k}u_{(i)}
		\leq
		u_{(1)}+(k-1)u_{(k)},
		\]
		and hence \((1-\beta)u_{(k)}\leq u_{(1)}\).
		Combining this estimate with~\eqref{Qb} yields
		\[
		\left\|
			\vit{z}[j][\nu]-\vit{u}[j][\nu]
		\right\|
		=
		\left\|
			\Pi_{D^j}^{\mathrm{C},\beta}
			\left(
			u
			\right)
			-
			u
		\right\|
		=
		{\Biggl(
			\sum_{i\in\mathcal{I}_{\beta}(u)}
			u_{i}^{2}
		\Biggr)}^{\nicefrac{1}{2}}
		\leq
		\sqrt{k}u_{(k)}
		\leq
		\frac{\sqrt{p_{j}}}{1-\beta}
		u_{(1)}
		=
		\frac{\sqrt{p_{j}}}{1-\beta}
		\dist(u,D^{j}),
		\]
		where \(\mathcal{I}_{\beta}\) is defined in~\eqref{Ib}, and
		the last equality holds because \(u>\bm{0}\).
		Note that \(j\in\tb[\nu]\) implies \(k>1\).
		By the convention in~\eqref{M_set},
		\(k=1\) whenever \(\mathcal{I}_{\leq0}(u)\neq\emptyset\); hence \(u>\bm{0}\).
		On the other hand, for \(j\in\{1,\ldots,r\}\setminus\tb[\nu]\), we have \(\|\vit{z}[j][\nu]-\vit{u}[j][\nu]\|=\dist(\vit{u}[j][\nu],D^{j})\).
		Therefore, since \(\bar x\) is feasible, \(\sit{\beta}[j][\nu]\leq\max_{j}\sit{\beta}[j][0]<1\), and \(\dist(\vit{u}[j][\nu],D^{j})\to_{K}0\), it follows that
		\(\vit{z}[j][\nu]\to_K F^j(\bar x)\) for all \(j\in\set{1,\ldots,r}\).

		In the second case, let \(\lambda_\nu\) be bounded away from zero.
		\cref{proposition_feasibility} implies that \(\|\Phi(x^\nu)-\z^\nu\|\to_{K}0\).
		Then, we obtain the desired convergence directly.
		Therefore, we conclude that \(\bar x\) is an AC-stationary point of \MPVCC{}.
		By~\eqref{multiplier_cone} and the fact that \(\ncone_{D^{j}}\) is the outer limit of the regular normal cone,
		we conclude that \(\bar x\) is AM-stationary whenever \(\tb[\nu]=\emptyset\) for infinitely many \(\nu\in K\).
	\end{proof}

	\cref{theorem_AM_is_M} implies that every AC- or AM-stationary point of \MPVCC{} is, respectively, a C- or M-stationary point under \CCPC{} or \CCPM{}.
	The following corollary is immediate.

	\begin{corollary}[Stationarity]\label{corollary_stationarity}
	Suppose that the sequence \(\seq{\sit{\epsilon}[\rm{d}][\nu]}[\nu\in\N]\) in \cref{algorithm_ALM} is chosen such that \(\sit{\epsilon}[\rm{d}][\nu]\searrow 0\),
	and \(\bar x\) is a feasible cluster point of the sequence
	\(\seq{x^\nu}[\nu\in\N]\) generated by \cref{algorithm_ALM}.
	Let \(\tb[\nu]\) be defined in~\eqref{Tb_set}.
	The following hold:
	\begin{itemize}
		\item If \CCPC{} holds at \(\bar x\),
	then \(\bar x\) is a C-stationary point of \MPVCC{}.
		\item If there exists a subsequence
	\(K\in \mathcal{N}_{\infty}^{\sharp}\) such that \(x^\nu\to_K \bar x\) and such that \(\tb[\nu]=\emptyset\) for infinitely many \(\nu\in K\),
	and if \CCPM{} holds at \(\bar x\),
	then \(\bar x\) is an M-stationary point of \MPVCC{}.
	\end{itemize}
	\end{corollary}

	\begin{remark}\label{remark_scale}
	By~\eqref{eta_1} and~\eqref{connect_L_sta_C}, the terms in
	\(\nabla_x\mathcal{L}_{\lambda_\nu,\be_\nu}\) associated with the convex set \(C\) can be expressed directly through the updated multiplier \(\vit{y}[0][\nu]\),
	exactly as in the standard ALM.
	In contrast, for the nonconvex set \(D^j\), the gradient involves the scaled multiplier \(\vit{y_{\beta}}[j][\nu]\), that is, the updated multiplier \(\vit{y}[j][\nu]\) premultiplied by \(Q_{\sit{\beta}[j][\nu]}\) or \(\tfrac{1}{1-\sit{\beta}[j][\nu]}\).
	It is therefore \(\y_{\beta}^{\nu}\), rather than \(\y^{\nu}\), that approximates a multiplier of \MPVCC{}.
	\end{remark}

\subsection{Convergence Rate Analysis}
	In this part, we discuss the local convergence properties of \cref{algorithm_ALM} with setting~\eqref{A_eq_ineq}.
	Let us define the \emph{critical cone} of \MPVCC{} at a feasible point \(\bar x\in\R^n\) as
	\begin{equation}\label{critical_cone}
		\mathcal{C}(\bar x)
		\coloneqq
		\left\{
		d \in \R^{n}
		\,\middle|\,
			\langle \nabla f(\bar x),d\rangle\leq0
		\right\}
		\cap
		\left\{
		d \in \R^{n}
		\,\middle|\,
		JA(\bar x)d\in \tcone_C(A(\bar x)),\
		JF^j(\bar x)d\in \tcone_{D^j}(F^j(\bar x))
		\ \text{for all}\ j
		\right\}.
	\end{equation}
	Throughout this subsection, we assume that \(f\) and \(\Phi\) are twice continuously differentiable near \(\bar x\).
	The second-order sufficient condition for \MPVCC{} is introduced below.

	\begin{definition}[MPVCC-SOSC]\label{definition_MPVCC_SOSC}
	Relative to \MPVCC{} with notation~\eqref{define_constraint}--\eqref{define_cone} and setting~\eqref{A_eq_ineq},
	a feasible point \(\bar x\) with the multiplier \(\bar\y\) is said to satisfy the \emph{MPVCC-second-order sufficient condition (MPVCC-SOSC)} if
	\[
		\left\langle
		\nabla_{xx}^{2}\mathcal{L}(\bar x,\bar\y)d,
		d
		\right\rangle
		>0,\quad
		\forall d\in\mathcal{C}(\bar x)\setminus\{\bm{0}\},
	\]
	where the MPVCC-Lagrangian \(\mathcal{L}\) is given by~\eqref{MPVCC_Lagrangian}.
	\end{definition}

	\SOSC{} is important for establishing the local error bounds.
	We further assume that \(\bar x\) is S-stationary.
	Then, following~\cite[Theorem 3.2]{guo2013a}, the point \(\bar x\) satisfying \SOSC{} is a strong local minimizer of \MPVCC{},
	which leads to the primal error bound stated below.

	\begin{corollary}[Quadratic growth]
	Relative to \MPVCC{} with notation~\eqref{define_constraint}--\eqref{define_cone} and setting~\eqref{A_eq_ineq},
	let \(\bar x\) be a feasible point with the multiplier \(\bar\y\) satisfying \SSC{} and \SOSC{}.
	Then, there exist \(c,\delta>0\) such that
	\[
		f(x)\geq f(\bar x)+c\|x-\bar x\|^2,\quad
		\forall x\in \Phi^{-1}(\Gamma)\cap \B(\bar x,\delta).
	\]
	\end{corollary}

	However, establishing the convergence rate requires a primal-dual error bound.
	Consequently, a stronger CQ at \(\bar x\) is needed.
	The following condition corresponds to the LICQ variant for \MPVCC{} proposed in~\cite{scheel2000}.

	\begin{definition}[MPVCC-LICQ]\label{definition_MPVCC_LICQ}
	Relative to \MPVCC{} with notation~\eqref{define_constraint}--\eqref{define_cone} and setting~\eqref{A_eq_ineq},
	let \(\bar x\) be a feasible point, and \(\bar\z=\Phi(\bar x)\).
	Then, \(\bar x\) is said to satisfy the \emph{MPVCC-LICQ} if the gradients in the set
	\begin{equation}\label{definition_LICQ_grad_set}
		\mathcal{G}
		\left(
		\bar x;
		\mathcal{E}\cup\mathcal{A}(\bar x),
		\mathcal{I}_{\rm uniq}(\bar\z)\cup\mathcal{I}_{\rm mult}(\bar\z)
		\right)
	\end{equation}
	are linearly independent,
	where the sets \(\mathcal{E}\), \(\mathcal{A}(\bar x)\), \(\mathcal{I}_{\rm uniq}(\bar \z)\), \(\mathcal{I}_{\rm mult}(\bar \z)\), and \(\mathcal{G}\) are defined in~\eqref{index_set_uniq_mult} and~\eqref{define_G}.
	\end{definition}

	Combining \SSC{}, \LICQ{} and \SOSC{} enables us to leverage classical results from~\cite{birgin2012,birgin2014,steck2018} to analyze the local convergence of \cref{algorithm_ALM}.
	To this end, under notation~\eqref{define_constraint}--\eqref{define_cone},
	for any \(\bar\z\in\Gamma\), we define
	\begin{equation}\label{convexify_A}
		\mathcal{A}^{\rm vcc}(\bar\z)
		\coloneqq
		\prod_{j=1}^{r}
		\set{S_{j}\subseteq\set{1,\ldots,p_{j}}}
		[
		\emptyset\neq S_{j}
		\subseteq\mathcal{I}_{=0}(\bar z^j)
		]
		\quad\text{and}\quad
		\alpha\coloneqq(S_{1},\ldots,S_{r})\in\mathcal{A}^{\rm vcc}(\bar\z).
	\end{equation}
	For each \(\alpha\in\mathcal{A}^{\rm vcc}(\bar\z)\),
	we introduce the convex sets
	\begin{equation}\label{convexify_D}
		D^{j}_{S}
		\coloneqq
		\set{z\in\R_{+}^{p_{j}}}
		[z_{i}=0,\ i\in S_{j}]
		\quad\text{and}\quad
		\Gamma_{\alpha}
		\coloneqq
		C\times
		\prod_{j=1}^{r}
		D^{j}_{S}.
	\end{equation}

	\begin{example}
		We illustrate the constructions in~\eqref{convexify_A} and~\eqref{convexify_D}.
		By~\eqref{convex_union}, \(D^{j}\) admits the decomposition \(D^{j}=\bigcup_{i=1}^{p_j}D^{j}_{i}\), where each \(D^{j}_{i}\) is convex.
		Consider VCCs with \(r=2\), \(p_{1}=2\) and \(p_{2}=3\), namely,
		\(D^{1}=\set{z\in\R_{+}^{2}}[z_{1}z_{2}=0]\) and \(D^{2}=\set{z\in\R_{+}^{3}}[z_{1}z_{2}z_{3}=0]\).
		Let \(\bar\z=(\bar z^{1},\bar z^{2})\) with \(\bar z^{1}=(0,0)\) and \(\bar z^{2}=(0,0,0)\).
		Then, \(S_{1}\) may be chosen as \(\{1\}\), \(\{2\}\), or \(\{1,2\}\),
		yielding the corresponding sets \(D^{1}_{S}\) as \(D^{1}_{1}\), \(D^{1}_{2}\), and \(D^{1}_{1}\cap D^{1}_{2}\).
		Similarly, \(S_{2}\) may be chosen as \(\{1\}\), \(\{2\}\), \(\{3\}\), \(\{1,2\}\), \(\{1,3\}\), \(\{2,3\}\), or \(\{1,2,3\}\),
		yielding the corresponding sets \(D^{2}_{S}\) as \(D^{2}_{1}\), \(D^{2}_{2}\), \(D^{2}_{3}\), \(D^{2}_{1}\cap D^{2}_{2}\), \(D^{2}_{1}\cap D^{2}_{3}\), \(D^{2}_{2}\cap D^{2}_{3}\),
		and \(D^{2}_{1}\cap D^{2}_{2}\cap D^{2}_{3}\).
		The set \(\mathcal{A}^{\rm vcc}(\bar\z)\) is obtained by combining all admissible choices of \(S_{1}\) and \(S_{2}\),
		where each \(\alpha\in\mathcal{A}^{\rm vcc}(\bar\z)\) corresponds to one such combination.
		Consequently, because each \(D^{j}_{i}\) is convex, the associated set \(\Gamma_{\alpha}\) is also convex.
	\end{example}

	The constructions in~\eqref{convexify_A} and~\eqref{convexify_D} isolate the (possibly nonunique) active convex set \(D^{j}_{S}\) associated with each \(D^{j}\).
	The following result shows that \SSC{}, \LICQ{} and \SOSC{} imply corresponding properties for the problem constrained by the active convex set \(\Gamma_{\alpha}\).

	\begin{lemma}\label{lemma_active_branch}
	Relative to \MPVCC{} with notation~\eqref{define_constraint}--\eqref{define_cone} and setting~\eqref{A_eq_ineq},
	let \(\bar x\) be a feasible point with the multiplier \(\bar\y\) satisfying \SSC{}, \LICQ{} and \SOSC{}.
	Let \(\bar\z=\Phi(\bar x)\), and consider the problem
	\begin{equation}\label{active_branch}
	\begin{aligned}
	&\minimize_{x\in\R^{n}} \ \ f(x)
	\\
	&\stt \ \ \Phi(x)\in \Gamma_{\alpha}.
	\end{aligned}
	\end{equation}
	Then, for every \(\alpha\in\mathcal{A}^{\rm vcc}(\bar \z)\), the pair \((\bar x,\bar\y)\) satisfies the S-stationarity condition, LICQ and SOSC for~\eqref{active_branch}.

	\end{lemma}

	\begin{proof}
	First observe that \(\Gamma\) and \(\cup_{\alpha\in\mathcal{A}^{\rm vcc}(\bar \z)}\Gamma_{\alpha}\) coincide in a neighborhood of \(\bar\z\),
	since \(\mathcal{I}_{=0}(z^j)\subseteq\mathcal{I}_{=0}(\bar z^{j})\) for all \(z^j\in D^{j}\) close to \(\bar z^{j}\).
	Since regular normal cones depend only on the local geometry, at an S-stationary point of \MPVCC{} it holds that
	\[
		\bar\y\in\Rncone_{\Gamma}(\bar \z)
		=
		\Rncone_{\cup_{\alpha\in\mathcal{A}^{\rm vcc}(\bar \z)}\Gamma_{\alpha}}(\bar \z)
		=
		\bigcap_{\alpha\in\mathcal{A}^{\rm vcc}(\bar \z)}
		\ncone_{\Gamma_{\alpha}}(\bar \z),
	\]
	where the second equality follows from~\cite[Lemma 2.2]{mehlitz2020c} and the convexity of \(\Gamma_{\alpha}\).
	Consequently, \(\bar\y\in\ncone_{\Gamma_{\alpha}}(\bar \z)\) for all \(\alpha\in\mathcal{A}^{\rm vcc}(\bar \z)\),
	which establishes the S-stationarity of~\eqref{active_branch}.

	As for LICQ, fix \(\alpha=(S_{1},\ldots,S_{r})\in\mathcal{A}^{\rm vcc}(\bar \z)\)
	and write problem~\eqref{active_branch} as a standard NLP with constraints~\eqref{A_eq_ineq},
	together with \(F_{i}^{j}(x)=0\) for \(i\in S_{j}\)
	and \(F_{i}^{j}(x)\geq0\) for \(i\notin S_{j}\).
	The constraints active at \(\bar x\) are those indexed by \(\mathcal{A}(\bar x)\cup\mathcal{E}\)
	and, for each \(j\), by \(S_{j}\cup(\mathcal{I}_{=0}(\bar z^{j})\setminus S_{j})=\mathcal{I}_{=0}(\bar z^{j})\).
	Therefore, the gradients of the active constraints of~\eqref{active_branch} at \(\bar x\) coincide with those appearing in~\eqref{definition_LICQ_grad_set}.
	Hence, LICQ for~\eqref{active_branch} is equivalent to \LICQ.

	As for SOSC, for each \(\alpha\in\mathcal{A}^{\rm vcc}(\bar \z)\), define the critical cone at \(\bar x\) by
	\[
		\mathcal{C}_{\alpha}(\bar x)
		\coloneqq
		\left\{
		d \in \R^{n}
		\,\middle|\,
			\langle \nabla f(\bar x),d\rangle\leq0
		\right\}
		\cap
		\left\{
		d \in \R^{n}
		\,\middle|\,
			JA(\bar x)d\in \tcone_C(\bar z^{0}),\
			JF^j(\bar x)d\in \tcone_{D^{j}_{S}}(\bar z^j)
			\ \text{for all}\ j
		\right\}.
	\]
	Since \(\tcone_{D^{j}_{S}}(\bar z^j)\subseteq\tcone_{D^{j}}(\bar z^j)\),
	it follows that \(\mathcal{C}_{\alpha}(\bar x)\subseteq\mathcal{C}(\bar x)\),
	where \(\mathcal{C}(\bar x)\) is defined in~\eqref{critical_cone}.
	Therefore, \SOSC{} at \((\bar x,\bar\y)\) yields
	\[
		\left\langle
		\nabla_{xx}^{2}\mathcal{L}(\bar x,\bar\y)d,
		d
		\right\rangle
		>0,\quad
		\forall d\in\mathcal{C}_{\alpha}(\bar x)\setminus\{\bm{0}\}
	\]
	for every \(\alpha\in\mathcal{A}^{\rm vcc}(\bar \z)\).
	Hence, SOSC for~\eqref{active_branch} holds.
	\end{proof}

	Note that \LICQ{} guarantees the uniqueness of the multiplier \(\bar\y\) for \MPVCC{} and,
	by \cref{lemma_active_branch}, for problem~\eqref{active_branch} for every \(\alpha\in\mathcal{A}^{\rm vcc}(\bar \z)\).
	Moreover, the S-stationarity condition and LICQ for problem~\eqref{active_branch} imply the \emph{strict Robinson condition}~\cite[Eq. (11)]{ding2017} at \((\bar x,\bar\y)\); that is, in dual form,
	\[
		\nullspace J\Phi{(\bar x)}^\top
		\cap
		(
		\ncone_{\Gamma_{\alpha}}(\bar\z)+\R\bar\y
		)
		=\set{\bm{0}},
	\]
	because \(\ncone_{\Gamma_{\alpha}}(\bar\z)+\R\bar\y\subseteq\spa\ncone_{\Gamma_{\alpha}}(\bar\z)\).
	Therefore, the following primal-dual error bound can be established using results from~\cite{ding2017}.

	\begin{lemma}[Uniform primal-dual error bound]\label{lemma_uniform_primal_dual_EB}
	Relative to \MPVCC{} with notation~\eqref{define_constraint}--\eqref{define_cone} and setting~\eqref{A_eq_ineq},
	let \(\bar x\) be a feasible point with the multiplier \(\bar\y\) satisfying \SSC{}, \LICQ{} and \SOSC{}.
	Let \(\bar\z=\Phi(\bar x)\), and define the mapping
	\begin{equation}\label{Theta_residual_mapping}
		\Theta_{\alpha}(x,\y)
		\coloneqq
		\left(
		\begin{aligned}
			&\nabla_x
			\mathcal{L}(x,\y)
			\\
			\Phi(x)&-\Pi_{\Gamma_{\alpha}}(\Phi(x)+\y)
		\end{aligned}
		\right).
	\end{equation}
	Then, there exist \(c_{1},c_{2},\delta>0\) such that, for every \(\alpha\in\mathcal{A}^{\rm vcc}(\bar \z)\), the uniform primal-dual error bound holds:
	\begin{equation}\label{uniform_primal_dual_EB}
		c_{1}
		\|\Theta_{\alpha}(x,\y)\|
		\leq
		\|x-\bar x\|+\|\y-\bar\y\|
		\leq
		c_{2}
		\|\Theta_{\alpha}(x,\y)\|,\quad
		\forall (x,\y)\in \B((\bar x,\bar\y),\delta).
	\end{equation}
	\end{lemma}

	\begin{proof}
	By \cref{lemma_active_branch}, the S-stationarity condition, LICQ and SOSC hold for problem~\eqref{active_branch}.
	Therefore, by the Lipschitz continuity of \(\Theta_{\alpha}(x,\y)\) and~\cite[Theorem 24 and Remark 5]{ding2017}, it follows that,
	for every \(\alpha\in\mathcal{A}^{\rm vcc}(\bar\z)\),
	there exist constants \(c_{1,\alpha},c_{2,\alpha},\delta_{\alpha}>0\) such that
	\[
		c_{1,\alpha} \|\Theta_{\alpha}(x,\y)\|
		\leq
		\|x-\bar x\|+\|\y-\bar\y\|
		\leq
		c_{2,\alpha} \|\Theta_{\alpha}(x,\y)\|,\quad
		\forall (x,\y)\in \B((\bar x,\bar\y),\delta_{\alpha}).
	\]
	Since \(\mathcal{A}^{\rm vcc}(\bar\z)\) is finite,
	we choose
	\(c_{1}\coloneqq\min_{\alpha\in\mathcal{A}^{\rm vcc}(\bar\z)}c_{1,\alpha}>0\),
	\(c_{2}\coloneqq\max_{\alpha\in\mathcal{A}^{\rm vcc}(\bar\z)}c_{2,\alpha}<\infty\) and
	\(\delta\coloneqq\min_{\alpha\in\mathcal{A}^{\rm vcc}(\bar\z)}\delta_{\alpha}>0\).
	Then, the uniform primal-dual error bound holds.
	\end{proof}

	Recall that the scaled multiplier \(\y_{\beta}^{\nu}\) is defined in~\eqref{scaled_y}.
	The next result shows that, for the sequences generated by \cref{algorithm_ALM},
	there exists a sequence of convex sets \(\seq{\Gamma_{\alpha^\nu}}[\nu\in\N]\) such that, at each iteration \(\nu\),
	the projections in~\eqref{update_v_C}--\eqref{update_v_j} are taken onto \(\Gamma_{\alpha^\nu}\),
	and \(\y_{\beta}^{\nu}\) belongs to the regular normal cone \(\ncone_{\Gamma_{\alpha^\nu}}(\z^\nu)\).

	\begin{lemma}\label{lemma_yb_cone_inclusion}
	Relative to \MPVCC{} with notation~\eqref{define_constraint}--\eqref{define_cone},
	consider the sequences generated by \cref{algorithm_ALM}.
	Let \(\mathcal{I}_{\sit{\beta}[j][\nu]}\), \(\tb[\nu]\), \(\y_{\beta}^{\nu}\), and \(\vit{u}[j][\nu]\) be defined in~\eqref{Ib},~\eqref{Tb_set},~\eqref{scaled_y}, and~\eqref{define_u}, respectively.
	Choose
	\begin{equation}\label{define_alpha_nu}
		\alpha^{\nu}=(S_{1}^{\nu},\ldots,S_{r}^{\nu})
		\quad\text{with}\quad
		S_{j}^{\nu}=
		\begin{cases}
		\mathcal{I}_{\sit{\beta}[j][\nu]}(\vit{u}[j][\nu])
		,\ \ \ &\mbox{if } j\in\tb[\nu]
		\\[5pt]
		\set{\min\arg\min_{\ell} {[\vit{u}[j][\nu]]}_{\ell} }
		,\ \ \ &\mbox{otherwise}
		\end{cases}
	\end{equation}
	for \(j\in\{1,\ldots,r\}\).
	Then, the following hold:
	\begin{equation}\label{define_yb_inclusion}
		\z^{\nu}=\Pi_{\Gamma_{\alpha^{\nu}}}(\Phi(x^{\nu})+\lambda_{\nu} \hat\y^{\nu})
		\quad\text{and}\quad
		\y_{\beta}^{\nu}\in\ncone_{\Gamma_{\alpha^{\nu}}}(\z^{\nu}).
	\end{equation}
	\end{lemma}

	\begin{proof}
		By~\eqref{Ib}, \eqref{convexify_D}, and~\eqref{define_alpha_nu},
		the components of \(\vit{z}[j][\nu]\) vanish on \(S_{j}^{\nu}\),
		coincide with those of \(\vit{u}[j][\nu]\) on the strictly positive components of \(\vit{u}[j][\nu]\) outside \(S_{j}^{\nu}\),
		and vanish on its negative components.
		Therefore, by~\eqref{define_u}, we obtain
		\(\z^{\nu}=\Pi_{\Gamma_{\alpha^{\nu}}}(\Phi(x^{\nu})+\lambda_{\nu} \hat\y^{\nu})\).

		As for the regular normal cone inclusion,
		for each \(j\in\{1,\ldots,r\}\), let \(\vit{D_S}[j][\nu]\) denote the set \(D^{j}_{S}\) corresponding to the index set \(S_{j}^{\nu}\).
		First consider \(j\in\tb[\nu]\).
		By \cref{proposition_connect_Rb_Pi},
		we have \({[\vit{z}[j][\nu]]}_{i}>0\) for every
		\(i\notin S_{j}^{\nu}\), thus,
		\[
		\ncone_{\vit{D_S}[j][\nu]}(\vit{z}[j][\nu])=
		\set{d\in\R^{p_j}}
		[
			d_{i}=0,\ i\notin
			S_{j}^{\nu}
		].
		\]
		By~\eqref{Clarke_cone}, \cref{proposition_cone_Rb}, and the fact that \(\Cncone_{D^j}\) is a cone, we have
		\[
		\vit{y_{\beta}}[j][\nu]=\lambda_{\nu}^{-1}R_{\sit{\beta}[j][\nu]}(\vit{u}[j][\nu])
		\in \Cncone_{D^j}(\vit{z}[j][\nu])
		\subseteq
		\ncone_{\vit{D_S}[j][\nu]}(\vit{z}[j][\nu]).
		\]
		Now consider \(j\notin\tb[\nu]\).
		In this case, \(\vit{D_S}[j][\nu]=\vit{D_i}[j]\), where
		\(i=\min\arg\min_{\ell} {[\vit{u}[j][\nu]]}_{\ell}\).
		By~\eqref{multiplier_cone} and the inclusion \(\vit{D_S}[j][\nu]\subseteq D^{j}\),
		we have
		\(
		\vit{y_{\beta}}[j][\nu]\in
		\Rncone_{D^{j}}(\vit{z}[j][\nu])
		\subseteq
		\ncone_{\vit{D_S}[j][\nu]}(\vit{z}[j][\nu])
		\).
		Thus, in either case, \(\y_{\beta}^{\nu}\in\ncone_{\Gamma_{\alpha^{\nu}}}(\z^{\nu})\).
	\end{proof}

	For \(\alpha^\nu\) given by~\eqref{define_alpha_nu} and \(\Theta_{\alpha}\) defined in~\eqref{Theta_residual_mapping}, we introduce
	\begin{equation}\label{define_theta_E}
		\vartheta_{\nu}\coloneqq\|\Theta_{\alpha^{\nu}}(x^{\nu},\y_{\beta}^{\nu})\|
		\quad\text{and}\quad
		E_{\nu}\coloneqq\|x^{\nu}-\bar x\|+\|\y_{\beta}^{\nu}-\bar\y\|,
	\end{equation}
	where \(\bar\y\) denotes the multiplier associated with \(\bar x\) as in \cref{assumption_ssc_licq_sosc}, unique by \LICQ{}.
	The main result of this subsection is presented below.

	\begin{theorem}[Convergence rate]
	Relative to \MPVCC{} with notation~\eqref{define_constraint}--\eqref{define_cone} and setting~\eqref{A_eq_ineq},
	consider the sequences generated by \cref{algorithm_ALM},
	with \(\seq{\sit{\epsilon}[\rm{d}][\nu]}[\nu\in\N]\) chosen such that
	\(\sit{\epsilon}[\rm{d}][\nu]\searrow 0\) and \(\sit{\epsilon}[\rm{d}][\nu]=o(\vartheta_{\nu-1})\).
	Suppose that the primal sequence \(\seq{x^\nu}[\nu\in\N]\) admits \(x^\nu\to\bar x\),
	and that \(\y^{\nu}\in Y\) eventually always holds.
	Suppose further that
	\begin{enumerate}
		\item the functions \(f\) and \(\Phi\) are twice continuously differentiable near \(\bar x\), and
		\item the primal-dual pair \((\bar x,\bar\y)\) satisfies \SSC{}, \LICQ{} and \SOSC{}.\label{assumption_ssc_licq_sosc}
	\end{enumerate}
	Let the scaled multiplier \(\y_{\beta}^{\nu}\) be defined in~\eqref{scaled_y}.
	Then, there exists \(\lambda_{c}>0\) such that,
	if \(\lambda_\nu<\lambda_{c}\) for some \(\nu\), the sequence
	\(\seq{x^\nu,\y_{\beta}^\nu}[\nu\in\N]\) converges to \((\bar x,\bar\y)\) Q-linearly.
	\end{theorem}

	\begin{proof}
	Let \(\alpha^{\nu}\) be given by~\eqref{define_alpha_nu}, let \(\bar\z=\Phi(\bar x)\), and let \(\mathcal{A}^{\rm vcc}(\bar\z)\) be defined in~\eqref{convexify_A}.
	We begin by proving that \(\alpha^{\nu}\in\mathcal{A}^{\rm vcc}(\bar\z)\) eventually holds.
	By the boundedness of \(Y\) and the assumption that \(\y^{\nu}\in Y\) eventually holds,
	the sequence \(\seq{\y^{\nu}}[\nu\in\N]\) is bounded for \(\nu\) sufficiently large.
	\cref{scaled_y}, together with \(0<\beta_{\min}\leq\sit{\beta}[j][\nu]\leq\max_{j}\sit{\beta}[j][0]<1\),
	implies that \(\seq{\y_{\beta}^{\nu}}[\nu\in\N]\) is bounded.
	Moreover, the penalty update rule in \cref{algorithm_ALM} ensures that \(\seq{\lambda_{\nu}}[\nu\in\N]\) is nonincreasing.
	If \(\inf_{\nu}\lambda_{\nu}>0\), then \(V_{\nu}\to0\), which implies \(\z^\nu\to \bar\z\).
	Otherwise, if \(\lambda_{\nu}\searrow0\), then
	\(\z^{\nu}=\Phi(x^{\nu})+\lambda_{\nu}(\hat \y^\nu-\y^\nu)\to\bar\z\),
	where the convergence follows from \(x^\nu\to\bar x\), the continuity of \(\Phi\), and the boundedness of \(\seq{\y^\nu,\hat\y^\nu}[\nu\in\N]\).
	Consequently, since
	\(
		S_{j}^{\nu}\subseteq\mathcal{I}_{=0}(\vit{z}[j][\nu])\subseteq\mathcal{I}_{=0}(\bar z^{j})
	\)
	for all \(\vit{z}[j][\nu]\in D^{j}\) close to \(\bar z^{j}\),
	we have
	\(\alpha^{\nu}\in\mathcal{A}^{\rm vcc}(\bar\z)\) for \(\nu\) sufficiently large,
	thereby enabling the application of the error bound~\eqref{uniform_primal_dual_EB}.

	We next prove the convergence of the scaled multiplier
	\(\y_{\beta}^{\nu}\to\bar\y\).
	By~\eqref{MPVCC_Lagrangian} and~\eqref{connect_L_sta}, we obtain
	\(
		\nabla_x
		\mathcal{L}_{\lambda_\nu,\be_\nu}(x^\nu,\hat\y^{\nu})
		=
		\nabla_x
		\mathcal{L}(x^\nu,\y_{\beta}^\nu),
	\)
	so that
	\(
		\|\nabla_x\mathcal{L}(x^{\nu},\y_{\beta}^{\nu})\|\leq\sit{\epsilon}[\rm{d}][\nu]\searrow0
	\).
	Let \(\y_{\beta}^{*}\) be any cluster point of \(\seq{\y_{\beta}^\nu}[\nu\in\N]\),
	and let \(K\in \mathcal{N}_{\infty}^{\sharp}\) be a subsequence along which \(\alpha^{\nu}=\alpha\) is constant.
	Then, \(\nabla f(\bar x)+J\Phi{(\bar x)}^{\top}\y_{\beta}^{*}=\bm{0}\), and,
	by \cref{lemma_yb_cone_inclusion} and the fact that \(\ncone_{\Gamma_{\alpha}}\) is a closed cone,
	\(\y_{\beta}^{*}\in\ncone_{\Gamma_{\alpha}}(\bar\z)\).
	Since \LICQ{} implies multiplier uniqueness, it follows that \(\y_{\beta}^{*}=\bar\y\).
	Because every cluster point of \(\seq{\y_{\beta}^{\nu}}[\nu\in\N]\) coincides with \(\bar\y\), we conclude that \(\y_{\beta}^{\nu}\to\bar\y\).

	Recall that \(\vit{u}[j][\nu]\) is defined in~\eqref{define_u}.
	We next show that the limiting scaled multiplier satisfies \(\vit{\bar y}[j]=\bm{0}\) whenever \(\vit{z}[j][\nu]\to\bar z^j\) with \(|\mathcal{I}_{=0}(\bar z^j)|>1\),
	and \(\vit{u}[j][\nu]\in\R_{++}^{p_{j}}\) for infinitely many \(\nu\).
	Let \(K\in \mathcal{N}_{\infty}^{\sharp}\) be a subsequence along which
	\(\vit{u}[j][\nu]\in\R_{++}^{p_{j}}\).
	By~\eqref{Pb_Mb} and~\eqref{connect_L_sta_C}, we have \(\vit{y_{\beta}}[j][\nu]=\lambda_{\nu}^{-1}R_{\sit{\beta}[j][\nu]}(\vit{u}[j][\nu])\geq\bm{0}\).
	On the other hand, by~\eqref{regular_normal_cone} and the assumption that \((\bar x,\bar\y)\) satisfies \SSC{}, we have \(\vit{\bar y}[j]\leq\bm{0}\).
	Since \(\y_{\beta}^{\nu}\to\bar\y\), we conclude that \(\vit{\bar y}[j]=\bm{0}\) in this case.
	Define \(\check\y\coloneqq(\bar y^{0},(1-\bar\beta_{1})\bar y^{1},\ldots,(1-\bar\beta_{r})\bar y^{r})\),
	where \(\bar\beta_{j}\) is defined in~\eqref{define_limit_beta},
	so that \(\check y^{j}=\bm{0}\) in the same case.

	We then derive the norm bound \(\|\y^{\nu}-\check\y\|\leq c_3\|\y_{\beta}^{\nu}-\bar\y\|\) for some constant \(c_{3}>0\).
	First consider \(\vit{\bar y}[j]=\bm{0}\).
	Suppose that \(|S_{j}^{\nu}|>1\).
	By~\eqref{Qb} and~\eqref{multiplier_update},
	the matrix \(Q_{\sit{\beta}[j][\nu]}\) has zero rows and columns corresponding to the zero components of \(\vit{y}[j][\nu]\),
	that is, for indices \(i\notin S_{j}^{\nu}\).
	Therefore, extracting the principal submatrix associated with the nonzero components of \(\vit{y}[j][\nu]\),
	and letting \(\beta=\sit{\beta}[j][\nu]\), \(S=S_{j}^{\nu}\), and \(k=|S|\), we obtain
	\[
		Q_{\beta,S}=
		\frac{1}{\beta}
		\left(
		\frac{\bm{1}\bm{1}^\top}{k-\beta}-I
		\right).
	\]
	This matrix is nonsingular and has two distinct eigenvalues,
	namely, \(\nicefrac{1}{(k-\beta)}\) and \(\nicefrac{-1}{\beta}\).
	Then, we have \(\|Q_{\beta,S}^{-1}\|=k-\beta\leq p_{j}\).
	Since \(\vit{\bar y}[j]=\bm{0}\), we also have \(\vit{\check y}[j]=\bm{0}\).
	Consequently, \(\|\vit{y}[j][\nu]\|\leq p_{j}\|\vit{y_{\beta}}[j][\nu]\|\).
	If instead \(|S|=1\),
	\cref{proposition_connect_Rb_Pi,proposition_connect_Rb_Pi_2} imply that
	\(\|\vit{y}[j][\nu]\|= (1-\bar\beta_{j})\|\vit{y_{\beta}}[j][\nu]\|\).
	Now consider \(\vit{\bar y}[j]\neq\bm{0}\).
	Suppose that \(|S_{j}^{\nu}|>1\) for infinitely many \(\nu\).
	Then \(j\in\tb[\nu]\) for those \(\nu\), and hence \(\mathcal{I}_{\leq0}(\vit{u}[j][\nu])=\emptyset\) by the convention in~\eqref{M_set},
	that is, \(\vit{u}[j][\nu]\in\R_{++}^{p_{j}}\);
	moreover
	\(
		S_{j}^{\nu}\subseteq\mathcal{I}_{=0}(\bar z^{j})
	\)
	gives
	\(|\mathcal{I}_{=0}(\bar z^j)|>1\).
	The preceding paragraph then yields \(\vit{\bar y}[j]=\bm{0}\), a contradiction.
	Hence \(|S_{j}^{\nu}|=1\), i.e., \(k_{\sit{\beta}[j][\nu]}(\vit{u}[j][\nu])=1\), for \(\nu\) sufficiently large.
	Then, by \cref{proposition_connect_Rb_Pi,proposition_connect_Rb_Pi_2}, we have
	\(\|\vit{y}[j][\nu]-\vit{\check y}[j]\|= (1-\bar\beta_{j})\|\vit{y_{\beta}}[j][\nu]-\vit{\bar y}[j]\|\).
	Therefore, there exists a constant \(c_3\coloneqq\max\set{1,p_{1},\ldots,p_{r}}\) such that
	\(\|\y^{\nu}-\check\y\|\leq c_3\|\y_{\beta}^{\nu}-\bar\y\|\).

	We are now in a position to establish the convergence rate of \cref{algorithm_ALM}.
	Define \(T^{\nu}(s)\coloneqq s-\Pi_{\Gamma_{\alpha^{\nu}}}(s+\y_{\beta}^\nu)\), which is nonexpansive by~\cite[Lemma 2.11]{steck2018}. Therefore,
	\begin{align}
		\begin{Vmatrix}
		\Phi(x^\nu)-\Pi_{\Gamma_{\alpha^{\nu}}}(\Phi(x^\nu)+\y_{\beta}^\nu)
		\end{Vmatrix}
		=
		\begin{Vmatrix}
		T^{\nu}(\Phi(x^\nu))
		\end{Vmatrix}
		&\leq
		\begin{Vmatrix}
		T^{\nu}(\Phi(x^\nu))-T^{\nu}(\z^{\nu})
		\end{Vmatrix}
		+
		\begin{Vmatrix}
		T^{\nu}(\z^{\nu})
		\end{Vmatrix}
		\nonumber\\
		&\leq
		\begin{Vmatrix}
		\Phi(x^\nu)-\z^{\nu}
		\end{Vmatrix}
		+
		\begin{Vmatrix}
		\z^{\nu}-\Pi_{\Gamma_{\alpha^{\nu}}}(\z^{\nu}+\y_{\beta}^\nu)
		\end{Vmatrix}
		\nonumber\\
		\text{\scriptsize by \cref{lemma_yb_cone_inclusion} and~\cite[Proposition 2.36]{steck2018}}
		&=
		\begin{Vmatrix}
		\Phi(x^\nu)-\z^{\nu}
		\end{Vmatrix}
		\nonumber\\
		\text{\scriptsize by~\eqref{multiplier_update}}
		&=
		\lambda_\nu
		\begin{Vmatrix}
		\y^{\nu}-\y^{\nu-1}
		\end{Vmatrix},\label{safeguarded_redundant}
	\end{align}
	where the last equality additionally relies on the assumption that \(\y^{\nu}\in Y\) for \(\nu\) sufficiently large.
	Under the notation in~\eqref{define_theta_E}, we have
	\allowdisplaybreaks\@
	\[
	\begin{aligned}
		\vartheta_\nu
		&\leq
		\begin{Vmatrix}
		\nabla_x
		\mathcal{L}(x^\nu,\y_{\beta}^\nu)
		\end{Vmatrix}
		+
		\begin{Vmatrix}
		\Phi(x^\nu)-\Pi_{\Gamma_{\alpha^{\nu}}}(\Phi(x^\nu)+\y_{\beta}^\nu)
		\end{Vmatrix}
		\\
		\text{\scriptsize by~\eqref{safeguarded_redundant}}
		&\leq
		\sit{\epsilon}[\rm{d}][\nu]
		+
		\lambda_\nu
		\begin{Vmatrix}
		\y^{\nu}-\y^{\nu-1}
		\end{Vmatrix}
		\\
		&\leq
		\sit{\epsilon}[\rm{d}][\nu]
		+
		\lambda_\nu
		\begin{Vmatrix}
		\y^{\nu}-\check\y
		\end{Vmatrix}
		+
		\lambda_\nu
		\begin{Vmatrix}
		\check\y-\y^{\nu-1}
		\end{Vmatrix}
		\\
		&\leq
		\sit{\epsilon}[\rm{d}][\nu]
		+
		c_3\lambda_\nu
		\begin{Vmatrix}
		\y_{\beta}^{\nu}-\bar\y
		\end{Vmatrix}
		+
		c_3\lambda_\nu
		\begin{Vmatrix}
		\bar\y-\y_{\beta}^{\nu-1}
		\end{Vmatrix}
		\\
		&\leq
		\sit{\epsilon}[\rm{d}][\nu]
		+
		c_3\lambda_\nu
		E_{\nu}
		+
		c_3\lambda_\nu
		E_{\nu-1}.
	\end{aligned}
	\]
	By \(\alpha^{\nu}\in\mathcal{A}^{\rm vcc}(\bar\z)\) and \cref{lemma_uniform_primal_dual_EB},
	it follows that \(E_{\nu}\leq c_{2}\vartheta_{\nu}\) and \(c_{1}\vartheta_{\nu-1}\leq E_{\nu-1}\).
	Combining these estimates yields
	\[
		\frac{E_{\nu}}{E_{\nu-1}}
		\leq
		\frac{c_{2}\sit{\epsilon}[\rm{d}][\nu]}{(1-c_2c_3\lambda_\nu)c_{1}\vartheta_{\nu-1}}
		+
		\frac{c_2c_3\lambda_\nu}{1-c_2c_3\lambda_\nu}
		=
		o(1)
		+
		\frac{c_2c_3}{\frac{1}{\lambda_\nu}-c_2c_3},
	\]
	where the last equality follows from
	\(\sit{\epsilon}[\rm{d}][\nu]\searrow 0\) and \(\sit{\epsilon}[\rm{d}][\nu]=o(\vartheta_{\nu-1})\).
	Let \(\lambda_{c}\coloneqq\nicefrac{1}{(2c_2c_3)}\).
	If \(\lambda_\nu<\lambda_{c}\) for \(\nu\) sufficiently large,
	the sequence \(\seq{E_{\nu}}[\nu\in\N]\) converges Q-linearly to zero.
	\end{proof}

	\section{Complexity Analysis}\label{sec_complexity_analysis}
		The previous section analyzes the asymptotic behavior of \cref{algorithm_ALM} as \(\nu\to\infty\).
In practice, however, the algorithm must terminate after a finite number of iterations while satisfying an accuracy tolerance.
In this section, we study the iteration complexity of \cref{algorithm_ALM} for obtaining an approximate stationary point, defined as follows.

\begin{definition}[Approximate stationary point]\label{definition_approximate_stationary}
Relative to \MPVCC{} with notation~\eqref{define_constraint}--\eqref{define_cone},
and given constants \(\epsilon_{\rm p},\epsilon_{\rm d}>0\),
a point \(x\in\R^{n}\) is said to be \emph{\((\epsilon_{\rm p},\epsilon_{\rm d})\)-approximate M-stationary} if there exist \(\y\in\R^{m+\bar p}\) and \(\z\in\Gamma\) such that
\[
\begin{cases}
	\begin{Vmatrix}
	\nabla f(x)
	+
	J\Phi{(x)}^{\top}\y
	\end{Vmatrix}
	\leq\epsilon_{\rm d}
	\\[5pt]
	\begin{Vmatrix}
	\Phi(x)-\z
	\end{Vmatrix}
	\leq\epsilon_{\rm p}
	\\[5pt]
	\y\in \ncone_{\Gamma}(\z).
\end{cases}
\]
If the above condition holds with
\(\ncone_{\Gamma}(\z)\)
being replaced by
\(\Cncone_{\Gamma}(\z)\),
then \(x\) is said to be \emph{\((\epsilon_{\rm p},\epsilon_{\rm d})\)-approximate C-stationary}.
If \(\epsilon_{\rm p}=\epsilon_{\rm d}=\epsilon\), then \(x\) is simply called an \emph{\(\epsilon\)-approximate M- or C-stationary point}.
\end{definition}

\subsection{Outer Complexity}
	In this subsection, we establish an upper bound on the number of outer iterations in \cref{algorithm_ALM}.
	To this end, we first present three key lemmas.

	\begin{lemma}\label{lemma_V}
	Relative to \MPVCC{} with notation~\eqref{define_constraint}--\eqref{define_cone},
	consider the sequences generated by \cref{algorithm_ALM}.
	Let \(V_\nu\), \(\tb[\nu]\), and \(\y_{\beta}^{\nu}\) be defined in~\eqref{update_residual},~\eqref{Tb_set}, and~\eqref{scaled_y}, respectively.
	Given \(\epsilon_{\rm p}>0\), then \(V_\nu\leq\epsilon_{\rm p}\) implies that
	\allowdisplaybreaks\@
	\[
		\begin{cases}
		\dist(\Phi(x^\nu),\Gamma)\leq\epsilon_{\rm p}
		\\[5pt]
		\begin{Vmatrix}
			\Phi(x^\nu)-\z^\nu
		\end{Vmatrix}
		\leq\epsilon_{\rm p}
		\\[5pt]
		\y_{\beta}^\nu\in
		\begin{cases}
			\Rncone_{\Gamma}(\z^\nu),\ \ &\mbox{if } \tb[\nu]=\emptyset
			\\[5pt]
			\Cncone_{\Gamma}(\z^\nu),\ \ &\mbox{otherwise.}
		\end{cases}
		\end{cases}
		\]
	\end{lemma}

	\begin{proof}
	The result follows from~\eqref{update_v_C},~\eqref{update_v_j},~\eqref{update_residual},~\eqref{dist_upper_bound} and~\eqref{multiplier_cone}.
	\end{proof}

	\begin{lemma}\label{lemma_boundedness}
	Suppose that the sequence \(\seq{x^\nu}[\nu\in\N]\) generated by \cref{algorithm_ALM} is bounded.
	Then the following hold for all \(\nu\in\N\):
	\begin{enumerate}
		\item There exists a constant \(\bar V>0\) such that \(V_\nu\leq\bar V\).\label{lemma_V_upper_bound}
		\item There exist constants \(\bar G>0\) and \(\bar J>0\) such that
	\[
		\begin{Vmatrix}
		\nabla f(x^\nu)
		\end{Vmatrix}
		\leq
		\bar G
		\quad\text{and}\quad
		\begin{Vmatrix}
		JA(x^\nu)
		\end{Vmatrix}
		\begin{Vmatrix}
		\vit{\hat y}[0][\nu]
		\end{Vmatrix}
		+
		\sum_{j=1}^{r}
		\begin{Vmatrix}
		JF^j(x^\nu)
		\end{Vmatrix}
		\begin{Vmatrix}
		\vit{\hat y}[j][\nu]
		\end{Vmatrix}
		\leq
		\bar J.
	\]\label{lemma_jacobian_upper_bound}
	\end{enumerate}
	\end{lemma}

	\begin{proof}
	Recall that \(\vit{u}[j][\nu]\) is defined in~\eqref{define_u}.
	By~\eqref{update_v_j}, we have
	\(\|\vit{z}[j][\nu]\|\leq\|\vit{u}[j][\nu]\|\).
	Moreover, \(\seq{\vit{u}[j][\nu]}[\nu\in\N]\) is bounded because
	\(\seq{x^\nu,\hat \y^\nu}[\nu\in\N]\) are bounded and \(\lambda_{\nu}\leq\lambda_{0}\).
	Hence, \(\seq{\vit{z}[j][\nu]}[\nu\in\N]\) is also bounded.
	Furthermore, the continuity of \(\Pi_{C}\) implies that
	\(\seq{\vit{z}[0][\nu]}[\nu\in\N]\) is bounded.
	Therefore, \(\seq{V_{\nu}}[\nu\in\N]\) is bounded.
	The results of \cref{lemma_jacobian_upper_bound} follow directly from the continuity of the functions involved,
	and the boundedness of \(\seq{x^\nu,\hat \y^\nu}[\nu\in\N]\).
	\end{proof}

	\cref{lemma_boundedness} assumes that the sequence \(\seq{x^\nu}[\nu\in\N]\) generated by \cref{algorithm_ALM} is bounded.
	In practice, if \(\seq{x^\nu}[\nu\in\N]\) is unbounded,
	the algorithm could be terminated through an appropriate safeguard condition, such as \(\|x^\nu\|>c_{\max}\).
	In such a case, the algorithm would be regarded as having failed on that instance.

	Given \(\epsilon<\bar V\) and a threshold \(\underline{\lambda}>0\), we define the integers
	\begin{equation}
		N_{\rm V}\coloneqq
		\left\lceil
		\frac{\log(\epsilon/\bar V)}{\log(\theta)}
		\right\rceil\geq1
		\quad\text{and}\quad
		N_{\rm R}(\underline{\lambda})\coloneqq
		\max
		\left\{
		0,
		\left\lceil
		\frac{\log(\underline{\lambda}/\lambda_0)}{\log(\rho)}
		\right\rceil
		\right\},
	\end{equation}
	where \(\rho,\theta\in(0,1)\) are the parameters specified in \cref{algorithm_ALM}.

	\begin{lemma}\label{lemma_outer_counting}
	Given \(\epsilon>0\) with \(\epsilon<\bar V\) and \(\underline{\lambda}>0\),
	set \(\sit{\epsilon}[\rm{p}]=\sit{\epsilon}[\rm{d}]=\epsilon\).
	Suppose that the sequence \(\seq{\sit{\epsilon}[\rm{d}][\nu]}[\nu\in\N]\) in \cref{algorithm_ALM} is chosen as \(\sit{\epsilon}[\rm{d}][\nu]=\epsilon\) for all \(\nu\in\N\),
	and that the sequence
	\(\seq{x^\nu}[\nu\in\N]\) generated by \cref{algorithm_ALM} is bounded.
	Let \(N\coloneqq(1+N_{\rm R}(\underline{\lambda}))N_{\rm V}+1\).
	Then, there exists \(\nu\in\set{0,\ldots,N}\) such that
	\(V_{\nu}\leq\epsilon\) or \(\lambda_{\nu}\leq\underline{\lambda}\).
	\end{lemma}

	\begin{proof}
	We proceed by contradiction.
	Suppose that \(V_{\nu}>\epsilon\) and \(\lambda_{\nu}>\underline{\lambda}\) for all \(\nu\in\set{0,\ldots,N}\).
	Since \(\lambda_{\nu}=\rho^{t_{\nu}}\lambda_{0}\), where \(t_{\nu}\) denotes the number of \(\lambda\) reductions performed before iteration \(\nu\),
	the condition \(\lambda_{\nu}>\underline{\lambda}\) implies that \(t_{N}\leq N_{\rm R}(\underline{\lambda})\).
	Equivalently, among the iterations \(\nu\in\set{1,\ldots,N-1}\), at most \(N_{\rm R}(\underline{\lambda})\) satisfy \(V_\nu>\theta V_{\nu-1}\).
	Between any two successive such iterations,
	the sequence \(\seq{V_\nu}[\nu\in\N]\) decreases geometrically for \(L\) iterations,
	satisfying \(V_\nu\leq\theta V_{\nu-1}\).
	In view of \cref{lemma_V_upper_bound},
	it follows that \(L\leq N_{\rm V}-1\); otherwise, \(V_{\nu}\leq\epsilon\) would hold.
	Counting the iterations in \(\set{1,\ldots,N}\), we obtain
	\(N\leq(N_{\rm R}(\underline{\lambda})+1)+(N_{\rm R}(\underline{\lambda})+1)(N_{\rm V}-1)=N-1\),
	where the first term accounts for at most \(N_{\rm R}(\underline{\lambda})\) reductions in \(\set{1,\ldots,N-1}\) together with the final iteration \(N\).
	Hence, a contradiction is reached.
	\end{proof}

	\begin{theorem}[Outer complexity with nonvanishing \(\lambda\)]\label{theorem_complexity_bounded}
	Given \(\epsilon>0\) with \(\epsilon<\bar V\),
	set \(\sit{\epsilon}[\rm{p}]=\sit{\epsilon}[\rm{d}]=\epsilon\).
	Suppose that the sequence \(\seq{\sit{\epsilon}[\rm{d}][\nu]}[\nu\in\N]\) in \cref{algorithm_ALM} is chosen as \(\sit{\epsilon}[\rm{d}][\nu]=\epsilon\) for all \(\nu\in\N\),
	that the sequence
	\(\seq{x^\nu}[\nu\in\N]\) generated by \cref{algorithm_ALM} is bounded,
	and that \(\lambda_\nu>\underline{\lambda}>0\) for all \(\nu\in\N\).
	Then, within the first
	\[
		\bar\nu=
		\left(
		1
		+
		N_{\rm R}(\underline{\lambda})
		\right)
		N_{\rm V}+1
	\]
	outer iterations, there exists an index \(\nu\leq\bar\nu\) such that \(x^{\nu}\) is an \(\epsilon\)-approximate M-stationary point of \MPVCC\@ if \(\tb[\nu]=\emptyset\),
	and an \(\epsilon\)-approximate C-stationary point otherwise.
	\end{theorem}

	\begin{proof}
	By \cref{lemma_outer_counting}, there exists \(\nu\leq\bar\nu\) with \(V_{\nu}\leq\epsilon\).
	Moreover, \cref{alg_ALM_stationary} in \cref{algorithm_ALM} yields
	\(
		\|
		\nabla f(x^{\nu})
		+
		J\Phi{(x^{\nu})}^{\top}\y_{\beta}^{\nu}
		\|
		\leq\epsilon
	\).
	Therefore, the claim follows from \cref{lemma_V}.
	\end{proof}

	The outer iteration complexity established in \cref{theorem_complexity_bounded} matches the result for standard constrained optimization; see~\cite[Theorem 3.1]{birgin2020}.
	However, in general, no strategy guarantees that \( \seq{\lambda_\nu}[\nu\in\N]\) remains bounded away from zero.
	This situation may arise, for instance, when the feasible set of \MPVCC\@ is empty.
	In the following result, we consider the case where the iterates fail to be approximately stationary in the sense of \cref{definition_approximate_stationary},
	but are nevertheless approximately stationary for the feasibility problem~\eqref{feasibility_problem} associated with the current smoothing parameters.

	\begin{theorem}[Outer complexity with vanishing \(\lambda\)]\label{theorem_complexity_unbounded}
	Given \(\epsilon>0\) with \(\epsilon<\bar V\) and \(\beta_{\min}\in(0,\min\{\nicefrac{1}{2},1-\max_{j}\sit{\beta}[j][0]\})\),
	set \(\sit{\epsilon}[\rm{p}]=\sit{\epsilon}[\rm{d}]=\epsilon\).
	Suppose that the sequence \(\seq{\sit{\epsilon}[\rm{d}][\nu]}[\nu\in\N]\) in \cref{algorithm_ALM} is chosen as \(\sit{\epsilon}[\rm{d}][\nu]=\epsilon\) for all \(\nu\in\N\),
	and that the sequence
	\(\seq{x^\nu}[\nu\in\N]\) generated by \cref{algorithm_ALM} is bounded.
	Let
	\begin{equation}\label{lambda_min}
		\lambda_{\min}\coloneqq\min
		\left\{
		1,\
		\tfrac{\epsilon}{2\bar G},\
		\tfrac{\beta_{\min}\epsilon}{2\bar J}
		\right\}.
	\end{equation}
	Then, within the first
	\[
		\bar\nu=
		\left(
		1+
		N_{\rm R}(\lambda_{\min})
		\right)
		N_{\rm V}+1
	\]
	outer iterations, there exists an index \(\nu\leq\bar\nu\) such that \(x^{\nu}\) is an \(\epsilon\)-approximate M- or C-stationary point of \MPVCC{}, or
	\begin{equation}\label{2e_sta_fea}
		\begin{Vmatrix}
		\nabla
		h_{\be_{\nu}}
		(x^{\nu})
		\end{Vmatrix}
		\leq2\epsilon
		\quad\text{and}\quad
		V_{\nu}>\epsilon,
	\end{equation}
	where \(h_{\be}\) is defined in~\eqref{feasibility_problem}.
	\end{theorem}

	\begin{proof}
	By \cref{lemma_outer_counting} with \(\underline{\lambda}=\lambda_{\min}\), there exists \(\nu\leq\bar\nu\) such that \(V_{\nu}\leq\epsilon\) or \(\lambda_{\nu}\leq\lambda_{\min}\).
	In the former case, \(x^{\nu}\) is an \(\epsilon\)-approximate M- or C-stationary point of \MPVCC{}, as shown in the proof of \cref{theorem_complexity_bounded}.
	In the latter case, we may assume \(V_{\nu}>\epsilon\) and show that the first inequality in~\eqref{2e_sta_fea} holds.
	Since \(\lambda_{\nu}\leq\lambda_{\min}\), we have
	\begin{equation}\label{lambda_min_2}
		\lambda_\nu\leq1,
		\quad
		\lambda_\nu\leq\tfrac{\epsilon}{2\bar G}
		\quad\text{and}\quad
		\lambda_\nu\leq\tfrac{\beta_{\min}\epsilon}{2\bar J}.
	\end{equation}
	By \cref{alg_ALM_stationary} in \cref{algorithm_ALM} and~\eqref{eta_1}, we obtain
	\begin{equation}\label{first_order_unbounded_lambda}
		\begin{Vmatrix}
		\nabla f(x^\nu)
		+
		\tfrac{1}{\lambda_\nu}
		JA{(x^\nu)}^\top
		\begin{bmatrix}
			A(x^\nu)+\lambda_\nu
			\vit{\hat y}[0][\nu]
			-
			\vit{z}[0][\nu]
		\end{bmatrix}
		+
		\tfrac{1}{\lambda_\nu}
		\sum_{j=1}^{r}JF^j{(x^\nu)}^\top
		R_{\sit{\beta}[j][\nu]}
		(
		F^j(x^\nu)
		+
		\lambda_\nu \vit{\hat y}[j][\nu]
		)
		\end{Vmatrix}
		\leq\epsilon.
	\end{equation}
	Multiplying both sides by \(\lambda_\nu\) and applying the triangle inequality, we obtain
	\begin{align}
		\begin{Vmatrix}
		JA{(x^\nu)}^\top
		\begin{bmatrix}
			A(x^\nu)+\lambda_\nu
			\vit{\hat y}[0][\nu]
			-
			\vit{z}[0][\nu]
		\end{bmatrix}
		+
		\sum_{j=1}^{r}JF^j{(x^\nu)}^\top
		R_{\sit{\beta}[j][\nu]}
		(
			F^j(x^\nu)+\lambda_\nu
			\vit{\hat y}[j][\nu]
		)
		\end{Vmatrix}
		&\leq
		\lambda_\nu\epsilon+\lambda_\nu\|\nabla f(x^\nu)\|
		\nonumber
		\\
		\text{\scriptsize by \cref{lemma_jacobian_upper_bound}}
		&\leq
		\lambda_\nu\epsilon+\lambda_\nu\bar G
		\nonumber
		\\
		\text{\scriptsize by~\eqref{lambda_min_2}}
		&\leq\epsilon+\tfrac{1}{2}\epsilon.\label{1_5_e}
	\end{align}
	Since \(\beta_{\min}\leq\sit{\beta}[j][\nu]\leq\sit{\beta}[j][0]\) by~\eqref{alg_ALM_reduce_beta}, and
	\(\beta_{\min}\in(0,\min\{\nicefrac{1}{2},1-\max_{j}\sit{\beta}[j][0]\})\),
	the Lipschitz moduli in \cref{proposition_Rb_smooth} satisfy
	\begin{equation}\label{b_min}
		\max\{2,L_{\sit{\beta}[1][\nu]},\ldots,L_{\sit{\beta}[r][\nu]}\}
		=
		\max
		\left\{
		2,\max\{\tfrac{1}{\sit{\beta}[1][\nu]},\tfrac{1}{1-\sit{\beta}[1][\nu]}\},\ldots,\max\{\tfrac{1}{\sit{\beta}[r][\nu]},\tfrac{1}{1-\sit{\beta}[r][\nu]}\}
		\right\}
		\leq
		\tfrac{1}{\beta_{\min}}.
	\end{equation}
	Using the nonexpansiveness of the projection \(\Pi_C\),
	the global Lipschitz continuity of \(R_\beta\), and the triangle inequality,
	we derive the following upper bound:
	\allowdisplaybreaks\@
	\begin{align}
		&\!\!\!\!\!\!\!\!\!\!\!\!\!\!\!\!\!\!\!\!\!\!\!\!
		\begin{Vmatrix}
		\nabla
		\begin{bmatrix}
			\tfrac{1}{2}\dist{(A(x^\nu),C)}^2
			+
			\sum_{j=1}^{r}
			r_{\sit{\beta}[j][\nu]}(F^j(x^\nu))
		\end{bmatrix}
		\end{Vmatrix}
		-
		\begin{Vmatrix}
		\begin{aligned}
		&
		JA{(x^\nu)}^\top
		\begin{bmatrix}
			A(x^\nu)+\lambda_\nu
			\vit{\hat y}[0][\nu]
			-
			\vit{z}[0][\nu]
		\end{bmatrix}
		\\
		&\quad+
		\sum_{j=1}^{r}JF^j{(x^\nu)}^\top
		R_{\sit{\beta}[j][\nu]}
		(
			F^j(x^\nu)+\lambda_\nu
			\vit{\hat y}[j][\nu]
		)
		\end{aligned}
		\end{Vmatrix}
		\nonumber\\
		&\!\leq
		\begin{Vmatrix}
		\begin{aligned}
		&
		JA{(x^\nu)}^\top
		\begin{bmatrix}
			A(x^\nu)+\lambda_\nu
			\vit{\hat y}[0][\nu]
			-
			\Pi_C
			(
			A(x^\nu)+\lambda_\nu
			\vit{\hat y}[0][\nu]
			)
			-
			A(x^\nu)
			+
			\Pi_C(A(x^\nu))
		\end{bmatrix}
		\\
		&\quad+
		\sum_{j=1}^{r}JF^j{(x^\nu)}^\top
		\begin{bmatrix}
			R_{\sit{\beta}[j][\nu]}
			(
			F^j(x^\nu)+\lambda_\nu
			\vit{\hat y}[j][\nu]
			)
			-
			R_{\sit{\beta}[j][\nu]}(F^j(x^\nu))
		\end{bmatrix}
		\end{aligned}
		\end{Vmatrix}
		\nonumber\\
		&
		\begin{aligned}
		\leq
		\|JA{(x^\nu)}\| \|\lambda_\nu \vit{\hat y}[0][\nu]\|
		&+
		\|JA{(x^\nu)}\|\|\Pi_C(A(x^\nu)+\lambda_\nu \vit{\hat y}[0][\nu])-\Pi_C(A(x^\nu))\|
		\\
		&+
		\sum_{j=1}^{r}\|JF^j{(x^\nu)}\|
		\|R_{\sit{\beta}[j][\nu]}(F^j(x^\nu)+\lambda_\nu\vit{\hat y}[j][\nu])
		-
		R_{\sit{\beta}[j][\nu]}(F^j(x^\nu))\|
		\end{aligned}
		\nonumber\\
		\text{\scriptsize by \cref{proposition_Rb_smooth}}
		&\leq
		2\lambda_\nu\|JA{(x^\nu)}\| \|\vit{\hat y}[0][\nu]\|
		+
		\sum_{j=1}^{r}L_{\sit{\beta}[j][\nu]}\lambda_\nu\|JF^j{(x^\nu)}\|
		\|\vit{\hat y}[j][\nu]\|
		\nonumber\\
		\text{\scriptsize by \cref{lemma_jacobian_upper_bound}}
		&\leq
		\lambda_\nu\max\{2,L_{\sit{\beta}[1][\nu]},\ldots,L_{\sit{\beta}[r][\nu]}\}\bar J
		\nonumber\\
		\text{\scriptsize by~\eqref{b_min}}
		&\leq
		\lambda_\nu\tfrac{\bar J}{\beta_{\min}}
		\nonumber\\
		\text{\scriptsize by~\eqref{lambda_min_2}}
		&\leq
		\tfrac{1}{2}\epsilon.\label{0_5_e}
	\end{align}
	Combining~\eqref{1_5_e} and~\eqref{0_5_e} yields the first inequality of~\eqref{2e_sta_fea}.
	\end{proof}

	\cref{theorem_complexity_unbounded} matches the result in~\cite[Theorem 3.5]{birgin2020} for standard constrained optimization.
	It is important to note that, even when \(V_{\nu}>\epsilon\),
	the iterate \(x^\nu\) may still be approximately feasible in the sense that
	\(\dist(\Phi(x^\nu),\Gamma)\leq\epsilon\);
	see~\eqref{dist_upper_bound}.
	However, in this situation, \(x^\nu\) does not satisfy the conditions required to characterize the approximate stationarity in \cref{lemma_V}.
	On the other hand, the \(2\epsilon\) bound in~\eqref{2e_sta_fea} can be further tightened by increasing the denominator factors in~\eqref{lambda_min},
	at the expense of requiring a larger number of outer iterations.

\subsection{Total Complexity}
	At each outer iteration, the subproblem in \cref{alg_ALM_stationary} of \cref{algorithm_ALM} is solved approximately
	using an iterative method for unconstrained and smooth optimization.
	Accordingly, the total complexity of \cref{algorithm_ALM} is analyzed in this subsection
	by taking into account the iteration complexity of the inner solver.

	\begin{theorem}[Total complexity]\label{theorem_complexity_total}
	Given \(\epsilon>0\) with \(\epsilon<\bar V\),
	set \(\sit{\epsilon}[\rm{p}]=\sit{\epsilon}[\rm{d}]=\epsilon\).
	Suppose that the sequence \(\seq{\sit{\epsilon}[\rm{d}][\nu]}[\nu\in\N]\) in \cref{algorithm_ALM} is chosen as \(\sit{\epsilon}[\rm{d}][\nu]=\epsilon\) for all \(\nu\in\N\),
	and that the sequence
	\(\seq{x^\nu}[\nu\in\N]\) generated by \cref{algorithm_ALM} is bounded.
	Furthermore, assume that there exist \(a_1>0\), \(a_2\geq0\), \(a_3\geq0\), and \(q>0\)
	such that the number of gradient evaluations of the inner solver to obtain an \(\epsilon\)-stationary point of the subproblem in \cref{alg_ALM_stationary}
	is bounded above by
	\begin{equation}\label{complexity_inner_solver}
		\frac{a_1\epsilon^{-q}}{{(\lambda_\nu)}^{a_2}{(\min_{j}\{\sit{\beta}[j][\nu]\})}^{a_3}}.
	\end{equation}
	The following hold:
	\begin{enumerate}
		\item If \(\lambda_\nu>\underline{\lambda}>0\) for all \(\nu\in\N\), then after at most
	\[
		\frac{a_1\epsilon^{-q}}{{(\underline{\lambda})}^{a_2}{(\beta_{\min})}^{a_3}}
		\left[
		\left(
		1
		+
		N_{\rm R}(\underline{\lambda})
		\right)
		N_{\rm V}
		+2
		\right]
	\]
	gradient evaluations,
	we obtain an \(\epsilon\)-approximate M- or C-stationary point of \MPVCC{}.
		\item Let \(\tilde\lambda_{\min}\coloneqq\min\{\lambda_{0},\rho\lambda_{\min}\}\).
		If \(\beta_{\min}\in(0,\min\{\nicefrac{1}{2},1-\max_{j}\sit{\beta}[j][0]\})\), then after at most
	\[
		\frac{a_1\epsilon^{-q}}{{(\tilde\lambda_{\min})}^{a_2}{(\beta_{\min})}^{a_3}}
		\left[
		\left(
		1
		+
		N_{\rm R}(\lambda_{\min})
		\right)
		N_{\rm V}
		+2
		\right]
	\]
	gradient evaluations,
	we either obtain an \(\epsilon\)-approximate M- or C-stationary point of \MPVCC{},
	or a point satisfying~\eqref{2e_sta_fea}, that is, a \(2\epsilon\)-approximate stationary point of the feasibility problem~\eqref{feasibility_problem} with the current smoothing parameters, while \(V_{\nu}>\epsilon\).
	\end{enumerate}
	\end{theorem}

	\begin{proof}
	By \cref{theorem_complexity_bounded,theorem_complexity_unbounded}, the desired point is found among the first \(\bar\nu+1\) iterates, namely \(x^{0},\ldots,x^{\bar\nu}\).
	In the first case, \(\lambda_\nu>\underline{\lambda}>0\) for all of these iterations;
	hence, \(\underline{\lambda}\) can be used in~\eqref{complexity_inner_solver}.
	In the second case, the proof of \cref{theorem_complexity_unbounded} considers the first index \(\nu\) for which \(V_{\nu}\leq\epsilon\) or \(\lambda_{\nu}\leq\lambda_{\min}\).
	Prior to this index, \(\lambda_{\nu}>\lambda_{\min}\), and at this index,
	\(\lambda_{\nu}\geq\min\{\lambda_{0},\rho\lambda_{\min}\}\), since \(\lambda\) is reduced by a factor of \(\rho\) at a time.
	Therefore, \(\tilde\lambda_{\min}\) serves as a valid lower bound in~\eqref{complexity_inner_solver}.
	In both cases, \(\min_{j}\sit{\beta}[j][\nu]\geq\beta_{\min}\),
	and thus, \(\beta_{\min}\) can be substituted into~\eqref{complexity_inner_solver}.
	The claim then follows by summing~\eqref{complexity_inner_solver} over these outer iterations.
	\end{proof}

	\cref{theorem_complexity_total} requires only oracles for function values and gradients,
	and its assumption~\eqref{complexity_inner_solver} is satisfied by first-order methods
	with known worst-case guarantees for nonconvex smooth problems, such as gradient descent with Armijo line search, for which \(q=2\); see \cite{cartis2010}.
	In practice, quasi-Newton methods are well known for their strong empirical performance,
	although they do not enjoy such worst-case guarantees in the nonconvex setting.
	In the implementation presented in the next section, we employ the off-the-shelf \texttt{L-BFGS} solver~\cite{liu1989}.

	\section{Implementation of Smoothed ALM and Comparison Methods}\label{sec_implementation}
		In this section, we present the parameter choices for \cref{algorithm_ALM}, as implemented in the Julia package
\(\mathtt{Lldo.jl}\),\footnote{\url{https://github.com/JiaaaWang/Lldo.jl}}
whose name stands for \emph{\underline{L}asry--\underline{L}ions double envelope-based methods for \underline{d}isjunctive \underline{o}ptimization}.
\(\mathtt{Lldo.jl}\) considers the convex set \(C\) defined in~\eqref{A_eq_ineq}.
Additionally, \(\mathtt{Lldo.jl}\) includes several other methods for solving \MPVCC{}, which are used for comparison in the numerical experiments.

\subsection{Smoothed ALM}\label{sec_implementation_salm}
	The initial Lasry--Lions double envelope parameters are chosen as \(\lambda_0=1\) and \(\be_0=0.9\cdot\bm{1}\),
	while the initial multiplier is set to \(\y^{-1}=\bm{0}\).
	The maximum number of outer iterations is fixed at 200.
	To solve the subproblems in \cref{alg_ALM_stationary},
	we employ the \texttt{L-BFGS} solver,
	using the implementation provided in \(\mathtt{Optim.jl}\)~\cite{mogensen2018}.
	The maximum number of inner iterations is set to \(10^4\),
	and the memory parameter is set to 10.
	The accuracy tolerances are set to \(\sit{\epsilon}[\rm{p}]=\sit{\epsilon}[\rm{d}]=\epsilon=10^{-4}\).
	The bounded multiplier set \(Y\) is defined as
	\(Y\coloneqq{[y_{\min},y_{\max}]}^{m+\bar p}\),
	where \(y_{\min}=-10^{20}\) and \(y_{\max}=10^{20}\).
	The remaining parameters of \cref{algorithm_ALM} are: \(\rho=0.9\), \(\theta=0.9\), \(\kappa=0.9\), and \(\beta_{\min}=10^{-3}\).
	In accordance with \cref{sec_complexity_analysis}, we use the constant inner tolerance \(\sit{\epsilon}[\rm{d}][\nu]=\epsilon\).

	\cref{algorithm_ALM} is terminated if either
	(i) the infeasibility measure \(V_{\nu}\) in~\eqref{update_residual} satisfies \(V_{\nu}\leq\epsilon\) and
	the stationarity condition in \cref{alg_ALM_stationary} satisfies
	\(\|\nabla_x\mathcal{L}_{\lambda_\nu,\be_\nu}(x^\nu,\hat\y^{\nu})\|\leq\epsilon\),
	or (ii) the maximum number of outer iterations is reached.
	The problem is considered solved by \cref{algorithm_ALM} if the algorithm terminates according to criterion (i).

\subsection{Homotopy Method}
	The Lasry--Lions double envelope-based homotopy method for solving MPCC was proposed in~\cite{wang2025},
	and it can be extended to \MPVCC{}.
	Specifically, the following subproblem is solved at each outer iteration:
	\[
		\minimize_{x\in\R^n}
		\ \
		f(x)
		+
		\env_{\lambda}\delta_C(A(x))
		+
		\sum_{j=1}^{r}
		\tfrac{1}{\lambda}
		r_{\beta_j}(F^j(x)).
	\]
	The homotopy method can therefore be implemented by setting the multipliers to zero
	and updating the penalty parameter by \(\lambda_{\nu+1}\leftarrow\rho\lambda_\nu\) at each outer iteration.
	All other aspects remain unchanged from \cref{algorithm_ALM}.

\subsection{Classical ALM}\label{subsec_class_alm}
	The classical ALM for solving \MPVCC{} was presented in~\cite[Algorithm 1]{guo2022} and~\cite[Algorithm 4.1]{jia2023}.
	The key idea of the classical ALM is to keep the VCCs as explicit constraints in the subproblems,
	while incorporating the remaining constraints into the AL function.
	Specifically, let \(\w=(w^1,\ldots,w^r)\in\R^{\bar p}\) with \(w^j\in\R^{p_j}\) denote the slack variables.
	The \emph{classical} augmented Lagrangian for \MPVCC\@ is defined as
	\begin{align}
	&\mathcal{L}_{\lambda}^{\rm class}(x,\w,\y)\coloneqq
	f(x)
	+\env_{\lambda}\delta_C(A(x)+\lambda y^0)-\tfrac{\lambda}{2}\|y^0\|^2
	+\sum_{j=1}^{r}
	\begin{bmatrix}
		\env_{\lambda}\delta_{\{\bm{0}\}}(F^j(x)-w^j+\lambda y^j)-\tfrac{\lambda}{2}\|y^j\|^2
	\end{bmatrix}
	\nonumber\\
	&=
	f(x)
	+\tfrac{1}{2\lambda}\dist{(A(x)+\lambda y^0,C)}^2-\tfrac{\lambda}{2}\|y^0\|^2
	+\tfrac{1}{2\lambda}\sum_{j=1}^{r}
	\begin{bmatrix}
		\begin{Vmatrix}
		F^j(x)-w^j+\lambda y^j
		\end{Vmatrix}^2
		-\lambda^2\|y^j\|^2
	\end{bmatrix},\label{classical_AL}
	\end{align}
	which is similar to~\eqref{AL_function_with_indicator},
	but without the indicator functions \(\delta_C\) and \(\delta_{D^j}\).
	Under this formulation, the subproblem in the classical ALM framework becomes
	\begin{equation}\label{MPVCC_class_ALM_subproblem}
	\minimize_{x,\w}
	\ \
	\mathcal{L}_{\lambda}^{\rm class}(x,\w,\y)
	\ \
	\stt\ \
	w^j\in D^j,\ j\in\{1,\ldots,r\},
	\end{equation}
	whose objective function is continuously differentiable, but whose constraints are nonconvex.
	Nevertheless, it can be solved using the nonmonotone projected gradient (NPG) method,
	provided that the projection \(\Pi_{D^j}\) is computationally tractable, as discussed in~\cite[Section 4]{guo2022} and~\cite[Section 3]{jia2023}.
	The NPG method implemented in \(\mathtt{Lldo.jl}\) is described in \cref{algorithm_class_ALM_innersolver}.
	Note that the superscript \(k\) denotes the iteration index of the NPG method,
	while the iteration index of the outer ALM loop, \(\nu\), remains fixed.

	\begin{algorithm}[h]
	\caption{Nonmonotone Projected Gradient Method for~\eqref{MPVCC_class_ALM_subproblem}.}
	\begin{algorithmic}[1]\label{algorithm_class_ALM_innersolver}
	\itemsep=3pt%
	\renewcommand{\algorithmicindent}{0.3cm}%

	\item[\algfont{Data:}]

	Algorithm parameters \(\tau>1\), \(0<\sigma<1\), \(0<\gamma_{\min}\leq\gamma_{\max}<\infty\) and \(M\in\N\).
	Tolerance \(\epsilon\geq0\).
	Initial parameter \(\gamma_{\min}\leq\gamma_{-1}\leq\gamma_{\max}\).
	Initial points \(x^0\in\R^n\) and \(\vit{w}[j][0]\in D^j\).
	\For{\(k=0,1,2,\ldots\)}

	\State\@ Choose \(\gamma_{k}\leftarrow\gamma_{k-1}\), and if \(k\geq1\), update \(\gamma_{k}\) by
	\begin{equation}\label{inner_solver_gamma_update}
	\gamma_{k}
	\leftarrow
	\max
	\left\{
		\gamma_{\min},\
		\min
		\left\{
		\gamma_{\max},\
		\frac{
		\substack{
			{(\nabla_{x}\mathcal{L}_{\lambda}^{\rm class}(x^{k},\w^{k},\y)
			-\nabla_{x}\mathcal{L}_{\lambda}^{\rm class}(x^{k-1},\w^{k-1},\y))}^\top(x^{k}-x^{k-1}) \\
			+
			{(\nabla_{\w}\mathcal{L}_{\lambda}^{\rm class}(x^{k},\w^{k},\y)
			-\nabla_{\w}\mathcal{L}_{\lambda}^{\rm class}(x^{k-1},\w^{k-1},\y))}^\top(\w^{k}-\w^{k-1})
		}
		}{
			\|x^{k}-x^{k-1}\|^2+\|\w^{k}-\w^{k-1}\|^2
		}
		\right\}
	\right\}.
	\end{equation}
	\State\@ Update \((x^{k+1},\w^{k+1})\) by the projected gradient step:\label{innersolver_pg_step}
	\begin{align}
	&
	x^{k+1}\leftarrow x^{k}-\frac{1}{\gamma_{k}}\nabla_{x}\mathcal{L}_{\lambda}^{\rm class}(x^{k},\w^{k},\y)
	\\[5pt]
	&
	\vit{w}[j][k+1]\leftarrow \Pi_{D^j}
	\left(
		\vit{w}[j][k]
		-
		\frac{1}{\gamma_{k}}\nabla_{w^j}\mathcal{L}_{\lambda}^{\rm class}(x^{k},\w^{k},\y)
	\right).
	\end{align}

	\State\@ Return \((x^{k+1},\w^{k+1})\) if the stopping criterion is satisfied:
	\begin{equation}\label{innersolver_stop_cond}
	\left\|
	\left(
	\begin{aligned}
		&
		\left\|
		\gamma_{k}(x^k-x^{k+1})
		+
		\nabla_{x}\mathcal{L}_{\lambda}^{\rm class}(x^{k+1},\w^{k+1},\y)
		-
		\nabla_{x}\mathcal{L}_{\lambda}^{\rm class}(x^{k},\w^{k},\y)
		\right\|
		\\
		&
		\left\|
		\gamma_{k}(\w^k-\w^{k+1})
		+
		\nabla_{\w}\mathcal{L}_{\lambda}^{\rm class}(x^{k+1},\w^{k+1},\y)
		-
		\nabla_{\w}\mathcal{L}_{\lambda}^{\rm class}(x^{k},\w^{k},\y)
		\right\|
	\end{aligned}
	\right)
	\right\|
	\leq
	\epsilon.
	\end{equation}

	\State\@ Go to \cref{innersolver_pg_step} and update \(\gamma_{k}\leftarrow\tau\gamma_{k}\),
	if \((x^{k+1},\w^{k+1})\) does not satisfy the criterion:
	\begin{equation}\label{innersolver_linesearch}
	\mathcal{L}_{\lambda}^{\rm class}(x^{k+1},\w^{k+1},\y)
	\leq
	\max\limits_{{[k-M]}_{+}\leq l\leq k}
	\mathcal{L}_{\lambda}^{\rm class}(x^{l},\w^{l},\y)
	-
	\frac{\sigma\gamma_{k}}{2}
	\left\|
		\left(
		\begin{aligned}
			&x^{k+1}-x^{k}
			\\
			&\w^{k+1}-\w^{k}
		\end{aligned}
		\right)
	\right\|^2.
	\end{equation}

	\EndFor\@
	\end{algorithmic}
	\end{algorithm}

	\cref{algorithm_class_ALM_innersolver} follows~\cite[Algorithm 2]{guo2022} and~\cite[Algorithm 3.1]{jia2023}.
	Specifically, the step-size rule~\eqref{inner_solver_gamma_update} follows that of~\cite[Section 4]{guo2022} and is also adopted in~\cite{birgin2000}.
	The stopping criterion~\eqref{innersolver_stop_cond} coincides with the one used in~\cite[Eq. (3.4)]{jia2023},
	while the nonmonotone line search condition~\eqref{innersolver_linesearch} is taken from~\cite[Eq. (30)]{guo2022}.
	The convergence of \cref{algorithm_class_ALM_innersolver} to an M-stationary point of~\eqref{MPVCC_class_ALM_subproblem} is established in~\cite[Theorem 3.4]{jia2023}.
	The algorithm parameters are adopted from~\cite[Section 5]{guo2022}:
	\(\tau=5\), \(\sigma=10^{-4}\), \(\gamma_{\min}=1\), \(\gamma_{\max}=10^{20}\), \(\gamma_{-1}=1\) and \(M=10\).
	The maximum number of iterations is set to \(10^4\).

	Since the VCCs are explicitly enforced in the subproblem~\eqref{MPVCC_class_ALM_subproblem},
	only the equality constraints \(F^j(x)-w^j=\bm{0}\) remain in the outer loop of the classical ALM\@.
	Accordingly, the multiplier update in~\eqref{multiplier_update} becomes
	\[
	\vit{y}[j][\nu]
	\leftarrow
	\vit{\hat y}[j][\nu]
	+
	\frac{1}{\lambda_\nu}
	(F^j(x^\nu)-\vit{w}[j][\nu])
	\]
	for \(j\in\{1,\ldots,r\}\),
	and the infeasibility measure in~\eqref{update_residual} reduces to
	\[
	V_\nu\leftarrow
	\begin{Vmatrix}
		\Phi(x^\nu)-( \vit{z}[0][\nu] ,\w^\nu )
	\end{Vmatrix}.
	\]

	Consistent with~\cite[Eq.~(4.5) and Section 6]{jia2023},
	the classical ALM is terminated if either
	(i) \(V_\nu\leq\epsilon\) and the condition~\eqref{innersolver_stop_cond} is satisfied,
	or (ii) the maximum number of outer iterations is reached.
	The problem is considered solved by the classical ALM if the algorithm terminates according to criterion (i).
	Moreover, whenever criterion~(i) is satisfied, the resulting point is an approximate M-stationary point of \MPVCC.

\subsection{The \texttt{Algencan} Solver}
	\texttt{Algencan} is an open-source NLP solver that implements the ALM in Fortran~\cite{andreani2008}.
	\(\mathtt{Lldo.jl}\) accesses it via the interface \(\mathtt{NLPModelsAlgencan.jl}\).
	The \texttt{Algencan} solver has been shown to exhibit strong numerical performance on MPCCs~\cite{izmailov2012};
	therefore, it is of interest to investigate its performance on \MPVCC{}.
	We equivalently reformulate \MPVCC{} as the following NLP{}:
	\[
		\begin{aligned}
		\minimize_{x\in\R^n}
		\ \
		& f(x)
		\\
		\stt\ \
		& A(x)\in C
		\\
		& F^j(x)\geq\bm{0},\ \prod_{i=1}^{p_j} F^{j}_{i}(x)\leq0,
		\ j\in\{1,\ldots,r\},
		\end{aligned}
	\]
	which is then passed directly to \texttt{Algencan}.
	We use \texttt{Algencan} version 3.1.1.
	Only function values and gradients are provided.
	The resulting subproblems are solved using a truncated Newton method
	with an approximated Hessian-vector product.
	Accordingly, the \texttt{Algencan} parameters are:
	\texttt{firstde=T},
	\texttt{seconde=F},
	\texttt{truehpr=F},
	\texttt{innslvr=TN} and
	\texttt{hptype in TN=HAPPRO}.
	Moreover, the optimality and feasibility tolerances are set as
	\texttt{epsopt=}\(\epsilon\) and
	\texttt{epsfeas=}\(\epsilon\).
	All remaining parameters are kept at their default values.

	\section{Numerical Results}
		\subsection{Illustrative Examples}
	We begin by presenting three examples from the literature~\cite{liang2012,xiong2010,xu2023}.

	\begin{example}
	We consider a Stackelberg game with equilibrium constraints as studied in~\cite[Example 6.1]{xu2023}.
	Let \(\hat x=(x_1,x_2)\), \(\check x=(x_3,x_4)\), \(\tilde x=(x_5,x_6)\) and \(\rm{e}=(1,1)\), and define
	\[
		A_1=
		\begin{pmatrix}
		2 & 1
		\\
		2 & 4
		\end{pmatrix},\
		A_2=
		\begin{pmatrix}
		1 & 3
		\\
		3 & 2
		\end{pmatrix},\
		A_3=
		\begin{pmatrix}
		1 & 2
		\\
		3 & 1
		\end{pmatrix},\
		A_4=
		\begin{pmatrix}
		5 & 1
		\\
		2 & 4
		\end{pmatrix}.
	\]
	Following the reformulation procedure described in \cref{sec_intro},
	the Stackelberg game given in~\cite[Eq. (37)]{xu2023} can be reformulated as the following \MPVCC{}:
	\[
		\begin{aligned}
		\minimize_{x\in\R^6}
		\ \
		& x_1^2+{(x_5-1)}^2+{(x_2-x_6+1)}^2+\exp(x_3)+x_4
		\\
		\stt\ \
		&
		x_5^2+2x_6-2=0,\
		x_5+x_6-2\leq0,
		\\
		&\min
		\begin{Bmatrix}
		x_1+x_5,\
		{[A_1 \check x - \rm{e}]}_1,\
		{[A_2 \check x - \rm{e}]}_1
		\end{Bmatrix}=0,
		\\
		&\min
		\begin{Bmatrix}
		x_2+x_6,\
		{[A_1 \check x - \rm{e}]}_2,\
		{[A_2 \check x - \rm{e}]}_2
		\end{Bmatrix}=0,
		\\
		&\min
		\begin{Bmatrix}
		x_3,\
		{[A_3^\top (\hat x+\tilde x) - \rm{e}]}_1,\
		{[A_4^\top (\hat x+\tilde x) - \rm{e}]}_1
		\end{Bmatrix}=0,
		\\
		&\min
		\begin{Bmatrix}
		x_4,\
		{[A_3^\top (\hat x+\tilde x) - \rm{e}]}_2,\
		{[A_4^\top (\hat x+\tilde x) - \rm{e}]}_2
		\end{Bmatrix}=0.
		\end{aligned}
	\]
	The starting point is chosen as \((1,1,1,0,1,1)\), following~\cite{xu2023}.
	\end{example}

	\begin{example}
	We consider the min-max-min problem presented in~\cite[Example 6.3]{xu2023},
	which was also studied in~\cite[Example 5.2]{liang2012}.
	By applying an analogous reformulation procedure, we obtain an \MPVCC{} as follows:
	\[
		\begin{aligned}
		\minimize_{x\in\R^5}
		\ \
		& x_3
		\\
		\stt\ \
		&
		{(x_1-1)}^2+x_2^2\leq x_3,\
		{(x_1-0.5)}^2+x_2^2\leq x_3,\
		x_4\geq0,\
		x_5\geq0,
		\\
		&\min
		\begin{Bmatrix}
		1.5+0.5x_1^2-x_2+x_4,\
		0.5+0.5x_1^2-x_2+x_4,
		\\
		-1.5+{(x_2-1)}^2/2-x_1+x_4,\
		-1.5+{(x_2-1)}^2/2+x_1+x_4
		\end{Bmatrix}=0,
		\\[5pt]
		&\min
		\begin{Bmatrix}
		-5+x_1^2+x_2+x_5,\
		2+2x_1^2+2x_1-x_2+x_5,\
		-4+2x_1^2-2x_1+x_2+x_5
		\end{Bmatrix}=0,
		\end{aligned}
	\]
	where the starting point is chosen as \((0.9,0.9,0,0,0)\).
	\end{example}

	\begin{example}
	Another example of a min-max-min problem is provided in~\cite[Example 2]{xiong2010} and~\cite[Example 6.4]{xu2023}, which can be stated as follows:
	\[
		\begin{aligned}
		\minimize_{x\in\R^4}
		\ \
		& 100{(x_2-x_1^2)}^2+{(1-x_1)}^2
		\\
		\stt\ \
		&
		x_3\geq0,\
		x_4\geq0,
		\\
		&\min
		\begin{Bmatrix}
		0.75+{(x_1-1)}^2-(x_2-1)+x_3,\ \
		-0.25-{(x_1-1)}^2-(x_2-1)+x_3,
		\\
		0.75-{(x_2-1.5)}^2-(x_1-1)+x_3,\ \
		0.75-{(x_2-1.5)}^2+(x_1-1)+x_3
		\end{Bmatrix}=0,
		\\[5pt]
		&\min
		\begin{Bmatrix}
		0.5+2{(x_1-1)}^2+(x_2-1)+x_4,\ \
		1-4{(x_1-1)}^2-2(x_1-1)-(x_2-1)+x_4,
		\\
		1-4{(x_1-1)}^2+2(x_1-1)-(x_2-1)+x_4
		\end{Bmatrix}=0.
		\end{aligned}
	\]
	The starting point is \((2,2,0,0)\), the same as in~\cite{xu2023}.
	\end{example}

	The numerical results for the three examples obtained using the smoothed ALM, the classical ALM, the homotopy method, and \texttt{Algencan} are summarized in \cref{table_MPVCC},
	where the constraint violation is measured by
	\(
		\left\|
		\left(
		A(x)-\Pi_{C}(A(x)),
		\min_{i}F_{i}^{1}(x),\ldots,
		\min_{i}F_{i}^{r}(x)
		\right)
		\right\|_{\infty}
	\).
	For Examples 7.1 and 7.2, all four methods converge to the same minimizer, consistent with the results reported in~\cite{liang2012,xiong2010,xu2023}.
	For Example 7.3, the classical ALM terminates at a point with a higher objective value,
	and \texttt{Algencan} terminates at an infeasible point,
	whereas the other two methods converge to the same minimizer reported in~\cite{xu2023}.
	Across these three examples, the smoothed ALM requires the fewest inner iterations and no more outer iterations than the classical ALM and the homotopy method.
	Moreover, it achieves the shortest computation time among all the methods tested.
	On the other hand, recall that a smaller final value of \(\lambda\) corresponds to a larger penalty parameter.
	The results show that the smoothed and classical ALMs maintain considerably smaller penalty parameters than the homotopy method.

	\begin{table}[H]
	\centering
	\caption{Numerical results for the three examples}
	\renewcommand{\arraystretch}{1.1}
	\setlength{\tabcolsep}{5pt}
	\begin{tabular}{llrrrrrr}
		\hline
		Example & Method & Objective & Violation & Time (s) & Outer & Inner & Final $\lambda$ \\
		\hline
		7.1 & Smoothed ALM  & 1.69 & $9.89\cdot10^{-6}$ & $\bm{1.74\cdot10^{-3}}$ & \textbf{4}  & \textbf{39}   & \textbf{1.00} \\
			& Classical ALM & 1.69 & $5.93\cdot10^{-5}$ & $9.01\cdot10^{-1}$ & 10 & 1872 & \textbf{1.00} \\
			& Homotopy      & 1.69 & $9.11\cdot10^{-5}$ & $1.37\cdot10^{-2}$ & 64 & 237  & $1.31\cdot10^{-3}$ \\
			& \texttt{Algencan} & 1.69 & $\bm{1.76\cdot10^{-6}}$ & $4.67\cdot10^{-2}$ & -- & --  & -- \\
		\hline
		7.2 & Smoothed ALM  & $6.25\cdot10^{-2}$ & $\bm{2.02\cdot10^{-5}}$ & $\bm{4.50\cdot10^{-4}}$ & \textbf{2}  & \textbf{29}   & \textbf{1.00} \\
			& Classical ALM & $6.25\cdot10^{-2}$ & $6.49\cdot10^{-5}$ & $1.97\cdot10^{-1}$ & \textbf{2}  & 505  & \textbf{1.00} \\
			& Homotopy      & $6.25\cdot10^{-2}$ & $6.45\cdot10^{-5}$ & $1.16\cdot10^{-2}$ & 86 & 197  & $1.29\cdot10^{-4}$ \\
			& \texttt{Algencan} & $6.25\cdot10^{-2}$ & $2.62\cdot10^{-5}$ & $2.10\cdot10^{-2}$ & -- & --  & -- \\
		\hline
		7.3 & Smoothed ALM  & $\bm{4.00\cdot10^{-2}}$ & $\bm{3.55\cdot10^{-6}}$ & $\bm{7.06\cdot10^{-4}}$ & \textbf{5}  & \textbf{37}   & $\bm{9.00\cdot10^{-1}}$ \\
			& Classical ALM & $2.00\cdot10^{-1}$  & $5.95\cdot10^{-6}$ & $3.27\cdot10^{-1}$ & 9  & 7355 & $6.56\cdot10^{-1}$ \\
			& Homotopy      & $\bm{4.00\cdot10^{-2}}$ & $9.70\cdot10^{-5}$ & $6.07\cdot10^{-3}$ & 63 & 344  & $1.46\cdot10^{-3}$ \\
			& \texttt{Algencan} & $2.23\cdot10^{1}$ & $6.39\cdot10^{-1}$ & $8.35$ & -- & --  & -- \\
		\hline
	\end{tabular}\label{table_MPVCC}
	\end{table}

\subsection{\texttt{MacMPEC} Benchmark Problems}
	We benchmark the methods in \(\mathtt{Lldo.jl}\) on the \texttt{MacMPEC},\footnote{\url{https://wiki.mcs.anl.gov/leyffer/index.php/MacMPEC}}
	which is a widely used collection of MPCCs in the literature.
	In~\cite{nurkanovic2025b}, the \texttt{MacMPEC} problems were translated from \(\mathtt{AMPL}\) to \(\mathtt{CasADi}\)~\cite{andersson2019},
	and made available in the repository,\footnote{\url{https://cloud.syscop.de/s/rBnTMocFoLcNLWG}}
	which we use in our implementation.
	As a preliminary result, we consider 102 problems from the collection
	with fewer than 120 variables,
	each initialized at its default starting point.

	\begin{figure}[h]
		\centering
		\begin{subfigure}[t]{0.64\textwidth}
			\centering
			\includegraphics[width=\linewidth]{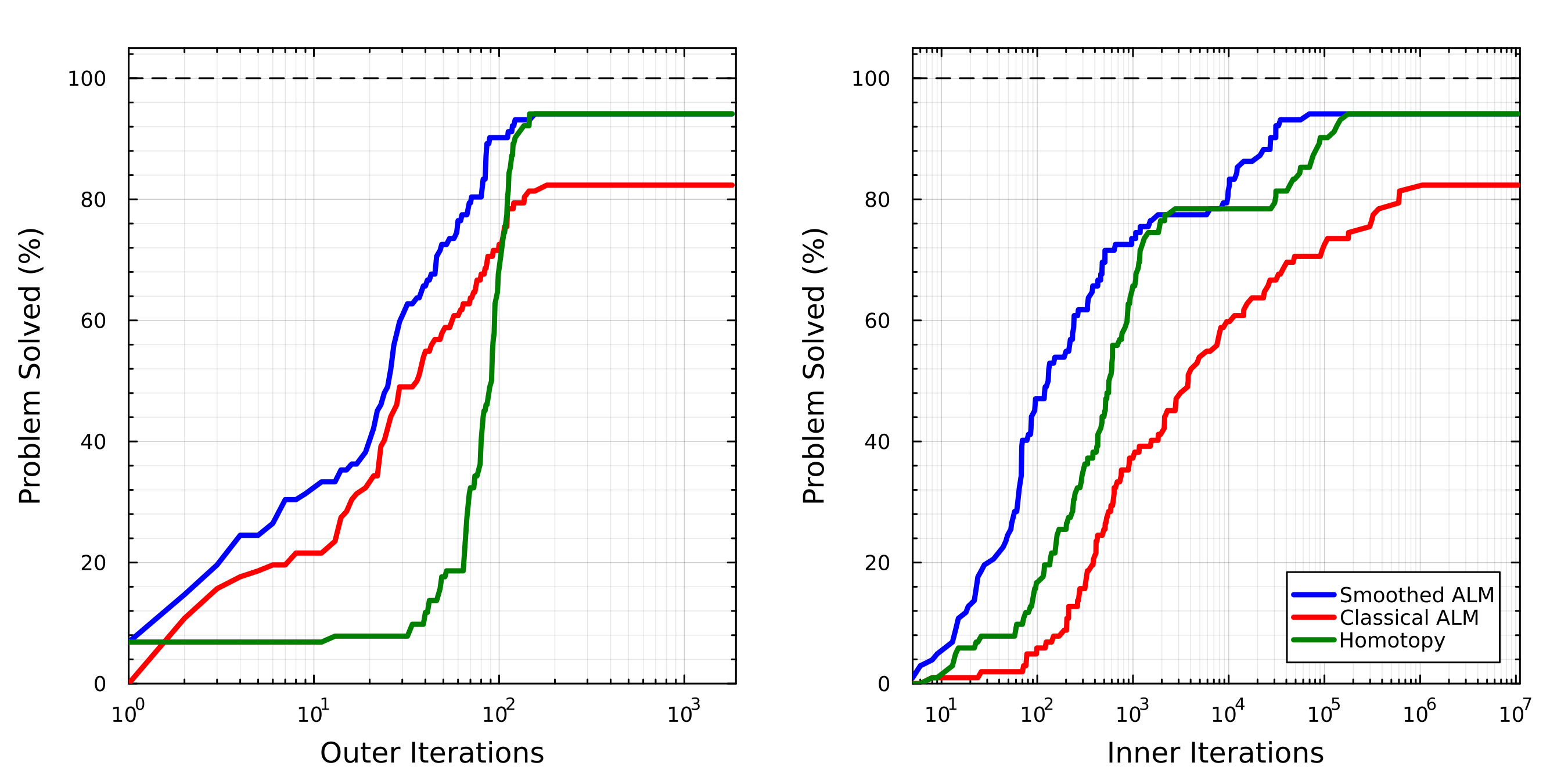}
		\end{subfigure}
		\hspace{-0.025\textwidth}%
		\begin{subfigure}[t]{0.34\textwidth}
			\centering
			\vspace{-52.5mm}
			\includegraphics[width=0.96\linewidth]{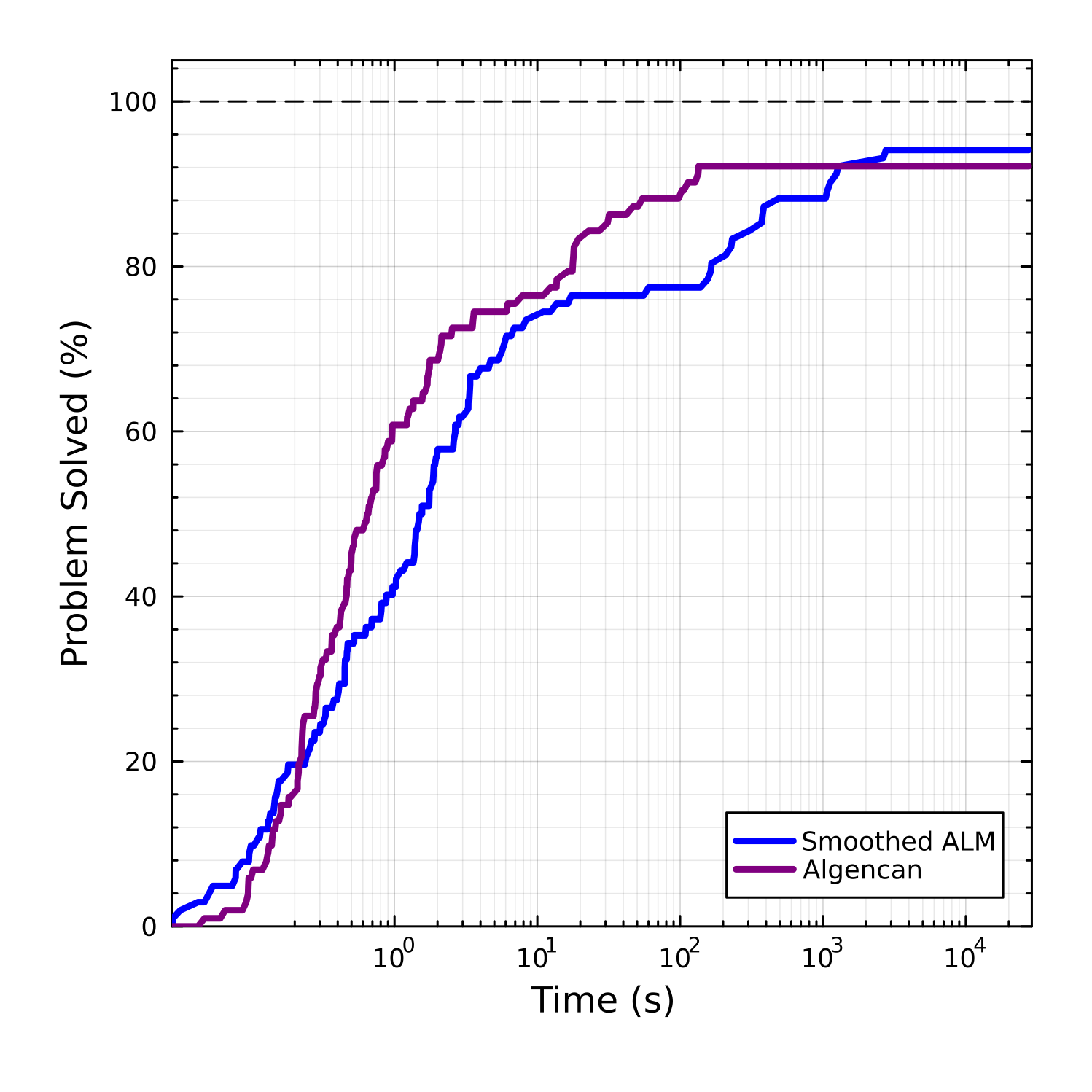}
		\end{subfigure}
		\caption{Outer and inner iterations, and computational time versus the proportion of problems solved on the \texttt{MacMPEC} benchmark.}\label{fig_data_profile_macmpec}
	\end{figure}

	The data profiles are presented in \cref{fig_data_profile_macmpec}.
	Compared with the classical ALM and the homotopy method,
	the smoothed ALM solves more problems for most budgets of outer and inner iterations,
	and the homotopy method eventually attains a similar number of solved problems.
	Since \texttt{Algencan} is an external package, it is more appropriate to compare performance in terms of computational time.
	The right subfigure shows that \texttt{Algencan} is faster than the smoothed ALM on the selected test set,
	whereas the smoothed ALM solves a slightly larger fraction of the instances overall.

	The numbers of unsolved problems and instances of convergence to suboptimal solutions under the four methods are reported in \cref{fig_unsolved_subopt_macmpec}.
	The count of unsolved problems includes both convergence to infeasible points and  failures of the \texttt{L-BFGS} and \texttt{Algencan} solvers.
	Regarding solution quality on the solved problems,
	a returned objective value is regarded as optimal if its relative suboptimality with respect to the best known objective value,
	\(\nicefrac{(\text{obj}-\text{obj}^{\star})}{\max\{1,|\text{obj}^{\star}|\}}\),
	does not exceed \(0.01\).
	\cref{fig_unsolved_subopt_macmpec} shows that the smoothed ALM, the homotopy method, and \texttt{Algencan} exhibit similar performance with respect to the numbers of unsolved and suboptimal instances,
	whereas the classical ALM has considerably more unsolved problems.

	\begin{figure}[h]\centering
	\includegraphics[width=0.6\textwidth]{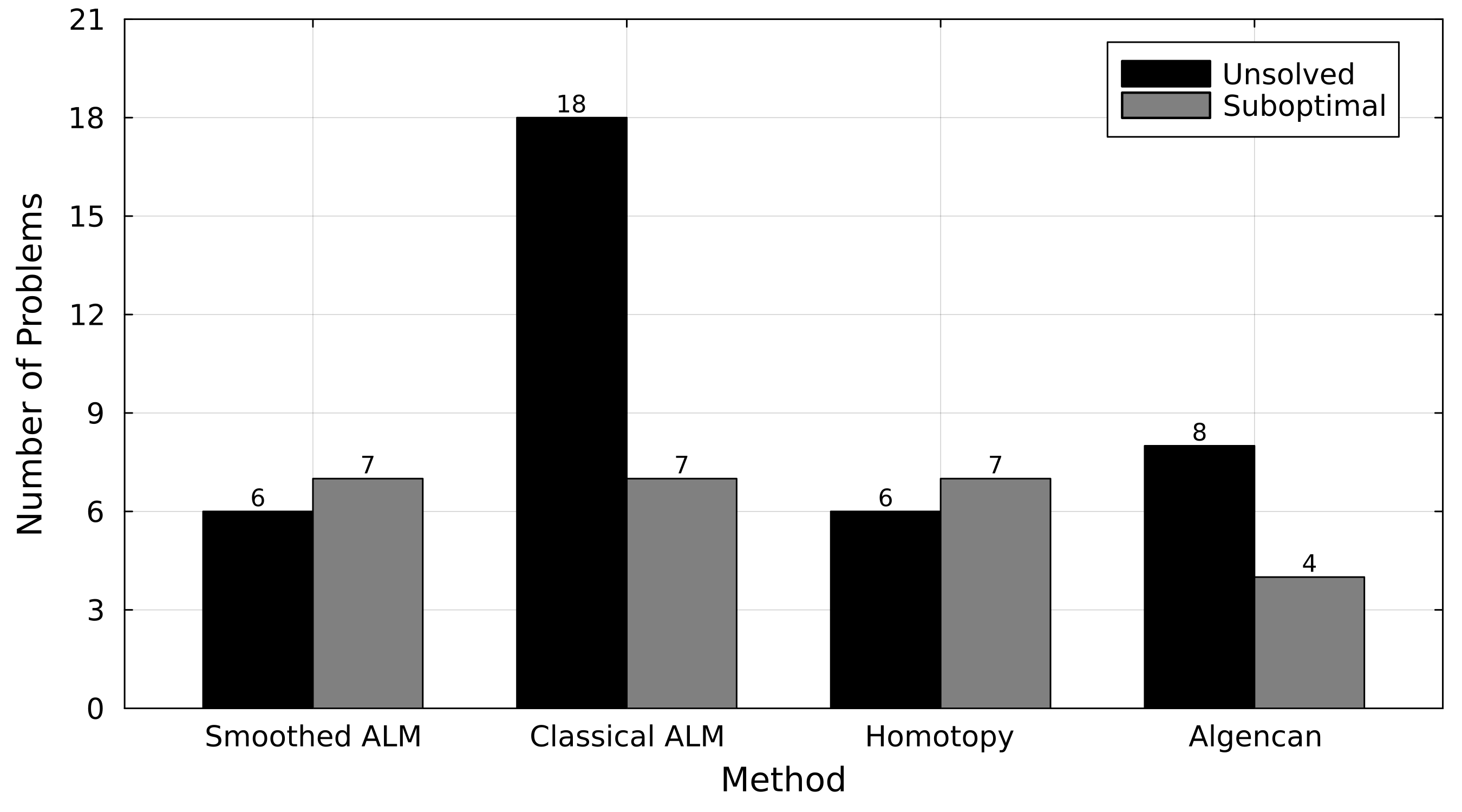}
	\caption{Number of unsolved problems and instances of convergence to suboptimal values on the \texttt{MacMPEC} benchmark.}\label{fig_unsolved_subopt_macmpec}
	\end{figure}

	\begin{figure}[h]\centering
	\includegraphics[width=0.8\textwidth]{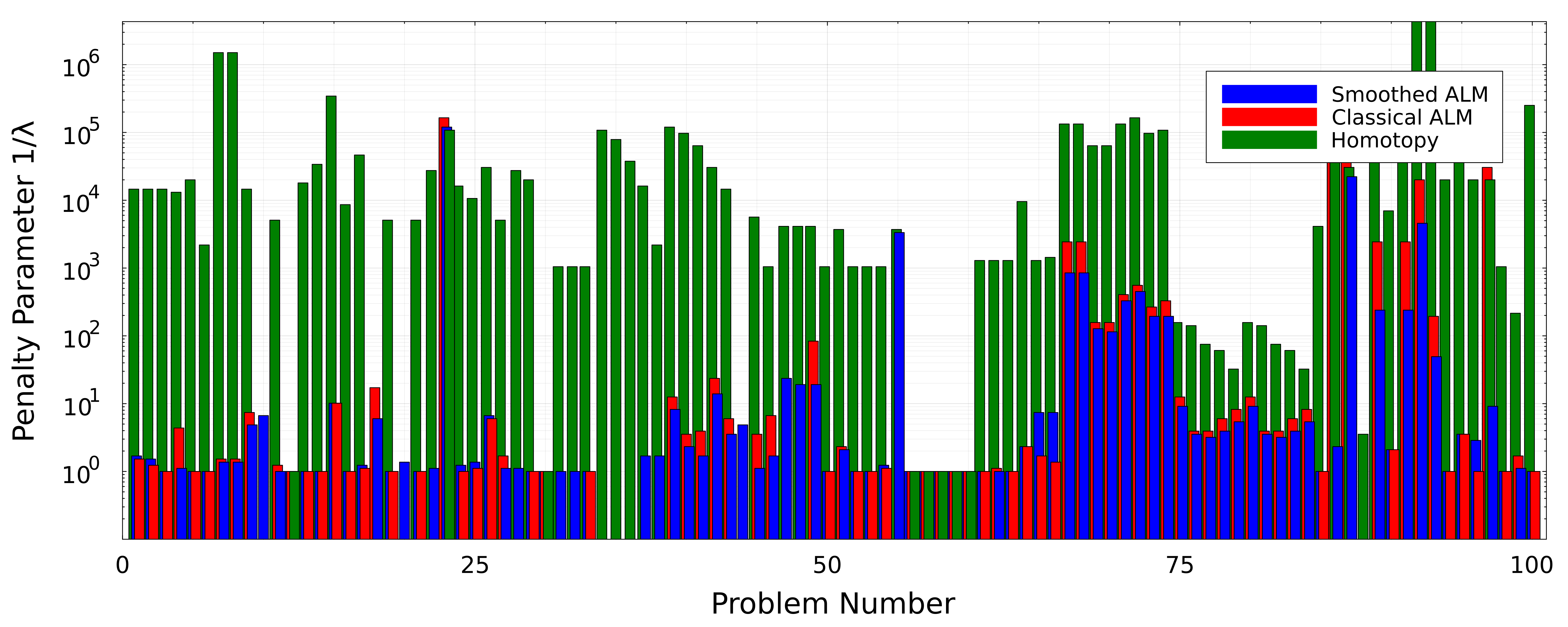}
	\caption{Penalty parameters on the \texttt{MacMPEC} benchmark.}\label{fig_param_macmpec}
	\end{figure}

	The final values of the penalty parameter \(\nicefrac{1}{\lambda}\) for the solved problems are reported in \cref{fig_param_macmpec}.
	It is evident that the smoothed and classical ALMs control the growth of the penalty parameter more effectively than the homotopy method,
	in which the penalty parameter is increased constantly at each outer iteration without monitoring the infeasibility measure,
	unlike the update rule in~\eqref{alg_ALM_update_lambda} of \cref{algorithm_ALM}.

\subsection{\texttt{QPECgen} Benchmark Problems}
	In this subsection, we consider larger MPCCs with quadratic objective and linearly constrained variables:
	\begin{equation}\label{QPCC_general}
	\begin{split}
	\minimize_{x\in \R^n}
	\ \
	& \tfrac{1}{2} x^\top Qx+g^\top x
	\\
	\stt\ \
	& Ax^0+a\leq \bm{0}
	\\
	& \bm{0}\leq x^1\ \bot\ Nx^0+Mx^1+q\geq\bm{0},
	\end{split}
	\end{equation}
	where \(x=(x^0,x^1)\in\R^{n}\), \(x^0\in\R^{n_0}\) and \(x^1\in\R^r\).
	We consider a range of problem sizes:
	\((n_0,r)=\)
	(10,10), (20,20), (30,30), (40,40), (50,50), (100,100), (200,200), (500,500), (1000,1000).
	For each size, we randomly generate five instances by using the code in~\cite{jiang1999},\footnote{\url{https://www3.eng.cam.ac.uk/~dr241/Papers/qpecgen.m}}
	with the random seed set from one to five.
	Each instance is initialized at the zero starting point.

	\begin{figure}[h]
		\centering
		\begin{subfigure}[t]{0.64\textwidth}
			\centering
			\includegraphics[width=\linewidth]{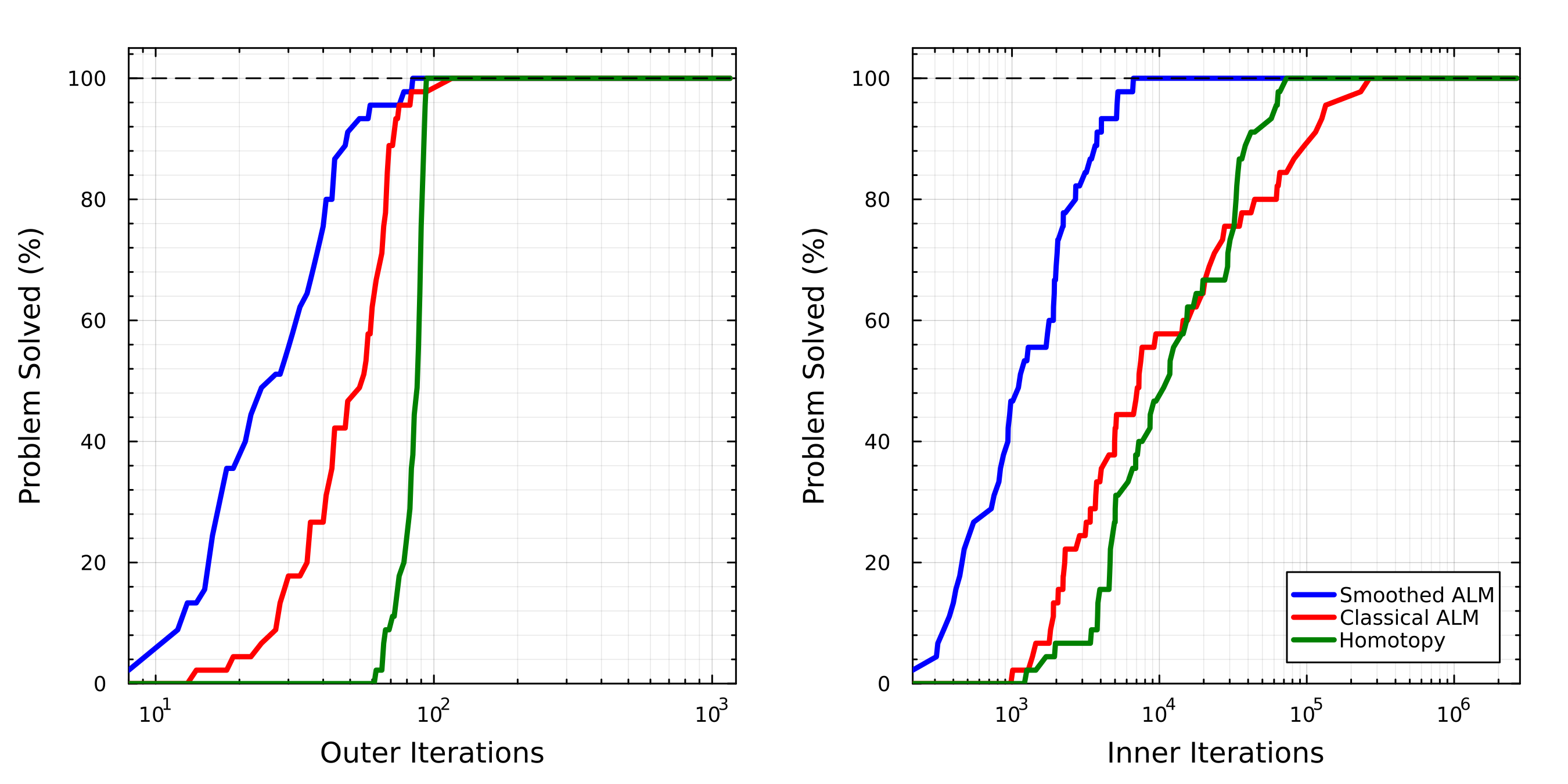}
		\end{subfigure}
		\hspace{-0.025\textwidth}%
		\begin{subfigure}[t]{0.34\textwidth}
			\centering
			\vspace{-52.5mm}
			\includegraphics[width=0.96\linewidth]{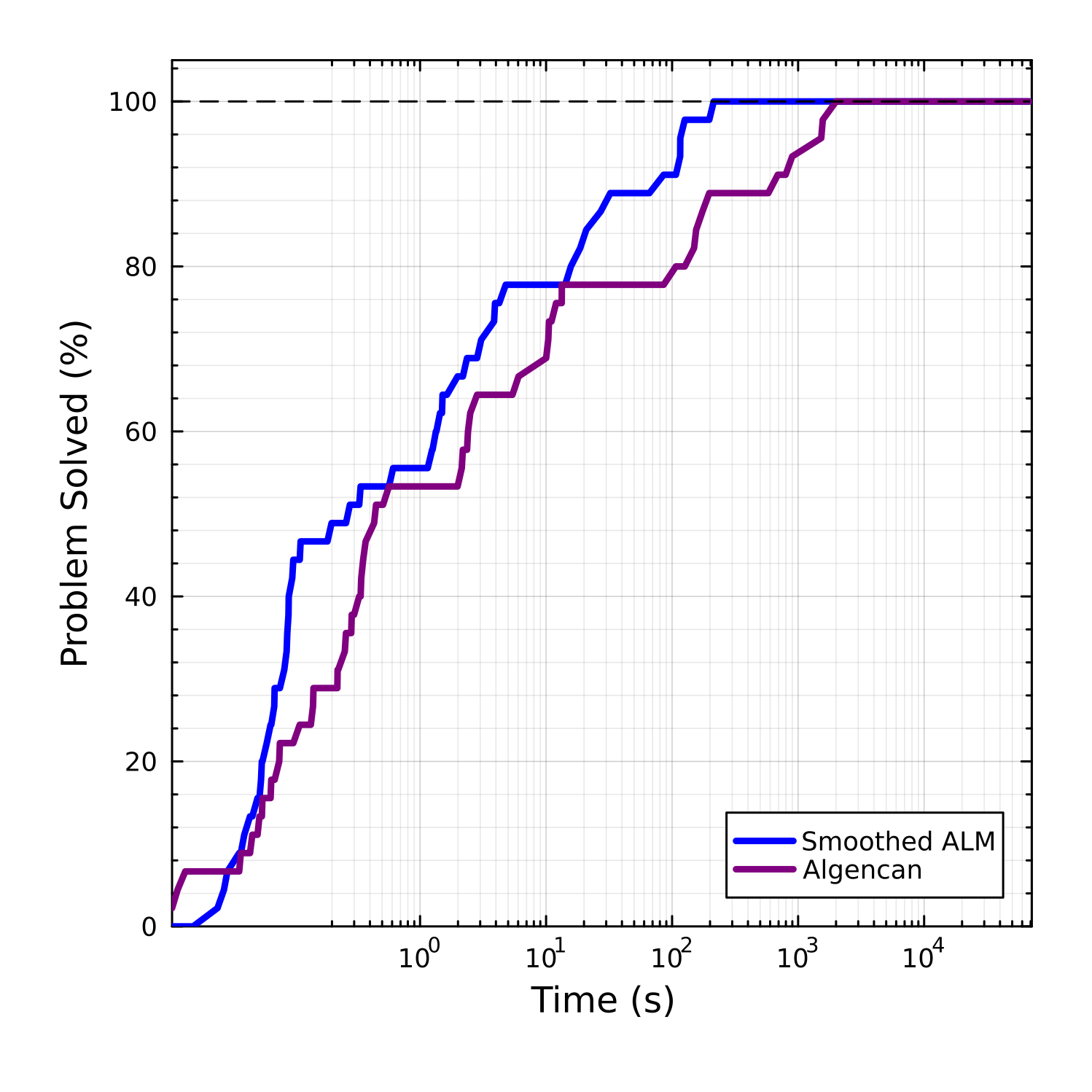}
		\end{subfigure}
		\caption{Outer and inner iterations, and computational time versus the proportion of problems solved on the \texttt{QPECgen} benchmark.}\label{fig_data_profile_qpecgen}
	\end{figure}

	For the problem data, we set the following: the matrix \(Q\) is positive definite with a condition number \(10\), the number of rows of \(A\) is \(\nicefrac{1}{2}\times n_0\),
	and the matrix \(M\) is nonsymmetric and nonmonotone (\(M+M^\top\) is not positive semidefinite) with a condition number \(20\).
	The primal-dual feasibility and complementary slackness in the optimality conditions of~\eqref{QPCC_general} are
	\begin{align*}
	&
	{[Ax^0+a]}_i\leq 0,\
	\alpha_i\geq0\ \text{and}\
	\alpha_i{[Ax^0+a]}_i=0,\
	i\in\mathcal{I}
	\\
	&
	x^1\geq0,\ Nx^0+Mx^1+q\geq0,\
	y_i^G x^1_i=0\ \text{and}\
	y_i^H {[Nx^0+Mx^1+q]}_{i}=0,\
	i\in\mathcal{P},
	\end{align*}
	where \(\mathcal{I},\mathcal{P}\) are index sets of inequality and complementarity constraints, respectively.
	By~\cite[Definition 1]{jiang1999}, an index \(i\in\mathcal{I}\) is called \emph{first-level degenerate} if \(\alpha_i={[Ax^0+a]}_i=0\),
	an index \(i\in\mathcal{P}\) is called \emph{second-level degenerate} if \(i\in I_{00}\) and an index \(i\in I_{00}\) is called \emph{mixed degenerate} if either \(y_i^G=0\) or \(y_i^H=0\),
	in which the set \(I_{00}\) is the bi-active index set of complementarity constraints.
	The cardinality of the first-level degeneracy index set is chosen to be \(\lceil\nicefrac{3}{10}\times n_0\rceil\).
	The cardinalities of second-level and mixed degeneracy index sets are chosen to be \(\lceil\nicefrac{3}{4}\times r\rceil\) and \(\lceil\nicefrac{1}{2}\times r\rceil\), respectively.

	The data profiles are shown in \cref{fig_data_profile_qpecgen}.
	All 45 instances are solved by the four methods,
	and the smoothed ALM requires the fewest outer and inner iterations, as well as the shortest computational time.
	Regarding solution quality, \texttt{Algencan} finds solutions that are better or equally as good as those obtained by the other methods.
	The smoothed ALM terminates at suboptimal solutions on 5 instances, the classical ALM on 11 instances, and the homotopy method on 14 instances.
	The final values of the penalty parameter \(\nicefrac{1}{\lambda}\) are presented in \cref{fig_param_qpecgen},
	which consistently indicate that the smoothed and classical ALMs provide better control over the growth of the penalty parameter than the homotopy method.

	\begin{figure}[htbp]\centering
	\includegraphics[width=0.8\textwidth]{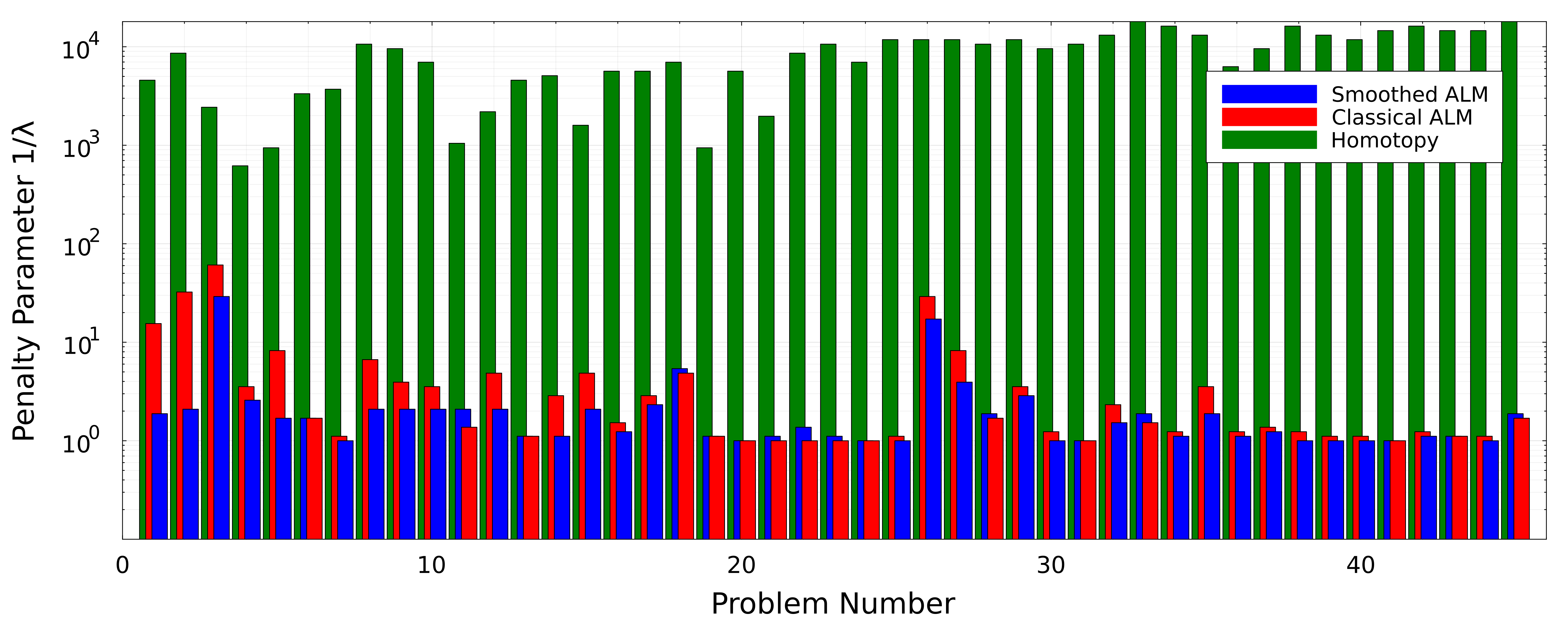}
	\caption{Penalty parameters on the \texttt{QPECgen} benchmark.}\label{fig_param_qpecgen}
	\end{figure}

\subsection{Comparison Between MPVCC and Its MPCC Reformulation}
	It is well known that VCCs can be reformulated as CCs by introducing auxiliary variables, which may make \MPVCC{} appear less attractive than MPCC{}.
	However, as shown in this subsection, the smoothed ALM can solve \MPVCC{} more efficiently than its MPCC reformulation.
	We consider the VLCP described in~\eqref{VLCP} as an illustrative example.
	As discussed in \cref{sec_intro}, the solution of the VLCP may not be unique;
	therefore, we adopt the objective function \(f(x)=\tfrac{1}{2}\|x\|^2\).
	Let \(x=(x^0,x^1,\ldots,x^\ell)\in\R^{r(\ell+1)}\) with
	\(x^0,x^1,\ldots,x^\ell\in\R^{r}\),
	\(B_1,\ldots,B_\ell\in\R^{r\times r}\), and
	\(b^1,\ldots,b^\ell\in\R^{r}\).
	Then, \MPVCC{} with objective \(f(x)\) and VCCs~\eqref{VLCP} can be formulated as
	\begin{equation}\label{VLCP_original}
		\begin{split}
		\minimize_{x\in \R^{r(\ell+1)}}
		\ \
		& \tfrac{1}{2}\|x\|^2
		\\
		\stt\ \
		&
		x^j=B_j x^0+ b^j,\ j\in\{1,\ldots,\ell\}
		\\
		&
		\min
		\left\{
		x^{0}_{i}, x^{1}_{i},\ldots, x^{\ell}_{i}
		\right\}=0,\
		i\in\{1,\ldots,r\}.
	\end{split}
	\end{equation}
	Following~\cite[Section 2.3]{ferris1997}, we introduce the variable
	\(\s=(s^{1},\ldots,s^{\ell})\in\R^{r\ell}\) with \(s^k\in\R^{r}\).
	Problem~\eqref{VLCP_original} admits the following CC reformulation:

	\begin{equation}\label{VLCP_reformulation}
		\begin{split}
		\minimize_{x,\s}
		\ \
		& \tfrac{1}{2}\|x\|^2
		\\
		\stt\ \
		&
		x^0=\sum_{k=1}^{\ell}s^k
		\quad\text{and}\quad
		x^{j'}=B_{j'} x^0+ b^{j'},\ j'\in\{1,\ldots,\ell\}
		\\
		&
		\bm{0}\leq
		s^\ell
		\perp
		x^\ell
		\geq\bm{0}
		\quad\text{and}\quad
		\bm{0}\leq
		s^j
		\perp
		x^j-\sum_{k=j+1}^{\ell}s^k
		\geq\bm{0},\ j\in\{1,\ldots,\ell-1\}.
		\end{split}
	\end{equation}
	Note that the number of variables in~\eqref{VLCP_original} is \(r(\ell+1)\),
	and the number of VCCs is \(r\), each of dimension \(\ell+1\).
	In contrast, the reformulated problem~\eqref{VLCP_reformulation} has \(r(\ell+1)+r\ell\) variables and \(r\ell\) CCs.
	We set \(r=2\) and consider \(\ell=5\text{--}15, 50, 100, 150, 200, 250, 300, 350, 400, 450, 500\).
	For each size, five instances of~\eqref{VLCP} are generated randomly following the procedure described in~\cite{mezzadri2022}.
	Specifically, we first generate a feasible solution based on~\cite[Eq.~(40)]{mezzadri2022}.
	We then construct random matrices \(B_1,\ldots,B_\ell\) with nonzero entries sampled from \([-1,1]\).
	Using the generated solution together with the equality constraints \(x^j = B_j x^0 + b^j\), we compute \(b^1,\ldots,b^\ell\).
	For all instances, the starting point is chosen randomly from the interval \([-10,10]\).

	\begin{figure}[htbp]\centering
	\includegraphics[width=0.45\textwidth]{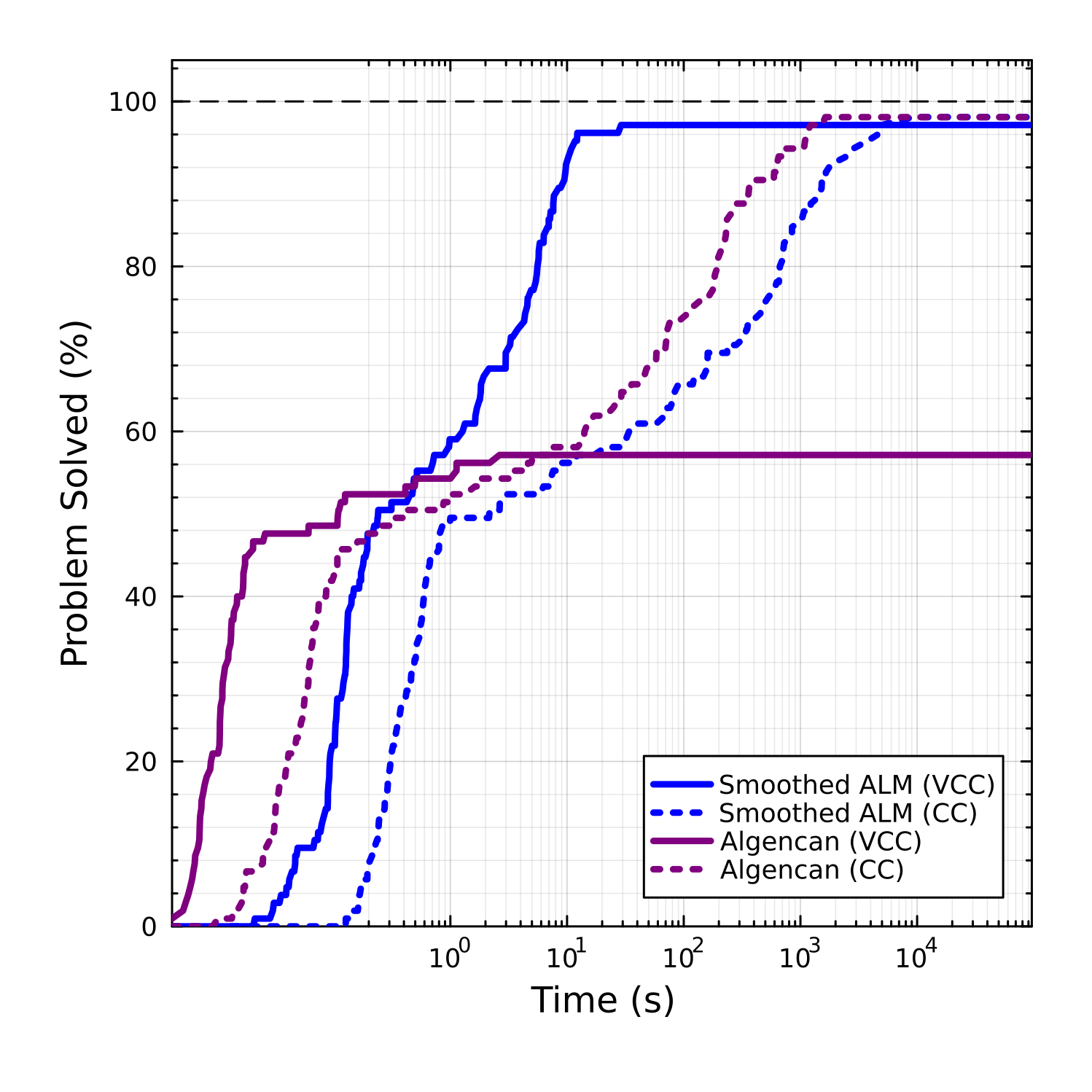}
	\caption{Computational time versus the proportion of problems solved for MPVCC and its MPCC reformulation on the VLCP benchmark.}\label{fig_data_profile_vlcp}
	\end{figure}

	\begin{figure}[h]\centering
	\includegraphics[width=0.6\textwidth]{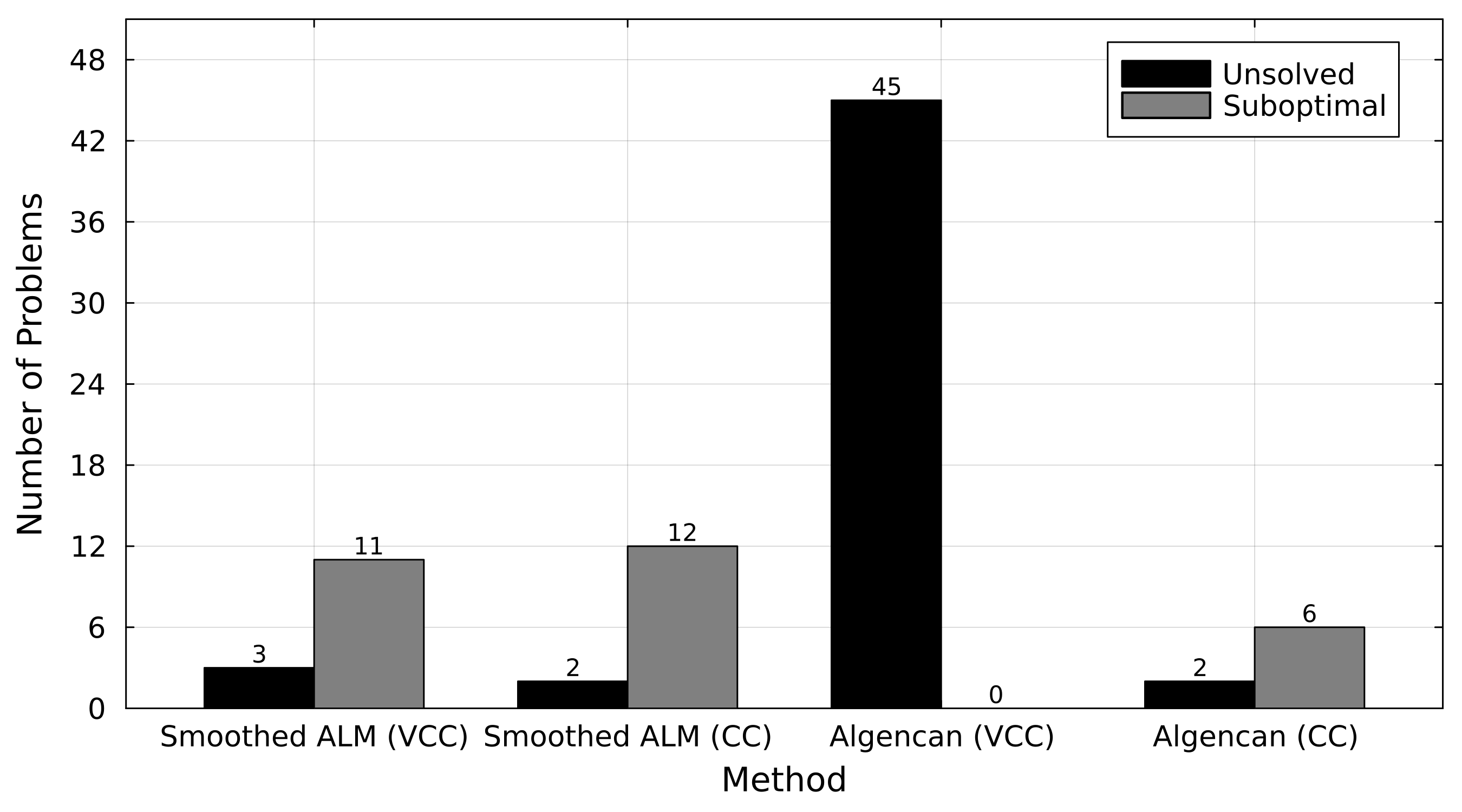}
	\caption{Number of unsolved problems and instances of convergence to suboptimal values on the VLCP benchmark.}\label{fig_unsolved_subopt_vlcp}
	\end{figure}

	The data profile is shown in \cref{fig_data_profile_vlcp}, indicating that the smoothed ALM can solve \MPVCC{} significantly faster than its MPCC reformulation.
	\texttt{Algencan} exhibits poor robustness when applied to~\eqref{VLCP_original}.
	In contrast, it successfully solves more instances when applied to the reformulation~\eqref{VLCP_reformulation}.
	This observation is also supported by the results in \cref{fig_unsolved_subopt_vlcp}.
	Moreover, \texttt{Algencan} (CC) converges to suboptimal solutions on 6 instances,
	whereas the smoothed ALM (VCC) and (CC) converge to suboptimal solutions on 11 and 12 instances, respectively.
	Nevertheless, the smoothed ALM (VCC) remains considerably faster than the competing methods.

	\section{Conclusion}
		We introduce a smoothed ALM for solving \MPVCC{}.
The AL function involving the VCCs is nonsmooth and nonconvex;
it is therefore smoothed using the Lasry--Lions double envelope,
for which we derive a closed-form expression,
to obtain a smooth unconstrained subproblem.
We show that every feasible cluster point of the proposed smoothed ALM (see \cref{algorithm_ALM}) is AM- or AC-stationary and,
consequently, M- or C-stationary under \CCPM{} or \CCPC{},
which are implied by \RCPLD{} and \RCRCQ, respectively.
Under \SSC{}, \LICQ{} and \SOSC{}, we establish the Q-linear convergence for a sufficiently large penalty parameter.
We further derive worst-case iteration complexity bounds for obtaining an approximate stationary point of either the original problem or an associated feasibility problem.
Numerical experiments on illustrative examples, MPCC benchmarks, and vertical linear complementarity problems indicate that the proposed method is competitive in terms of computational effort.
As future work, the following two directions will be pursued:
(i) investigating verifiable conditions that render \CCPM{} and \CCPC{} practically applicable; and
(ii) extending the smoothed ALM framework to nonsmooth functions,
thereby covering applications such as \(\ell_1\)-regularization and the hinge-loss minimization.

	\phantomsection
	\addcontentsline{toc}{section}{References}
	\bibliographystyle{plain}
	\bibliography{TeX/CCs.bib}

\end{document}